\documentclass[10pt, a4paper]{article}
\usepackage{graphicx}
\usepackage{color}
\usepackage{amsmath}
\usepackage{amssymb}
\usepackage{amscd}
\usepackage{bbm}
\usepackage{tikz}
\usepackage{framed}
\usetikzlibrary{patterns}
\usepackage{hyperref}
\hypersetup{
    bookmarks=true,         
    colorlinks=true,       
    linkcolor=red,          
    citecolor=green,        
    filecolor=magenta,      
    urlcolor=cyan           
}
\usepackage[most]{tcolorbox}
\usepackage[utf8]{inputenc}
\usepackage{amsthm}
\usepackage{bbm} 
\usepackage{glossaries}
\usepackage[linesnumbered,lined,boxed,commentsnumbered]{algorithm2e}

\usepackage{marginnote}

\newtheorem{Theorem}{Theorem}
\newtheorem{Assumption}{Assumption}
\newtheorem{Definition}{Definition}
\newtheorem{Lemma}{Lemma}
\newtheorem{Proposition}{Proposition}
\newtheorem{Corollary}{Corollary}
\newtheorem{Remark}{Remark}
\newtheorem{Example}{Example}

\newcounter{BigConst}                     
\newcommand{\nC}{                   
    \refstepcounter{BigConst}             
    \ensuremath{C_{\theBigConst}}
    }
\newcommand{\oC}[1]{\ensuremath{C_{\ref*{#1}}}}  

\newcounter{SmallConst}                     
\newcommand{\nc}{                   
    \refstepcounter{SmallConst}             
    \ensuremath{c_{\theSmallConst}}
    }
\newcommand{\oc}[1]{\ensuremath{c_{\ref*{#1}}}}  

\newcounter{gamma}                     

\newcounter{kappa}                     

\newcounter{delta}                     

\newcounter{theta}                     

\newcounter{L}                     

\newcounter{eps}                     

\newcommand{\inr}[1]{\bigl< #1 \bigr>}

\newcommand{\norm}[1]{\left\|#1\right\|}%
\newcommand{\vertiii}[1]{{\left\vert\kern-0.25ex\left\vert\kern-0.25ex\left\vert #1 
    \right\vert\kern-0.25ex\right\vert\kern-0.25ex\right\vert}}

\newcommand{\vertiiii}[1]{{\left\vert\kern-0.25ex\left\vert\kern-0.25ex\left\vert\kern-0.25ex\left\vert #1 
    \right\vert\kern-0.25ex\right\vert\kern-0.25ex\right\vert\kern-0.25ex\right\vert}}

\newcommand{\beginproof}{{\bf Proof. {\hspace{0.2cm}}}}
\def \endproof
{{\mbox{}\nolinebreak\hfill\rule{2mm}{2mm}\par\medbreak}}

\DeclareMathOperator*{\argmin}{argmin}
\DeclareMathOperator*{\supp}{supp}

\DeclareMathOperator{\Span}{span}

\DeclareMathOperator*{\rank}{rank}

\DeclareMathOperator*{\diam}{diam}

\def\ds1{\textrm{1\kern-0.25emI}} 

\newcommand{\1}{\ensuremath{\mathbbm{1}}}

\newcommand \cC{{\cal C}}

\newcommand \cF{{\cal F}}

\newcommand \cH{{\cal H}}

\newcommand \cN{{\cal N}}

\newcommand \cR{{\cal R}}

\newcommand \cX{{\cal X}}

\newcommand \bC{{\mathbb C}}

\newcommand \bE{{\mathbb E}}

\newcommand \bN{{\mathbb N}}

\newcommand \bP{{\mathbb P}}

\newcommand \bR{{\mathbb R}}

\newcommand \bX{{\mathbb X}}

\newcommand{\bbeta}{{\boldsymbol{\beta}}}

\newcommand \fL{{\mathfrak L}}

\newcommand \fR{{\mathfrak R}}

\newcommand \fU{{\mathfrak U}}

\newcommand{\vf}{{\boldsymbol{f}}}

\newcommand{\va}{{\boldsymbol{a}}}
\newcommand{\vb}{{\boldsymbol{b}}}

\newcommand{\ve}{{\boldsymbol{e}}}

\newcommand{\vlambda}{{\boldsymbol{\lambda}}}

\newcommand{\vv}{{\boldsymbol{v}}}

\newcommand{\vzero}{{\boldsymbol{0}}}
\newcommand{\vx}{{\boldsymbol{x}}}
\newcommand{\vz}{{\boldsymbol{z}}}

\newcommand{\vu}{{\boldsymbol{u}}}
\newcommand{\vw}{{\boldsymbol{w}}}

\newcommand{\vm}{{\boldsymbol{m}}}

\newcommand{\vy}{{\boldsymbol{y}}}

\newcommand{\vdelta}{{\boldsymbol{\delta}}}

\newcommand{\bxi}{{\boldsymbol{\xi}}}

\DeclareMathOperator*{\Tr}{Tr}

\newcommand{\mysymbol}[3]{%
\newglossaryentry{#1}{%
      name={\ensuremath{#2}},%
      text={\ensuremath{#2}},%
      description={#3},%
      sort={#1}%
    }%
\expandafter\newcommand\expandafter{\csname smb#1\endcsname}{\gls{#1}}%
\expandafter\newcommand\expandafter{\csname #1\endcsname}{\ensuremath{#2}}%
}

\newcommand{\op}{\text{op}}
\usepackage[english]{babel}
\usepackage{enumitem}
\usepackage{pgfplots}
\usepackage[mathscr]{euscript}
\pgfplotsset{compat=1.15}

\newif\ifEnon
\Enontrue
\ifEnon
\newcommand{\Guillaume}[1]{\color{red}{[{\bf G:} #1}]\color{black}\,}
\newcommand{\Zong}[1]{\color{blue}{[{\bf Z:} #1}]\color{black}\,}
\else
\newcommand\Guillaume[1]{}
\newcommand{\Zong}[1]{}
\fi

\begin{document}
\title{Sharp convergence rates for Spectral methods via the  feature space decomposition method}
\author{Guillaume Lecu{\'e}, Zhifan Li, and Zong Shang  \\ email: \href{mailto:lecue@essec.edu}{lecue@essec.edu} email: \href{mailto:zhifanli@bimsa.cn}{zhifanli@bimsa.cn}, email: \href{mailto:zong.shang@ensae.fr}{zong.shang@ensae.fr} \\ESSEC, business school, 3 avenue Bernard Hirsch, 95021 Cergy-Pontoise, France.\\
Beijing Institute of Mathematical Sciences and Applications, Beijing, China.\\
CREST-ENSAE, Institut Polytechnique de Paris,\\ 5, avenue Henry Le Chatelier 91120 Palaiseau, France.}
\maketitle

\begin{abstract}
    In this paper, we apply the Feature Space Decomposition (FSD) method developed in \cite{lecue_geometrical_2024,gavrilopoulos_geometrical_2025,lecueGeneralizationErrorMean2025,adamczakFeatureSpaceDecompositioninpreperation} to obtain, under fairly general conditions, matching upper and lower bounds for the population excess risk of spectral methods in linear regression under the squared loss, for every covariance and every signal. This result enables us, for a given linear regression problem, to define a pre-order on the set of spectral methods according to their convergence rates, thereby characterizing which spectral algorithm is superior for that specific problem. Furthermore, this allows us to generalize the saturation effect proposed in inverse problems and to provide necessary and sufficient conditions for its occurrence. Our method also shows that, under broad conditions, any spectral algorithm  cannot overcome the barrier of the information exponent in problems such as single-index learning. 
\end{abstract}

This paper is the third one in the series on the Feature Space Decomposition following \cite{lecue_geometrical_2024}, \cite{gavrilopoulos_geometrical_2025} and the up-coming ones  \cite{lecueGeneralizationErrorMean2025,adamczakFeatureSpaceDecompositioninpreperation}. The position of this paper within the FSD series is as follows: by studying spectral methods and the saturation effect, it illustrates how the FSD method improves the analysis of the population excess risk for these classical estimators as it did previously for minimum norm interpolant estimators as well as for ridge regression.




\section{Introduction}\label{sec:intro}

We are concerned with a supervised regression problem where we observe a vector of output $\vy\in\bR^N$ and a design matrix $\bX\in\bR^{N\times p}$ such that
\begin{equation*}
    \vy = \bX\bbeta^*+\bxi
\end{equation*}where $\bX = [X_1|\cdots|X_N]^\top\in\bR^{N\times p}$, $\bbeta^*\in\bR^p$ and $\bxi=(\xi_i)_{i=1}^N$. We assume that $X_1, \ldots, X_N$ are $N$ i.i.d. vectors in $\bR^p$ with probability distribution denoted by $\mu$ and $\xi_1, \ldots, \xi_N$ are $N$ i.i.d. centered Gaussian random variable with variance $\sigma_\xi^2$ independent of the $X_i$'s. Let $\Sigma = \bE[X\otimes X]:\vv\in\bR^p\mapsto \bE[\langle\vv,X\rangle X]\in\bR^p$ and  $\Sigma = \sum_{j=1}^p\sigma_j\ve_j\otimes\ve_j$ be the spectral decomposition of $\Sigma$ such that $ \sigma_1\geq\sigma_2\geq \cdots\geq\sigma_p>0$.  Given a linear regression problem characterized by a triple $(\Sigma, \bbeta^*, \sigma_\xi)$, our goal is to obtain sharp convergence rates for the estimation error $\|\Sigma^{1/2}(\hat\bbeta-\bbeta^*)\|_2^2$ of estimators $\hat\bbeta$ in a large class of spectral methods.
 


\paragraph{Spectral Methods.} We now introduce the family of estimators of interest in this paper, namely, the spectral methods. We denote $\hat\Sigma = \frac{1}{N}\bX^\top\bX = \frac{1}{N}\sum_{i=1}^N X_i\otimes X_i$ the empirical version of  $\Sigma$.

\begin{Definition}[Spectral method]\label{def:spec_algo}
    Let $(\varphi_t)_{t\geq1}$ be a family of real-valued functions defined on $\bR^+$ call the \it{filter functions}. For all $t\geq1$, we define the \it{spectral method} associated with $\varphi_t$ by:
\begin{equation}\label{eq:def_hat_bbeta}
    \hat\bbeta:\vy\in\bR^N\mapsto \hat\bbeta(\vy) = \frac{1}{N}\varphi_t(\hat\Sigma)\bX^\top\vy = \frac{1}{N}\bX^\top\varphi_t(\frac{1}{N}\bX\bX^\top)\vy
\end{equation} where $\varphi_t(\hat\Sigma)$ and $\varphi_t(\frac{1}{N}\bX\bX^\top)$ are defined via the spectral calculus. When there is no ambiguity, we abbreviate $\hat\bbeta(\vy)$ as $\hat\bbeta$.
\end{Definition} 

A spectral method is uniquely characterized by its filter function. 
There is also a compagnion function to a given filter function that plays an important role regarding the statistical properties of the associated spectral method: it is called the {\it{residual function}} defined for all $t\geq1$ as $\psi_t:x\in\bR^+\to 1-x\varphi_t(x)$. Spectral methods encapsulte several important estimators and algorithms. We are now listing several of them.

\begin{Example}\label{example:spec_algo}
 {\bf{Gradient flow}} with respect to the square loss and linear parameterization initialized at $\boldsymbol{0}$: that is, the solution of the ODE $\dot{\bbeta_t} = -(\nabla \frac{1}{2N}\|\vy-\bX\cdot\|_2^2)(\bbeta_t)$ for any $t\geq1$, starting from $\bbeta_1=\vzero$. Then $\hat\bbeta = \bbeta_t$ is the spectral method associated with the filter and residual functions
        \begin{equation}\label{eq:def_GF}
            \varphi_t:x\in \bR^+ \mapsto\left\{\begin{array}{cc}
             \frac{1-\exp(-tx)}{x} & \mbox{if } x>0\\
             t & \mbox{if } x=0   
            \end{array}\right. \mbox{ and }
            \psi_t:x\in \bR^+ \mapsto \exp(-tx).
    \end{equation}

{\bf{Ridge regression}} with regularization parameter $t^{-1}$, i.e., $\hat\bbeta = \frac{1}{N}(\frac{1}{N}\bX^\top\bX+t^{-1}I_p)^{-1}\bX^\top\vy$,  is the spectral method for the choice of filter and associated residual functions
    \begin{align}\label{eq:def_ridge}
        \varphi_t(x) = (t^{-1}+x)^{-1}\mbox{ and }\psi_t(x) = \frac{1}{xt+1}.
    \end{align} 

     {\bf{Gradient descent}} starting at $\bbeta_1=\vzero$ with step-size $0<\eta<1/8$ and at step $t\in \bN^*$ for minimizing $\bbeta\mapsto \frac{1}{2N}\|\vy-\bX\bbeta\|_2^2$, i.e. $\bbeta_{t} = \bbeta_{t-1} - \eta\nabla(\frac{1}{2N}\|\vy-\bX\cdot\|_2^2)(\bbeta_{t-1})$, is the spectral method for the filter function $\varphi_t(x) = (1-(1-\eta x)^t)/x$ and its associated residual function $\psi_t(x)=(1-\eta x)^t$. 

     {\bf{The heavy-ball method, \cite[Section 3.2.1]{Solr-KOHA-OAI-TEST:19722} and Nesterov's acceleration, \cite{nesterovMethodSolvingConvex1983}}} with variable parameters are also examples of spectral algorithms (see \cite{paglianaImplicitRegularizationAccelerated2019}). Their residual functions admit recursive definitions with no known closed-form expressions.

    {\bf{Principle Components Regression (PCR)}} estimator is $\hat\bbeta\in \argmin(\|\vy - \bX\bbeta\|_2^2:\, \bbeta\in \hat V_{\leq k})$ where $\hat V_{\leq k}$ is the subspace spanned by the first $k$ eigenvectors of $\hat\Sigma$. PCR equals to the spectral method with tuning parameter $\hat\sigma_{k+1}\leq bt^{-1}<\hat\sigma_k$  - where $\hat\sigma_k$ and $\hat\sigma_{k+1}$ are the $k$-th and $k+1$-th largest eigenvalue of $\hat\Sigma$ - for the filter function and its associated residual function given for some constant $b>0$ by
    \begin{align*}
        \varphi_t:x\in\bR^+\mapsto \frac{1}{x}\1( bt^{-1}\leq x)\mbox{ and } \psi_t(x) = \1(bt^{-1}>x).
    \end{align*}
\end{Example}

We are now describing the class of spectral methods  considered in this work.

\begin{Assumption}\label{ass:filter_fct} The family of filter functions $(\varphi_t)_{t\geq1}$ is such that for all $t\geq1$, $\varphi_t$ has an holomorphic extension to an open subset of $\bC$ containing the contour $\cC_t$ defined in Section~\ref{sec:choice_contour}. Furthermore,  there are two absolute constants $0\leq \nc\label{c_filter_1} \leq \nC\label{C_filter_1}$ such that for all $t\geq1$ and all $x\in[0,8]$:
\begin{equation}\label{eq:def_filter_fct}
    \frac{\oc{c_filter_1}}{x+t^{-1}}\leq \varphi_t(x) \leq \frac{\oC{C_filter_1}}{x+t^{-1}}.
\end{equation}  
\end{Assumption}

Filter functions of gradient flow, ridge regression and gradient descent all satisfy Assumption~\ref{ass:filter_fct}. Indeed, for gradient flow, \eqref{eq:def_filter_fct} holds for all $x\geq0$ if one take $c_1=1$ and $C_1=2$  and the same does for ridge regression with $c_1=C_1=1$. For gradient descent, \eqref{eq:def_filter_fct} holds only for $x\in[0,8]$ and for $c_1=\eta/2$ and $C_1=2$. In Assumption~\ref{ass:filter_fct}, we only ask \eqref{eq:def_filter_fct} to be true for $x\in[0,8]$ because later we will apply this inequality only on an event where both spectra of $\Sigma$ and $\hat\Sigma$ are in $[0,8]$.

We assume the existence of an holomorphic extension for technical reason related to the residual theorem, it however discards the PCR estimator for which we develop a special analysis. Regarding the assumption on the shape of the residual functions in \eqref{eq:def_filter_fct}: we ask for the residual function to be equivalent to the one of the ridge estimator with regularization parameter $t^{-1}$.  However, the family of spectral methods satisfying this assumption is pretty wide. We also note that \eqref{eq:def_filter_fct} is weaker than the classical assumptions used in the field of spectral methods that we recall below in Remark~\ref{assumption:classical}.


\begin{Remark}[Classical assumptions]\label{assumption:classical}
In several works \cite{blanchard_lepskii_2019}, the filter function is assumed to satisfy the following: there exist absolute constants  $\tau\in\bN_+\cup\{\infty\}$, $\nC\label{C_filter_2} = \oC{C_filter_2}(\tau)\geq 1$ such that
    \begin{enumerate}
        \item for any $0\leq\alpha\leq 1$ and any $t\geq1$, $\sup(x^\alpha\varphi_t(x):\, 0\leq x\leq 1)\leq \oC{C_filter_1}t^{1-\alpha}$;
        \item for any $t\geq1$, $\sup(|\psi_t(x)|(x+t^{-1})^\tau:\, 0\leq x\leq 1)\leq \oC{C_filter_2}t^{-\tau}$;
        \item for any $0\leq x\leq 1$ and $1<t<\infty$, $\oc{c_filter_1}\leq (x+t^{-1})\varphi_t(x)$.
    \end{enumerate}It is straightforward to see that \textit{item 1.} for $\alpha=0$ and $\alpha=1$ together with \textit{item 3.} implies \eqref{eq:def_filter_fct}.
\end{Remark}
 

The study of spectral methods, as far as we know, originated with Tikhonov regularization \cite{Engl_Hanke_Neubauer_1996} (ridge regression) and Landweber regularization (gradient descent) for (ill-posed) statistical inverse problems. The classical analysis of the statistical properties of spectral methods is generally based on regression problems in Sobolev spaces i.e. under regularity assumptions. Specifically, one assumes that $\Sigma$ exhibits power decay, i.e., there exists $\alpha > 1$ such that $\sigma_j \sim j^{-\alpha}$ for all $j$, and that there exists $s \geq 1$ such that $\|\Sigma^{\frac{1-s}{2}} \bbeta^*\|_2$ is bounded, known as the Hölder source condition. Under this framework, the properties of spectral methods are well understood; to name a few, \cite{smale_learning_2007,yao_early_2007,bauer_regularization_2007,lo_gerfo_spectral_2008,blanchard_kernel_2016,pillaud-vivien_statistical_2018,paglianaImplicitRegularizationAccelerated2019,blanchard_lepskii_2019,zhang_optimality_2023,li_generalization_2024}.  

However, beyond this setting, the statistical properties of spectral methods are not yet fully understood—even though such algorithms have existed for almost three decades \cite{Engl_Hanke_Neubauer_1996,engl_regularization_2000}. We emphasize that in modern mathematical statistics, particularly in problems motivated by machine learning, a linear regression setup often does not satisfy the above Hölder source condition. In fact, in such problems, $\Sigma$ and $\bbeta^*$ may follow arbitrary patterns (we will present an example in Section~\ref{sec:no_feature}). Thus, it is genuinely necessary to understand the statistical properties of spectral methods for arbitrary linear regression problems. 

\begin{framed}
    \centering
    Our first objective is, for a given linear regression problem $(\Sigma,\bbeta^*,\sigma_\xi)$, to obtain matching high-probability upper and lower bounds for $\|\Sigma^{1/2}(\hat\bbeta - \bbeta^*)\|_2^2$ where $\hat\bbeta$ is a spectral method whose filter function satisfy Assumption~\ref{ass:filter_fct}. Our second objective is to show how the the Feature Space Decomposition method can be used on spectral methods to achieve this goal.
\end{framed}


\subsection{Structure of this paper and Notation}
In Section~\ref{sec:FSD}, we introduce the Feature Space Decomposition. In Section~\ref{sec:main_results}, we present our main results on spectral methods. In Section~\ref{sec:partial_order}, we introduce a pre-order on the set of spectral methods based on their convergence rates. In this section, we also provide the definition of the generalized saturation effect. In Section~\ref{sec:conclusions}, we summarize the main contributions of this paper and propose several directions for future research. The proofs of all results are in Section~7 and beyond.

We use $a\lesssim b$ (respectively $a\gtrsim b$) to represent the fact that there exists an absolute constant $C$ such that $a\leq Cb$ ($a>Cb$). We use $a\sim b$ if $a\lesssim b$ and $b\lesssim a$. We say $a\lesssim_K b$ if $C=C(K)$. For a probability measure $\mu$, we write $\mu^{\otimes N}$ as its $N$-fold tensor product. We denote the $\ell_2\to\ell_2$ operator norm of a matrix by $\|\cdot\|_{\op}$ and by $\|\cdot\|_{HS}$ its Hilbert-Schmidt norm.

\section{The Feature Space Decomposition method}\label{sec:FSD}
In this section, we present the Feature Space Decomposition (FSD): a method for analyzing the population excess risk of an estimator. To that end, we consider a general scalar supervised regression problem where we aim at predicting an output $Y$ based on some input vector $X$ given $N$ examples $(X_i,Y_i)_{i=1}^N$ of this input/output relationship. Let \(\cF \subset L^2(\mu)\) be a sub-space called the model or the feature space such that $Y=f^*(X)+\xi$ for some centered noise $\xi$ that is independent of $X$ and for some unknown function $f^*\in \cF$. We are looking for a predictor $\hat f \in\cF$ with a small excess squared risk $\bE (Y-\hat f(X))^2 - \bE (Y-f^*(X))^2 = \norm{\hat f - f^*}_{L^2(\mu)}^2$.

At the heart of the FSD method is an orthogonal decomposition $V_J\oplus V_{J^c} = \cF$ in $L^2(\mu)$ of the feature space. In the FSD approach, we refer to $V_J$ as the 'estimation part' of $\cF$ and to $V_{J^c}$ as the 'noise absorption part' of $\cF$. We may justifying this terminology as follows. Given an estimator $\hat f\in\cF$, we decompose it as the sum of its two projections: $\hat f = \hat f_J + \hat f_{J^c}$ - where $P_J$ and $P_{J^c}$ are orthogonal projection operators onto $V_J$ and $V_{J^c}$ and for $f\in\cF$, $f_J = P_Jf$ and $f_{J^c} = P_{J^c}f$. Then the excess risk decomposition used in the FSD method is 
\begin{align}\label{eq:risk_decomposition_squared}
     \norm{\hat f - f^*}_{L^2(\mu)}^2 \begin{cases}
                =\norm{\hat f_J - f_J^*}_{L^2(\mu)}^2 + \norm{\hat f_{J^c} - f^*_{J^c}}_{L^2(\mu)}^2, & \mbox{ if }V_J\perp V_{J^c}\mbox{ in }L^2(\mu),\\
                \leq 2\norm{\hat f_J - f_J^*}_{L^2(\mu)}^2 + 2\norm{\hat f_{J^c} - f^*_{J^c}}_{L^2(\mu)}^2, & \mbox{ otherwise.}
            \end{cases}
\end{align}Regardless of whether $V_J$ is orthogonal to $V_{J^c}$, for the upper bound of $\|\hat f_{J^c}-f_{J^c}^*\|_{L^2(\mu)}^2$, we apply the triangle inequality to obtain $\|\hat f_{J^c}-f_{J^c}^*\|_{L^2(\mu)}^2 \leq 2\|\hat f_{J^c}\|_{L^2(\mu)}^2 + 2\|f_{J^c}^*\|_{L^2(\mu)}^2$ for the statistical analysis of $\hat f$.
In particular, the triangle inequality used in the second term says that we do not expect $\hat f_{J^c}$ to estimate $f^*_{J^c}$; whereas we expect $\hat f_J$ to estimate $f^*_J$, hence, the name 'estimation part' for $V_J$. In our previous applications of the FSD method to minimum norm interpolant estimators \cite{lecue_geometrical_2024} and ridge estimators \cite{gavrilopoulos_geometrical_2025}, it appears that $\hat f_{J^c}$ was used to either interpolate the noise or estimate it as long as we may look at $f^*_{J^c}(X)$ as been part of the noise. This explain the name 'noise absorption part' for $V_{J^c}$.  Of course this picture coming from the excess risk decomposition \eqref{eq:risk_decomposition_squared} will work only if we choose correctly the feature space decomposition $\cF = V_J\oplus V_{J^c}$.



In the context of the linear model $Y=\inr{\bbeta^*, X} + \xi$ and for spectral methods and minimum $\ell_2$-norm interpolant estimator, it appears that the optimal choice for $V_J$ and $V_{J^c}$ are two orthogonal eigenspaces of $\Sigma$:
\begin{equation}\label{eq:VJ}
   V_J = \Span(\ve_j: j\leq k^*) \mbox{ and  } V_{J^c} = \Span(\ve_j: j\geq k^*+1)
\end{equation}for some optimal choice of $k^*$ that we will call later the \textit{estimation dimension}. In particular, the FSD defined in \eqref{eq:VJ} satisfies $V_J \perp V_{J^c}$ in $L^2(\mu)$.
Furthermore, the estimator will have good estimation property when the signal is well aligned with $\Sigma$ because in that case \eqref{eq:risk_decomposition_squared} reduces to the following inequality: 
\begin{equation}\label{eq:excess_risk_decomp_v0}
  \|\Sigma^{1/2}(\hat \bbeta - \bbeta^*)\|_2 \leq \|\Sigma_J^{1/2}(\hat \bbeta_J - \bbeta^*_J)\|_2 + \|\Sigma_{J^c}^{1/2}\hat \bbeta_{J^c}\|_2 + \|\Sigma_{J^c}^{1/2}\bbeta_{J^c}^*\|_2
\end{equation} where $\Sigma_J = \bE[X_J\otimes X_J]$ and $\Sigma_{J^c} = \bE[X_{J^c}\otimes X_{J^c}]$. Hence, $\|\Sigma_{J^c}^{1/2}\bbeta_{J^c}^*\|_2^2$ is part of the estimator error of the estimator. To make this term small we need $\bbeta^*$ to have most of its 'energy supported' on the first $k^*$ eigenvectors of $\Sigma$; that is what we call \textit{signal alignment}.

\paragraph{FSD as an Analytical Method.} FSD is a mathematical method for analyzing the excess risk of estimators. That is to say, statisticians have no direct control over the choice of $V_J$ and $V_{J^c}$—because the estimator itself does not take $V_J$ or $V_{J^c}$ as parameters. Therefore, we assert that the decomposition of $\cF$ into two subspaces is performed autonomously by the estimator, not by the statistician. Consequently, when statisticians execute this statistical algorithm, this selection occurs as a black-box operation. For estimators with tunable parameters, given a parameter set by the statistician, the estimator automatically determines the decomposition based on both this parameter and the regression problem itself. Different estimators employ distinct decomposition strategies. For instance, the ridge regression studied in \cite{gavrilopoulos_geometrical_2025} decomposes the feature space $\cF$ solely based on the spectrum of $\Sigma$, whereas the basis pursuit studied in \cite{adamczakFeatureSpaceDecompositioninpreperation} decomposes the feature space $\cF$ depending on the alignment between $\bbeta^*$ and the eigenvectors of $\Sigma$—hence possessing the sparsity recovery property.


In summary, this decomposition of $\cF$ affects the bound we obtain for $\|\hat f-f^*\|_{L^2(\mu)}$. Thus, we should understand FSD as follows: For each $(V_J,V_{J^c})$, we obtain a bound $r(V_J,V_{J^c})$ such that with high probability $\|\hat f-f^*\|_{L^2(\mu)} \leq r(V_J, V_{J^c})$. However, since the decomposition is not unique, there must exist an optimal decomposition $(V_{J_*},V_{J_*^c})$ among all feasible decompositions, yielding $\|\hat f-f^*\|_{L^2(\mu)}\leq r(V_{J_*},V_{J_*^c}) = \min\{r(V_J,V_{J^c}) : (V_J,V_{J^c})\}$, where this optimal decomposition is selected autonomously by the estimator $\hat f$. Our result holds for all feasible decompositions and consequently applies to this optimal decomposition $(V_{J_*}, V_{J_*^c})$ as well. The final step being to show that the rate $r(V_{J_*},V_{J_*^c})$ is optimal up to absolute constant by proving a matching lower bound. That is, one needs to show that for the decomposition $(V_{J_*},V_{J_*^c})$, there exists an absolute constant $c>0$ such that, with high probability or in expectation, $\|\hat f - f^*\|_{L^2(\mu)} \geq c\, r(V_{J_*},V_{J_*^c})$. This would demonstrate that the decomposition indeed captures the essence of the population excess risk of $\hat f$. Note that this lower bound differs from a minimax lower bound. We emphasize that the present lower bound concerns a given regression problem $(f^*,\mu,\sigma_\xi)$ and a fixed estimator $\hat f$, providing a lower bound on its population excess risk, whereas a minimax lower bound is a worst case analysis, over a family of regression problems $\{(f^*,\mu,\sigma_\xi)\}$ and a class of estimators $\{\hat f\}$, the minimal possible population excess risk attainable among them. Our bound depends on all three parameters $(f^*,\mu,\sigma_\xi)$ of a regression problem showing how the optimal rate depend on the interaction between the signal $f^*$ and $\Sigma$.

\paragraph{FSD and Alignment Property.}
In this paragraph, we consider the optimal feature space decomposition \((V_{J_*}, V_{J_*^c})\) of the feature space \(\cF\) induced by \(\hat f\). From the previous paragraph, we know that this decomposition characterizes the estimation ability of the estimator \(\hat f\)---that is, the estimation of \(f_{J_*}^*\) by \(\hat f_{J_*}\).  We define alignment property as follows.

\begin{Definition}[Alignment Property]\label{def:alignment_property}
    Let $(\mu,f^*,\xi)$ be a supervised regression problem, let $\cH$ be an RKHS, and let $\hat f_N$ be an estimator taking values in $\cH$. Let $T_K:L^2(\mu)\to L^2(\mu)$ be its integral operator and $\{f_j\}_{j=1}^\infty$ be its eigenfunctions. Let $(V_{J_*},V_{J_*^c})$ be any optimal FSD. Let $f_{\cH}$ be any element in $L^2(\mu)$. Given a sequence of non-negative real numbers $\{\gamma_j\}_{j\in J_*^c}$, a real number $0<\delta<1$ and a strictly increasing function $\Phi:\bR_+\to\bR_+$, we say that $\hat f_N$ satisfies the $(\Phi,f_\cH,J_*,\delta)$ alignment property with respect to $\{\gamma_j\}_{j\in J_*^c}$ if with probability at least $1-\delta$, there holds $\|\hat f_N-f_{\cH}\|_{L^2(\mu)}^2\leq\Phi(\sum_{j\in J_*^c}\gamma_j\langle f_\cH,f_j\rangle^2)$, where $\langle f_\cH,f_j\rangle$ is the inner product in $L^2(\mu)$ between $f_\cH$ and $f_j$. For convenience, we say that $\hat f_N$ satisfies the $(f_\cH,V_{J_*})$-alignment property.
\end{Definition}

For example, it is proved in \cite{gavrilopoulos_geometrical_2025} that there exist some absolute constants \(0 < b< 1\), such that for the ridge regression \eqref{eq:def_ridge} with tuning parameter $t^{-1}$, we have \(V_{J_*} = \mathrm{Span}(\ve_j : j \le k^{**})\), where
\begin{align}\label{eq:def_estimation_dimension_ridge}
k^{**}= \min\bigl\{ k\in[p]: \sigma_{k+1} N \le b\, (\Tr(\Sigma_{k+1:p})+Nt^{-1}) \bigr\}.
\end{align} 
In this paper, we prove that, under mild assumptions, spectral algorithm are all sharing the same optimal estimation dimension $k^*$ (which is equivalent to $k^{**}$ under our assumptions) and have the alignment property.

\paragraph{FSD method applied for spectral algorithm.} In contrast to the self-regularization techniques employed in \cite{lecue_geometrical_2024,gavrilopoulos_geometrical_2025,adamczakFeatureSpaceDecompositioninpreperation} for the study of ridge regression and the minimum norm interpolant estimator, our setting allows for a closed-form solution of the projections of the spectral method (see \eqref{eq:def_hat_bbeta}). Consequently, rather than expressing $\hat\bbeta$ as the solution to a convex optimization problem as in \cite{lecue_geometrical_2024,gavrilopoulos_geometrical_2025,adamczakFeatureSpaceDecompositioninpreperation}, we directly decompose its bias and variance components over the two mutually orthogonal subspaces $V_J$ and $V_{J^c}$.
Therefore, the FSD method adopted in our analysis constitutes an ``estimation-noise absorption'' decomposition that goes beyond the standard bias-variance framework. Specifically, this approach primarily gives rise to five terms: the bias and variance of both $\hat\bbeta_J$ and  $\hat\bbeta_{J^c}$ as well as the alignment term $\norm{\Sigma^{1/2}_{J^c}\bbeta^*_{J^c}}_2$ that follows from the excess risk decomposition \eqref{eq:excess_risk_decomp_v0}. 


\section{Main Results}\label{sec:main_results}
In this section, we present the main results of this paper. Our main results are divided into two settings: the case where $X$ is a sub-Gaussian random vector, and the case in the RKHS regression problems.

\paragraph{Sub-Gaussian design.}
We first gather all the model assumptions in sub-Gaussian design case.
\begin{Assumption}\label{assumption:main_upper}
    We assume that $\|\Sigma\|_{\op}\leq 1$.
    The noise $\xi$ satisfies $\xi\sim\cN(0,\sigma_\xi^2)$ and it is independent with $X$. Assume $X$ is sub-Gaussian: there exists an absolute constant $C>0$ such that for any $\vv\in\bR^p$ and $q\geq 2$, $\|\langle X,\vv\rangle\|_{L^q(\mu)}\leq C\sqrt{q}\|\langle X,\vv\rangle\|_{L^2(\mu)}$.
\end{Assumption}

Next, we introduce the optimal dimension used to split the feature space in the case of spectral methods.

\begin{Definition}\label{def:estimation_dimension}
 Let $b>0$ and $t\geq1$. The \underline{estimation dimension} of the spectral method $\hat \bbeta$ with filter function $\varphi_t$ is defined as
\begin{align}\label{eq:def_estimation_dimension}
    k^*=k_{t^{-1},b}^* = \min \Bigl\{ k\in[p]: \sigma_{k+1} \le bt^{-1} \Bigr\}.
\end{align}
\end{Definition}

The estimation dimension $k^*$ is the dimension of the space $V_{J_*}$ where estimation of the spectral method $\hat \bbeta$ with filter function $\varphi_t$ happens. It coincides with the optimal one for ridge regression recalled in \eqref{eq:def_estimation_dimension_ridge} when $\Tr[\Sigma_{J_{**}^c}]\leq N t^{-1}$ where $\Tr[\Sigma_{J_{**}^c}] = \sum_{j>k^{**}}\sigma_j$.  In particular, we see that this dimension does not depend on the shape of the filter function but just on its parameter $t$. However, the optimal convergence rate of a spectral method depends on its filter function via its residual function since we will show that it is given by
\begin{align}\label{eq:def_r_V_J_star}
    r(V_{J_*},V_{J_*^c}) = \norm{\Sigma_{J_*}^{1/2}\psi_t(\Sigma)\bbeta_{J_*}^*}_2 +\sigma_\xi\sqrt{\frac{|J_*|}{N}} + \norm{\Sigma_{J_*^c}^{1/2}\bbeta_{J_*^c}^*}_2 +  \sigma_\xi t\sqrt{\frac{\Tr(\Sigma_{J_*^c}^2)}{N}}, 
\end{align}where $V_{J_*}=\Span(\ve_j:\, j\in J_*)$, $J_* = [k^*]$, $(\ve_j)_j$ are the eigenvectors of $\Sigma$ and  $\psi_t$ is the residual function defined in Definition~\ref{def:spec_algo}.

We are now in a position to state our main results in sub-Gaussian case: two upper and lower bounds for the excess risk of spectral methods and a corollary identifying the conditions where the two bounds match, giving the optimal rate from \eqref{eq:def_r_V_J_star}.  The proof of the following results may be found in Section~\ref{sec:proof_main} for the upper bound and in Section~\ref{sec:optimal_FSD} for the lower bound.

\begin{Theorem}[Main Result - upper bound]\label{theo:main} We consider a linear regression model with parameter $(\bbeta^*, \Sigma, \sigma_\xi)$ satisfying Assumption~\ref{assumption:main_upper}. Let $(\varphi_t)_{t\geq1}$ be a family of filter functions satisfying Assumption~\ref{ass:filter_fct} for $c_1=0$. Let $t\geq1$.
Then, there exists an absolute constant $c>0$ such that for all $0<\square<1/9$, if $\square^2 N\gtrsim \Tr\left(\Sigma(\Sigma + t^{-1}I_p)^{-1}\right)\vee 1$ and $\square\lesssim \log^{-1}(et)$ then with probability at least $1-2\exp(-c|J_*|)-\exp(-c\square^2 N)$,
\begin{align*}
    \norm{\Sigma^{1/2}(\hat\bbeta-\bbeta^*)}_2 \lesssim r(V_{J_*}, V_{J_*^c}) +  \frac{\square}{t}\norm{\Sigma_{J_*}^{-\frac{1}{2}}\bbeta^*_{J_*}}_2.
\end{align*}
\end{Theorem}



\begin{Theorem}[Main result - lower bound]\label{theo:main_LB} There are absolute positive  constants $c_0, c, c_2$ and $c_3$ such that the following holds. Let $(\bbeta^*, \Sigma, \sigma_\xi)$ be the parameters of a linear regression model under Assumption~\ref{assumption:main_upper} where $X$ is assumed to have independent and centered coordinates with respect to $\{\ve_1,\cdots,\ve_p\}$. Let $\hat\bbeta$ be a spectral method with filter function satisfying Assumption~\ref{ass:filter_fct} for $0<c_1\leq C_1$. Let $0<\square<1/9$ be such that $\square \lesssim \log^{-1}(et)$ and $\square^2 N\gtrsim \Tr\left(\Sigma(\Sigma + t^{-1}I_p)^{-1}\right)\vee 1$.  Let $k^*$ be the estimation dimension introduced in Definition~\ref{def:estimation_dimension} for some $0<b\leq c_0$ and $J_*=[k^*]$.  Then, with probability at least $1 - c\exp(-k^*/c)-\exp(-\square^2 N/c)$, 
\begin{equation}\label{eq:lower_main_result}
    \norm{\Sigma^{1/2}(\hat \bbeta-\bbeta^*)}_2 \geq c_2 r(V_{J_*}, V_{J_*^c}) - \frac{c_3 \square}{t} \norm{\Sigma_{J_*}^{-\frac{1}{2}}\bbeta^*_{J_*}}_2.
\end{equation}



\end{Theorem}

The next result is a high probability upper and lower bound for spectral methods showing that $r(V_{J_*}, V_{J_*^c})$ is the right quantity describing the statistical properties of these estimators for a given linear regression model. It follows from Theorem~\ref{theo:main} and Theorem~\ref{theo:main_LB}.

\begin{Corollary}\label{coro:main_coro}
There are absolute positive  constants $c_0,c,(c_k)_{k=2,3,4,5}$ such that the following holds.
Under the same assumptions as in Theorem~\ref{theo:main_LB}. Let $t\geq1$ and $0<\square<1/9$ be such that $\square \leq c_0 \log^{-1}(et)$, $\square^2 N\geq c (\Tr\left(\Sigma(\Sigma + t^{-1}I_p)^{-1}\right)\vee 1)$, $k^*\geq c$ and 
    \begin{equation}\label{eq:condition_for_matching_UB_LB}
\frac{\square}{t} \norm{\Sigma_{J_*}^{-\frac{1}{2}}\bbeta^*_{J_*}}_2  \leq c_2 r(V_{J_*},V_{J_*^c}).
    \end{equation}Then, with probability at least $1 - c_3\exp(-k^*/c_3)-\exp(-\square^2 N/c_3)$, 
\begin{equation*}
    c_4 r(V_{J_*}, V_{J_*^c}) \leq \norm{\Sigma^{1/2}(\hat \bbeta-\bbeta^*)}_2 \leq c_5 r(V_{J_*}, V_{J_*^c}).
\end{equation*} 
\end{Corollary}

Condition~\eqref{eq:condition_for_matching_UB_LB} holds when $\big(\square/t\big)\norm{\Sigma_{J_*}^{-\frac{1}{2}}\bbeta^*_{J_*}}_2 $ is smaller than one of the four terms in $r(V_{J_*}, V_{J_*^c})$; for instance, it holds when 
\begin{enumerate}
    \item $\frac{1}{t\sigma_\xi}\norm{\Sigma_{J_*}^{-\frac{1}{2}}\bbeta^*_{J_*}}_2\lesssim \frac{1}{\square}\sqrt{\frac{|J_*|}{N}},$ where we recall that $t^{-1}\|\Sigma_{J_*}^{-\frac{1}{2}}\bbeta^*_{J_*}\|_2$ is the bias of $\hat\bbeta_J^{(\mathrm{Ridge})}$ when $\hat\bbeta_J^{(\mathrm{Ridge})}$ is the ridge regression with tuning parameter $t$, and $\frac{1}{\square}$ may be taken to be $\sqrt{N/(\Tr(\Sigma(\Sigma + t^{-1}I_p)^{-1})\wedge 1)}$;
    \item or when $
\frac{\square}{t} \norm{\Sigma_{J_*}^{-\frac{1}{2}}\bbeta^*_{J_*}}_2  \lesssim \norm{\Sigma_{J_*}^{1/2}\psi_t(\Sigma)\bbeta_{J_*}^*}_2,$ which is the case when $\square/t$ is small enough so that $\psi_t(x)\geq \square/(tx)$ for all $x\in[0,1]$ (recall that we assumed that $\norm{\Sigma}_{\op}\leq1$ in Assumption~\ref{assumption:main_upper}) which is equivalent to assume that $\varphi_t(x)\leq (t-\square x)/(xt)$.
\end{enumerate}

As mentioned earlier the case of PCR is special since it requires a property on the $k^*$-th spectral gap of $\Sigma$. We therefore state a result devoted to PCR. The proof of the following result is different from the one of Theorem~\ref{theo:main} and may be found in Section~\ref{sec:statistical_analysis_of_pcr}.

\begin{Theorem}[Upper bound for PCR]\label{theo:main_PCR} We consider a linear regression model with parameter $(\bbeta^*, \Sigma, \sigma_\xi)$ satisfying Assumption~\ref{assumption:main_upper}. Let $t\geq1$ and $0<b<1$.  Denote by $\hat \bbeta$ the PCR estimator with filter function $\varphi_t:x>0\mapsto x^{-1}\1(x\geq bt^{-1})$. Let $0<\square<1/9$ and assume that $\square^2 N\gtrsim \Tr\left(\Sigma(\Sigma + t^{-1}I_p)^{-1}\right)\vee 1$ and that $\theta>0$ where
\begin{equation}\label{eq:def_theta_PCR}
    \theta := \min\left(bt^{-1} - \big(\sigma_{k^*+1} + \square(\sigma_{k^*+1} + t^{-1})\big), \big(\sigma_{k^*} - \square(\sigma_{k^*} + t^{-1})\big) - bt^{-1}\right) .
\end{equation}Then, there exists an absolute constant $c>0$ such that with probability at least $1-2\exp(-c|J_*|)-\exp(-c\square^2 N)$,
\begin{align*}
    \norm{\Sigma^{1/2}(\hat\bbeta-\bbeta^*)}_2 \lesssim r(V_{J_*}, V_{J_*^c}) +  \frac{\square}{\theta^2}\norm{\Sigma_{J_*}^{-\frac{1}{2}}\bbeta^*_{J_*}}_2.
\end{align*}
\end{Theorem}In the case of PCR, the convergence rate $r(V_{J_*}, V_{J_*^c})$ contains only the three terms
\begin{equation*}
    \norm{\Sigma_{J_*^c}^{1/2}\bbeta_{J_*^c}^*}_2 +  \sigma_\xi \sqrt{\frac{|J_*|}{N}} + \sigma_\xi t \sqrt{\frac{\Tr(\Sigma_{J_*^c}^2)}{N}}
\end{equation*}since $\norm{\Sigma_{J_*}^{1/2}\psi_t(\Sigma)\bbeta_{J_*}^*}_2 = 0$ because $\psi_t(\Sigma)=P_{J_*^c}$. Note also that compare with Theorem~\ref{theo:main} we don't need to choose $\square$ less than $\log^{-1}(et)$ and so one can choose $\square$ to be of the order of a constant. The choice $\square\sim \sqrt{k^*/N}$ is also legitimate as long as the sample complexity assumption $\square^2 N\gtrsim \Tr\left(\Sigma(\Sigma + t^{-1}I_p)^{-1}\right)\vee 1$ is satisfied that is when $k^*\gtrsim \Tr\left(\Sigma(\Sigma + t^{-1}I_p)^{-1}\right)\vee 1$ which holds (see the discussion below \eqref{eq:sample_complexity_ass_imply}) when $k^*\gtrsim t \Tr[\Sigma_{J_*^c}]$. This is for instance, the case when $\sigma(\Sigma)$ has a fast decay. However, Theorem~\ref{theo:main_PCR} requires $\theta>0$ that holds iff the $k^*$-th spectral gap of $\Sigma$ is large enough:
\begin{equation*}
    \sigma_{k^*} - \sigma_{k^+1}> \square\left(\sigma_{k^*} + \sigma_{k^+1} + 2t^{-1}\right)
\end{equation*}and when $bt^{-1}\in\left[\sigma_{k^*+1} + \square(\sigma_{k^*+1} + t^{-1})\big), \sigma_{k^*} - \square(\sigma_{k^*} + t^{-1})\right]$.

\paragraph{Kernel case.} In this section, we study spectral methods in Reproducing Kernel Hilbert Spaces (RKHS). Let $\cX$ be a compact set, let $K:\cX\times\cX\to\bR$ be a bounded kernel function. Later in Assumption~\ref{assumption:bounded}, we assume, without loss of generality, that $\|K\|_{L^\infty(\mu\otimes\mu)}\leq 1$. Let $\cH$ be the RKHS generated by $K$. We call the map $\phi:\vx\in\cX\mapsto K(\vx,\cdot)\in\cH$ the canonical feature map of $\cH$. Although, in $\cH$, functions are in general non linear, the space $\cH$ itself is a linear space, which induces a linear parametrization on this space. Therefore, we still refer to this setting as linear regression. Let $X$ be a $\cX$-valued random variable, and let $X_1,\cdots,X_N$ be its independent copies. In that case, we define the sampling operator $\bX:f\in\cH\mapsto (f(X_i))_{i=1}^N\in\bR^N$ and its adjoint $\bX^\top:\vlambda\in\bR^N\mapsto \sum_{i=1}^N \lambda_i \phi(X_i)\in\cH$ since for all $f\in\cH$ and $\lambda\in\bR^N$, we have $\inr{\bX f, \lambda} = \inr{\bX^\top \lambda, f}_{\cH}$. Moreover, $\hat\Sigma=\frac{1}{N}\sum_{i=1}^N \phi(X_i)\otimes_\cH\phi(X_i)$ and $\Sigma=\bE[\phi(X)\otimes_\cH\phi(X)]:f\in\cH\mapsto \bE[f(X)\phi(X)]\in\cH$ is the covariance operator of this RKHS - where $(f\otimes_\cH g)(h) = \inr{g,h}_\cH f$ for all $f,g,h\in\cH$. We now use the following notation for spectral methods:
\begin{align*}
    \hat f_N:\vy\in\bR^N\mapsto \hat f_N(\vy) = \frac{1}{N}\varphi_t(\hat\Sigma)\bX^\top\vy\in\cH.
\end{align*}When there is no ambiguity, we abbreviate $\hat f_N(\vy)$ as $\hat f_N$. In kernel regression, we assume that there exists $f^\circ\in L^2(\mu)$ such that $Y = f^\circ(X)+\xi$, where $\xi$ is a centered random variable independent of $X$ with variance $\sigma_\xi^2$. Here, $f^\circ$ does not necessarily belong to $\cH$ so that we allow $\cH$ to be a misspecified. Define $\overline{\cH}$ as the closure of $\cH$ in $L^2(\mu)$. We let $f^* = \mathrm{Proj}_{\overline{\cH}}f^\circ$, where the projection is taken with respect to the $L^2(\mu)$ inner product.

For RKHSs, in this section, we consider the spectral decomposition $\Sigma = \sum_{j=1}^\infty \sigma_j\ve_j\otimes\ve_j$, where $\ve_j=\ve_j(\cdot)\in\cH$ is a real-valued function. Let $S:\cH\to L^2(\mu)$ be the canonical embedding, that is, $Sg=g$ as an $L^2(\mu)$ function. For convenience, we generally omit the $S$ in $Sg$, and treat $g$ directly as a function in $L^2$. Then $\Sigma = S^*S:\cH\to\cH$. Define $T_K = SS^*:g\in L^2(\mu)\mapsto T_Kg(\cdot) = \int_\cX K(\vx,\cdot)g(\vx)d\mu(\vx)\in L^2(\mu)$ as the integral operator associated with the kernel function $K$. For any $j\in\bN_+$, let $f_j = \frac{1}{\sqrt{\sigma_j}}S\ve_j$. Then $T_K f_j = \sigma_j f_j$ and $\langle f_i,f_j\rangle_{L^2(\mu)} = \delta_{ij}$ - $\Sigma$ and $T_K$ have the same nonzero spectrum and $\{f_j\}_{j=1}^\infty$ forms an ONB of $L^2(\mu)$.

In the kernel regression problem, $f^*$ does not necessarily belong to $\cH$. Therefore, we perform FSD directly in $L^2(\mu)$. That is, for any $J \subset \bN_+$ finite, let $V_J = \Span(f_j: j\in J)$ and $V_{J^c} = \overline{\Span(f_j: j\in J^c)}^{L^2(\mu)}$. For any $f\in L^2(\mu)$, we denote by $f_J$ the projection of $f$ onto $V_J$, and by $f_{J^c}$ the projection of $f$ onto $V_{J^c}$. For $J=J_*$, we define the optimal rate function analogously
\begin{align}\tag{\ref{eq:def_r_V_J_star}'}
    r(V_{J_*},V_{J_*^c}) = \norm{\psi_t(T_K)f_{J_*}^*}_{L^2(\mu)} +\sigma_\xi\sqrt{\frac{|J_*|}{N}} + \norm{f_{J_*^c}^*}_{L^2(\mu)} +  \sigma_\xi t\sqrt{\frac{\sum_{j\in J_*^c}\sigma_j^2}{N}}, \mbox{ where } f^* = \mathrm{Proj}_{\overline{\cH}}f^\circ.
\end{align}

\begin{Assumption}\label{assumption:bounded}
    We assume that $\|K\|_{L^\infty(\mu\otimes\mu)}\leq 1$. The noise $\xi$ satisfies either $\xi\sim\cN(0,\sigma_\xi^2)$ or $\|\xi\|_{L^\infty}<\infty$. We assume that $\xi$ is independent of $X$. Assume that there exists $f^\circ\in L^2(\mu)$ such that $Y=f^\circ(X)+\xi$ (where $f^\circ$ is not assumed to be in $\cH$).
\end{Assumption}
The following is the main result in this setting. The proof may be found in Section~\ref{sec:proof_main_RKHS}. We also provide some examples in Section~\ref{sec:examples_theo_main_RKHS} where the assumptions of Theorem~\ref{theo:main_RKHS} hold.

\begin{Theorem}\label{theo:main_RKHS}
    Let $(\Sigma,f^\circ,\sigma_\xi)$ be a kernel regression problem satisfying Assumption~\ref{assumption:bounded}.
    Let $f^* = \mathrm{Proj}_{\overline{\cH}}f^\circ$. Let $(\varphi_t)_{t\geq 1}$ be a family of filter functions satisfying Assumption~\ref{ass:filter_fct} with $\oc{c_filter_1}=0$. Suppose $1\geq t^{-1}>100\frac{\Tr(\Sigma_{J_{**}^c})}{N}$. Suppose $N\square^2\gtrsim\Tr(T_K(T_K+t^{-1}I)^{-1})\vee 1$ and $\square\leq c\log^{-1}(et)\wedge\frac{1}{9}$ for some absolute constant $c$. Suppose $k^*\geq C'\frac{\|\xi\|_{L^\infty}^2}{\sigma_\xi^2}$ if $\xi$ is bounded, or $k^*\geq C'$ if $\xi$ is Gaussian, for some absolute constant $C'$. Then there exist absolute constants $C$ and $N_0\in\bN_+$, such that for any $N\geq N_0$, with probability at least $\frac{9990}{10000}$,
    \begin{align*}
        \|\hat f_N - f^\circ \|_{L^2(\mu)} \leq C \| f^\circ - f^*\|_{L^2(\mu)} + C r(V_{J_*},V_{J_*^c}) + C\frac{\square}{t}\|T_K^{-1} f_{J_*}^*\|_{L^2(\mu)},
    \end{align*}and if $(\varphi_t)_{t\geq 1}$ further satisfies Assumption~\ref{ass:filter_fct} with $\oc{c_filter_1}>0$, then there exists an absolute constant $c$ such that with probability at least $\frac{9990}{10000},$
    \begin{align*}
        \|\hat f_N - f^\circ \|_{L^2(\mu)} \geq c\|f^\circ - f^*\|_{L^2(\mu)} + c\sigma_\xi\sqrt{\frac{k^*}{N}} + c\|\psi_t(T_K)f_J^*\|_{L^2(\mu)} + c\|f_{J_*^c}^*\|_{L^2(\mu)} - C\frac{\square}{t}\|T_K^{-1}f_{J_*}^*\|_{L^2(\mu)}.
    \end{align*}Moreover, if there exists an absolute constant $0<\delta<1$ such that $\bP(\sum_{j\in J_*^c}\sigma_j^2 f_j^2(X)\geq\frac{1}{10}\sum_{j\in J_*^c}\sigma_j^2)\geq \delta$, then with the same probability minus $\delta'$ (which is given in Lemma~\ref{lemma:lower_trace_lower_bound_variance}),
    \begin{align*}
        \|\hat f_N - f^\circ \|_{L^2(\mu)} \geq c\|f^\circ - f^*\|_{L^2(\mu)} + cr(V_{J_*},V_{J_*^c}) - C\frac{\square}{t}\|T_K^{-1}f_{J_*}^*\|_{L^2(\mu)}.
    \end{align*}
\end{Theorem}
Theorem~\ref{theo:main_RKHS} shows that, when the condition $\bP(\sum_{j\in J_*^c}\sigma_j^2 f_j^2(X)\geq\frac{1}{10}\sum_{j\in J_*^c}\sigma_j^2)\geq \delta$ holds, the estimation error of spectral methods is almost equivalent to $\|f^\circ - f^*\|_{L^2(\mu)} + r(V_{J_*},V_{J_*^c})$. In Section~\ref{sec:examples_theo_main_RKHS}, we provide several examples of RKHSs satisfying the assumption $\bP(\sum_{j\in J_*^c}\sigma_j^2 f_j^2(X)\geq\frac{1}{10}\sum_{j\in J_*^c}\sigma_j^2)\geq \delta$, including RKHSs whose eigenfunctions are uniformly bounded in $L^\infty(\mu)$, spherical inner-product kernels, Sobolev kernels, and so on.  In the proof, this assumption is only used to provide a nontrivial small-ball probability for $\sum_{j\in J_*^c}\sigma_j^2 f_j^2(X)$ relative to its expectation.

Let us now comment on the consequences of the results above.
\paragraph{Contribution to the understanding of the statistical properties of spectral methods.}

For an arbitrary linear regression problem $ (\Sigma, \bbeta^*, \sigma_\xi)$, Corollary~\ref{coro:main_coro} provides, under fairly general conditions, matching upper and lower bounds (up to a multiplicative constant) for the population excess risk of spectral methods in this problem. 
            \begin{enumerate}
                \item Compared with classical results in the statistical properties of spectral methods, such as \cite{smale_learning_2007,yao_early_2007,bauer_regularization_2007,lo_gerfo_spectral_2008,blanchard_kernel_2016,blanchard_optimal_2018,blanchard_lepskii_2019,zhang_optimality_2023,li_generalization_2024}, we observe that the classical results are typically restricted to Sobolev spaces (which impose a power decay on the eigenvalues of $\Sigma$), or require certain eigenvalue decay conditions. Among them, \cite{blanchard_kernel_2016} does not rely on power decay, but still requires the eigenvalues to satisfy certain specific decay conditions. In contrast, Theorem~\ref{theo:main} imposes no restrictions on the spectrum of $\Sigma$.
                \item In addition, the aforementioned classical literature typically assumes that $\bbeta^*$ satisfies a certain Hölder-type source condition, namely, that there exists $s > 1$ such that $\|\Sigma^{\frac{1-s}{2}}\bbeta^*\|_2$ is bounded. In contrast, our Theorem~\ref{theo:main} requires no assumptions whatsoever on $\bbeta^*$, yet still yields a precise characterization of its statistical properties.
            \end{enumerate}
            Precisely because Theorem~\ref{theo:main} yields a precise (up to a multiplicative constant) characterization of the population excess risk for any linear regression problem, it allows us to describe the statistical properties of spectral methods in the most general linear regression setting. To the best of our knowledge, this is the first result that establishes a universal statistical property of spectral methods valid for any linear regression problem.
            
From Section~\ref{sec:FSD}, we know that estimation of $\bbeta^*$ occurs only on $V_{J_*}$, while absorption of noise occurs on $V_{J_*^c}$. Theorem~\ref{theo:main} shows that, for any given linear regression problem $(\Sigma, \bbeta^*, \sigma_\xi)$ and tuning parameter $t$, the space $V_{J_*}$ where estimation takes place is determined solely by the spectrum of $\Sigma$ and the tuning parameter, and is independent of the signal $\bbeta^*$ to be approximated, the eigenvectors of $\Sigma$, and the family of filter functions $(\varphi_t)_{t\geq1}$. This observation indicates the following facts:
                \begin{enumerate}
                    \item Since $V_{J_*}$ is independent of $(\varphi_t)_{t\geq1}$, we know that for a given linear regression problem, all algorithms in the class of spectral methods decompose the feature space in the same way to estimate the signal. By examining the definition of $r(V_{J_*}, V_{J_*^c})$ in \eqref{eq:def_r_V_J_star}, we find that only the term $\|\Sigma_{J_*}^{1/2}\psi_t(\Sigma)\bbeta_{J_*}^*\|_2$ depends on the specific choice of the filter / residual functions. In other words— the only difference in the statistical properties of different spectral methods for a given linear regression problem lies in how close the residual function $\psi_t$ is to $0$ on $\{x>0: tx>b\}$—the closer it is to $0$, the better the statistical properties (i.e., the faster the convergence rate). For example, when the eigenvalues of $\Sigma$ satisfy power decay, i.e., there exists $\alpha > 1$ such that $\sigma_j \sim j^{-\alpha}$ for all $j$ (corresponding to regression problems in Sobolev spaces with sufficient smoothness), the residual function of ridge regression is $\psi_t(x) = \tfrac{1}{xt+1}$, that of gradient flow is $\psi_t(x) = \exp(-t x)$, and that of gradient descent is $\psi_t(x) = (1 - \eta x)^t$, see Example~\ref{example:spec_algo}. For the latter two, when $t x > b$, their convergence to $0$ as functions of $x$ is much faster than that of ridge regression. This provides an explanation of the saturation effect \cite{bauer_regularization_2007}: on the set $\{x > 0 : t x > b\}$, the residual function of ridge regression decays too slowly. We provide more general situations in Section~\ref{sec:partial_order}.

                    \item In Theorem~\ref{theo:main_RKHS}, our lower bound does not require imposing any condition on the components of $\phi(X)$. This is in sharp contrast to \cite{tsigler_benign_2023,wuRiskComparisonsLinear2025}, where the coordinates of the design vector are required to satisfy conditions such as symmetry or independence. This is precisely what allows our lower bound to remain valid for kernel regression.


                    \item The fact that $V_{J_*}$ does not depend on $\bbeta^*$ has the following drawback: if the alignment between $\bbeta^*$ and $\Sigma$ is poor, spectral methods exhibit unfavorable statistical properties. For example, when the support of $\bbeta^*$ satisfies $\supp(\bbeta^*) = V_{J_*^c}$, the term $\|\Sigma_{J_*^c}^{1/2}\bbeta_{J_*^c}^*\|_2$ in $r(V_{J_*},V_{J_*^c})$ reduces to $\|\langle X, \bbeta^* \rangle\|_{L^2(\mu)}$, which may be big. Of course, one can change $V_{J_*}$ by adjusting the tuning parameter $t$, but we stress that statisticians usually do not know the support of $\bbeta^*$ in the basis of eigenvectors of  $\Sigma$ in advance, and hence cannot preselect an appropriate $t$. Therefore, unlike statistical algorithms with the sparsity inducing property such as basis pursuit or the LASSO (see \cite[pp. 30]{shangFeatureSpaceDecomposition2026}), or the algorithms with feature learning property like mean-field Langevin dynamics (see \cite{lecueGeneralizationErrorMean2025}), the fact that $V_{J_*}$ is independent with $\bbeta^*$ implies that, when the signal and the eigenvectors of $\Sigma$ are poorly aligned, spectral methods generally have inferior statistical performance. We discuss further in Section~\ref{sec:partial_order}. 

                    \item Theorem~\ref{theo:main}, Theorem~\ref{theo:main_LB}, and Theorem~\ref{theo:main_RKHS} in fact establish the Gaussian universality of spectral methods in RKHSs; see \cite{gavrilopoulos_geometrical_2025}.
                \end{enumerate}

From Example~\ref{example:spec_algo}, we know that the residual function of gradient flow is smaller than the one of ridge regression. Therefore, for a given linear regression problem and for the same tuning parameter, we always have $r^{(\mathrm{GF})}(V_{J_*},V_{J_*^c}) \leq r^{(\mathrm{Ridge})}(V_{J_*},V_{J_*^c})$. This means that, from the perspective of population excess risk, whenever one can choose between ridge regression and gradient flow, gradient flow should always be preferred, regardless of the specific linear regression problem under consideration. 
In Section~\ref{sec:partial_order}, we will further discuss the notion of pre-order on the set of spectral algorithms.

\paragraph{Contribution within the FSD series of papers.}
The high-level idea of the proof of Theorem~\ref{theo:main} is to wrap the classical analysis of the statistical properties of spectral methods, such as \cite{li_generalization_2024}, with a FSD layer—namely, instead of analyzing the statistical properties over the entire feature space $\mathbb{R}^p$, we restrict the analysis to $V_J$, while on $V_{J^c}$ we perform only a signal-free analysis. Remarkably, we obtain the precise result of Theorem~\ref{theo:main}. We therefore believe that the proof of Theorem~\ref{theo:main} itself suggests that the FSD method may serve as a systematic tool in mathematical statistics for deriving precise non-asymptotic results on the population excess risk of general supervised learning algorithms.
        
Theorem~\ref{theo:main} can be regarded as an extension of the results of \cite{mei_generalization_2022,tsigler_benign_2023,cheng_dimension_2022,barzilai_generalization_2024,gavrilopoulos_geometrical_2025} on ridge regression to spectral methods. In this theorem, we apply the FSD method for the first time to estimators beyond ridge regression and the minimum norm interpolant estimator. Unlike the ridge results in \cite{mei_generalization_2022,tsigler_benign_2023,cheng_dimension_2022,barzilai_generalization_2024,gavrilopoulos_geometrical_2025}, in \eqref{eq:def_estimation_dimension} we do not observe an ``effective regularization'' term of the form $Nt^{-1} + \mathrm{Tr}(\Sigma_{J^c})$. This is because we only consider the well-regularized regime, namely, when the spectral algorithm is far from overfitting. The overfitting regime of spectral methods—for example, when the running time $t$ of gradient descent/flow tends to infinity—yields the minimum $\ell_2$ norm interpolant estimator, which has already been studied in \cite{tsigler_benign_2020,lecue_geometrical_2024}.

\section{Pre-order of Spectral Algorithms, Generalized Saturation Effect, and Absence of Feature Learning}\label{sec:partial_order}
Thanks to the FSD method, Corollary~\ref{coro:main_coro} provides matching upper and lower bounds for arbitrary $\cR$, rather than being restricted to a specific spectrum decay or a particular class of $\bbeta^*$. Therefore, in a rough sense, Corollary~\ref{coro:main_coro} characterize the following fact: the random variable $\|\Sigma^{1/2}(\hat\bbeta-\bbeta^*)\|_2$ is ``equivalent'', with high probability, to the real number $r(V_{J_*},V_{J_*^c})$. Consequently, for any $\cR$, comparing the population excess risk of two spectral methods is reduced to comparing two real numbers.
Corollary~\ref{coro:main_coro} also enables us to generalize the definition of the saturation effect. In fact, to the best of our knowledge, the notion of saturation effect was first introduced by \cite{bauer_regularization_2007}. It describes the following phenomenon: when $\sigma_j$ exhibits power decay, i.e., there exists $\alpha > 1$ such that for any $j \in [p]$ we have $\sigma_j \sim j^{-\alpha}$ (a classical result for nonparametric regression in Sobolev spaces), and when there exists $s \geq 1$ such that $\|\Sigma^{\frac{1-s}{2}} \bbeta^*\|_2 < \infty$ (meaning that $\bbeta^*$ has good smoothness in the eigen-basis of $\Sigma$), ridge regression, even with the optimal tuning parameter, achieves a (squared loss) population excess risk convergence rate of only $N^{-\frac{\alpha (s \wedge 2)}{1+\alpha (s \wedge 2)}}$. Since larger $s$ corresponds to smoother $\bbeta^*$ (in the Fourier sense), one might expect ridge regression to exploit this information and achieve a faster convergence rate; however, ridge regression saturates—when $s \geq 2$, the convergence rate of ridge regression cannot be further improved. This phenomenon is called the saturation effect. \cite{bauer_regularization_2007} showed that such a saturation effect arises because ridge regression corresponds to a finite value of $\tau$ in Assumption~\ref{assumption:classical}, item \emph{2.}. For gradient flow/descent, we have $\tau = \infty$, and thus no saturation effect occurs.

In this section, we use Corollary~\ref{coro:main_coro} to define a generalized saturation effect, specify the conditions under which it occurs, and provide a geometric perspective on the phenomenon.  

\subsection{Pre-order on Spectral Algorithms}

We begin by extending the definition of the saturation effect. The original definition given in \cite{bauer_regularization_2007} was intended to describe the relative advantages and disadvantages of ridge regression versus other spectral methods (such as gradient flow) for certain specific statistical problems (e.g., regression on Sobolev spaces). We follow this line of thought to generalize the definition. Since a spectral algorithm is uniquely determined by its filter function, we consider two spectral methods $\hat{\bbeta}_{t_A}^{(A)}$ and $\hat{\bbeta}_{t_B}^{(B)}$ with parameters $t_A,t_B$ , and with filter functions $\varphi_{t_A}^{(A)}$ and $\varphi_{t_B}^{(B)}$, respectively. By Theorem~\ref{theo:main}, there exist $r_{t_A}^{(A)}(V_{J_*}^{(A)},V_{J_*^c}^{(A)})$ and $r_{t_B}^{(B)}(V_{J_*}^{(B)},V_{J_*^c}^{(B)})$ characterizing the squared loss population excess risk $\|\Sigma^{1/2}(\hat{\bbeta}_{t_A}^{(A)} - \bbeta^*)\|_2$ and $\|\Sigma^{1/2}(\hat{\bbeta}_{t_B}^{(B)} - \bbeta^*)\|_2$ for these two spectral methods in this linear regression problem. Given any $\cR = (\Sigma, \bbeta^*, \sigma_\xi) \in \mathbb{R}^{p \times p} \times \mathbb{R}^p \times \mathbb{R}$, we define the following pre-order ``$\preceq_\cR$'' on the set of all spectral methods.

\begin{Definition}[Pre-order of Spectral Algorithms in Linear Regression Problems]\label{def:partial_order}
    For the linear regression problem $\cR := (\Sigma, \bbeta^*, \sigma_\xi)$, we write $\hat{\bbeta}_{t_A}^{(A)} \preceq_\cR \hat{\bbeta}_{t_B}^{(B)}$ if $r_{t_A}^{(A)}(V_{J_*}^{(A)},V_{J_*^c}^{(A)}) = O\!\left(r_{t_B}^{(B)}(V_{J_*}^{(B)},V_{J_*^c}^{(B)})\right)$ as $N$ and $p$ go to infinity. In particular, if $r_{t_A}^{(A)}(V_{J_*}^{(A)},V_{J_*^c}^{(A)}) = \Theta\!\left(r_{t_B}^{(B)}(V_{J_*}^{(B)},V_{J_*^c}^{(B)})\right)$, we write $\hat{\bbeta}_{t_A}^{(A)} \asymp_\cR \hat{\bbeta}_{t_B}^{(B)}$. It is straightforward to verify that ``$\asymp_\cR$'' defines an equivalence relation on the set of all spectral methods, while $\preceq_\cR$ defines a pre-order.
\end{Definition}

Definition~\ref{def:partial_order} describes, for a specific linear regression problem $\cR = (\Sigma, \bbeta^*, \sigma_\xi)$, the relative speed of convergence of the population excess risk for any two given spectral methods $\hat\bbeta_{t_A}^{(A)}$ and $\hat\bbeta_{t_B}^{(B)}$, thereby characterizing the relative performance of different spectral methods for that problem.

In the following, we consider the case when $t_A=t_B$.
Since the choice of $V_{J_*}$ for a given $t\geq1$ in the optimal decomposition of the feature space given by Theorem~\ref{theo:main} is universal for any spectral algorithm (see \eqref{eq:def_estimation_dimension}), it follows that, for any fixed $(\Sigma, \bbeta^*, \sigma_\xi)$, Theorem~\ref{theo:main} can be applied to any spectral algorithm to obtain the corresponding $r(V_{J_*}, V_{J_*^c})$. In the sense of equality up to a multiplicative constant, the squared loss population excess risk of each spectral algorithm differs only in the bias term $\|\Sigma_{J_*}^{1/2} \psi_t(\Sigma) \bbeta_{J_*}^*\|_2$ of $\hat{\bbeta}_J$. This means that, for any spectral algorithm $\hat{\bbeta}$, the variance of $\hat{\bbeta}_J$ and both the bias and variance of $\hat{\bbeta}_{J^c}$ are identical—the only difference lies in the convergence rate of $\hat{\bbeta}_J$ used to estimate $\bbeta_{J}$. Therefore, we have the following corollary.

\begin{Corollary}\label{coro:partial_order}
    Given any linear regression problem $\cR=(\Sigma,\bbeta^*,\sigma_\xi)$. For any $t\geq1$ satisfying the assumptions of Theorem~\ref{theo:main} and Theorem~\ref{theo:main_LB}, $\hat\bbeta_t^{(A)}\preceq_\cR \hat\bbeta_t^{(B)}$ if and only if as $N$ and $p$ go to infinity
    \begin{align*}
        \norm{\Sigma_{J_*}^{\frac{1}{2}}\psi_t^{(A)}(\Sigma)\bbeta_{J_*}^*}_2 = O\bigg( \norm{\Sigma_{J_*}^{\frac{1}{2}}\psi_t^{(B)}(\Sigma)\bbeta_{J_*}^*}_2 \bigg).
    \end{align*}
\end{Corollary}

Corollary~\ref{coro:partial_order} characterizes the following: for any two spectral methods, given the same $t$, if they satisfy the assumptions of Corollary~\ref{coro:main_coro}, then the necessary and sufficient condition for the pre-order $\preceq_\cR$ depends solely on the bias of $\hat{\bbeta}_J$—which is consistent with our intuition—because, as we noted in Section~\ref{sec:FSD}, only the component of $\bbeta^*$ projected onto $V_{J_*}$ is actually estimated by $\hat{\bbeta}$. We emphasize that Corollary~\ref{coro:partial_order} itself merely provides a formal verification of the definition introduced in Definition~\ref{def:partial_order}. However, when the conditions of Corollary~\ref{coro:main_coro} are satisfied, Definition~\ref{def:partial_order} genuinely reflects the population excess risk associated with the corresponding spectral methods. Consequently, Corollary~\ref{coro:partial_order} captures the pre-order of spectral methods in terms of their population excess risk. We further stress that this framework is particularly effective for comparing the population excess risk of ridge regression with that of other spectral methods, since the lower bound for ridge does not require any condition (see \cite{gavrilopoulos_geometrical_2025}). This observation naturally leads to the following corollary. The following corollary is a direct consequence of the elementary inequality $\exp(-tx) \leq 1/(1+xt)$.
\begin{Corollary}[GF outperforms Ridge]\label{coro:spiked_cov}
    For any linear regression problem such hat \eqref{eq:condition_for_matching_UB_LB} holds, $\varphi_t^{(\mathrm{GF})}\preceq_\cR\varphi_t^{(\mathrm{Ridge})}$, where $\varphi_t^{(\mathrm{Ridge})}$ is the filter function of ridge regression, \eqref{eq:def_ridge}; while $\varphi_t^{(\mathrm{GF})}$ is the filter function of gradient flow, \eqref{eq:def_GF}.
\end{Corollary}


For a fixed parameter $t$, the difference in population excess risk between different spectral methods arises from the structure of their residual function $\psi_t$, and this naturally leads to the saturation effect -- the cause of the saturation effect also lies in the properties of the residual function. We first introduce the following generalized definition.

\begin{Definition}[Generalized Saturation Effect]\label{def:generalized_saturation}
    For any linear regression problem $\cR$, any interval $I\subset[1,+\infty)$ and families of filter functions $\{\varphi_t^{(A)}\}_{t\geq1}$ and $\{\varphi_t^{(B)}\}_{t\geq1}$, we write $\{\varphi_t^{(A)}\}_{t\in I} \preceq_{\cR} \{\varphi_t^{(B)}\}_{t\in I}$ if as $N$ and $p$ go to infinity
    \begin{align*}
        \inf\!\left( r_{t_A}^{(A)}(V_{J_*},V_{J_*^c}) : t_A \in I \right) = O\!\left( \inf\!\left( r_{t_B}^{(B)}(V_{J_*},V_{J_*^c}) : t_B \in I \right) \right).
    \end{align*}
    If $\{\varphi_t^{(A)}\}_{t\in I} \preceq_{\cR} \{\varphi_t^{(B)}\}_{t\in I}$, we say that the spectral algorithm $\hat{\bbeta}^{(B)}$ defined by the filter function family $\{\varphi_t^{(B)}\}_{t\geq1}$ is saturated compared to the filter function family $\{\varphi_t^{(A)}\}_{t\geq1}$ in $I$. In particular, if $I = \bR_+$, we write $\{\varphi_t^{(A)}\}_{t\geq1} \preceq_{\cR} \{\varphi_t^{(B)}\}_{t\geq1}$ and say that the spectral algorithm $\hat{\bbeta}^{(B)}$ defined by the filter function family $\{\varphi_t^{(B)}\}_{t\geq1}$ is saturated compared to the filter function family $\{\varphi_t^{(A)}\}_{t\geq1}$. It is straightforward to verify that $\preceq_{\cR}$ is a pre-order. Similarly, we can define an equivalence relation $\asymp_{\cR}$ on families of filter functions. When big-$O$ is replaced by small-$o$, we denote by $\prec_{\cR}.$
\end{Definition}
Definition~\ref{def:partial_order} describes the relative performance of two spectral methods for given parameters $t_A$ and $t_B$, whereas Definition~\ref{def:generalized_saturation} concerns their relative performance under their respective optimal parameters within interval $I$. It is easy to see that the classical saturation effect defined in \cite{bauer_regularization_2007} corresponds  to the pre-order on the following set of linear regression problems.
\begin{align*}
    \cR\in \fR_{\mathrm{Sob}}(s,\alpha):= \bigg\{ (\Sigma,\bbeta^*,\sigma_\xi):\,  \Sigma = \sum_{j=1}^p \sigma_j\ve_j\otimes\ve_j, \, \sigma_j\sim j^{-\alpha},\, \|\Sigma^{\frac{1-s}{2}}\bbeta^*\|_2<\infty,\sigma_\xi \mbox{ is constant} \bigg\}.
\end{align*}Moreover, in \cite{bauer_regularization_2007}, $\{r_t^{(B)}\}_{t\geq1}$ is the family of ridge regression, \eqref{eq:def_ridge}.
In addition, on $\fR_{\mathrm{Sob}}(s,\alpha)$, the optimal tuning parameter is $t^{-1} \sim N^{-\frac{\alpha}{1 + \tilde{s} \alpha}}$, where $\tilde{s} = s \wedge \tau$ and $\tau$ is defined in Assumption~\ref{assumption:classical}, item~\emph{2}. We say this choice is optimal, because it achieves the minimax rate on $\fR_{\mathrm{Sob}}$, \cite{li_asymptotic_2023}. 
Applying to $\varphi_t^{(A)} : x \mapsto (1 - \exp(-t x))/x$, i.e., gradient flow \eqref{eq:def_GF}, and to $\varphi_t^{(B)} : x \mapsto (x + t^{-1})^{-1}$, i.e., ridge regression \eqref{eq:def_ridge}, we have the following. For the same $t \sim N^{\frac{\alpha}{1 + \tilde{s} \alpha}}$, \cite{gavrilopoulos_geometrical_2025} computed that $\|\Sigma_{J_*}^{1/2} \psi_t^{(B)}(\Sigma) \bbeta_{J_*}^*\|_2 \sim N^{-\frac{\alpha (s \wedge 2)}{1 + \alpha (s \wedge 2)}}$, while the following corollary yields $\|\Sigma_{J_*}^{1/2} \psi_t^{(A)}(\Sigma) \bbeta_{J_*}^*\|_2 \sim N^{-\frac{\alpha s}{1 + \alpha s}}$. Combined with Corollary~\ref{coro:partial_order}, this recovers the classical saturation effect in the sense of \cite{bauer_regularization_2007}. The proof of Corollary~\ref{coro:Sobo} may be found in Section~\ref{sec:proof_Coro_Sobo}
\begin{Corollary}[Saturation Effect in Sobolev Space]\label{coro:Sobo}
    Let $\varphi_t^{(\mathrm{GF})} : x \mapsto (1 - \exp(-t x))/x$ and $\varphi_t^{(\mathrm{Ridge})} : x \mapsto (x + t^{-1})^{-1}$. Let $\cR\in\fR_{\mathrm{Sob}}(s,\alpha)$. We have $\{\varphi^{(\mathrm{GF})}\}_{t\geq1}\preceq_{\cR} \{\varphi^{(\mathrm{Ridge})}\}_{t\geq1}$. Moreover, when $t^{-1}\sim N^{-\frac{\alpha}{1 + \tilde{s} \alpha}}$, where $\tilde s=s\wedge 2$ for ridge regression, and $\tilde s=s$ for gradient flow, we have $(r_t^{(\mathrm{GF})}(V_{J_*},V_{J_*^c}))^2\sim N^{-\frac{\alpha s}{1+s\alpha}}$ and $(r_t^{(\mathrm{Ridge})}(V_{J_*},V_{J_*^c}))^2\sim N^{-\frac{\alpha\tilde s}{1+\tilde s\alpha}}$.
\end{Corollary}
Here, however, we offer a geometric perspective on the classical saturation effect: its occurrence is due to the fact that, on $V_{J_*}$, the residual function of ridge regression decays too slowly in the eigen-basis with power decay, compared to the residual function of gradient flow. We emphasize that Corollary~\ref{coro:partial_order} provides not only this most classical example of the saturation effect in Sobolev spaces, but also necessary and sufficient conditions for the occurrence of more general saturation effects.

\begin{Corollary}[Saturation effect in the plateau covariance model]\label{coro:saturation}
    Suppose there exists some $k\lesssim N\lesssim p-k$, $\sigma>\varepsilon>0$ such that $\sigma_1=\cdots=\sigma_k=\sigma$, and $\sigma_{k+1}=\cdots=\sigma_p=\varepsilon$. Let $J = \{1,\cdots,k\}$ and suppose there exists a real number $\alpha_*>0$ such that $|\langle\bbeta^*,\ve_j\rangle|=\alpha_*$ for any $j\in J$ while $\langle\bbeta^*,\ve_j\rangle=0$ otherwise. Let
    \begin{align*}
        \mathrm{SNR} = \frac{\|\Sigma^{1/2}\bbeta^*\|_2}{\sigma_\xi}\frac{\sigma\sqrt{N}}{\sqrt{\Tr(\Sigma_{J^c}^2)}}.
    \end{align*}Suppose $4<\mathrm{SNR}\leq b\frac{\sigma}{\varepsilon}$, where $b$ is from \eqref{eq:def_estimation_dimension}. Let $I = \{t>1:\, b^{-1}\varepsilon\leq t^{-1}<\sigma\}$. Then
    \begin{align*}
    \min_{t\in I}r^{(\mathrm{GF})}(V_{J_*},V_{J_*^c})\leq \min_{t\in I} r^{(\mathrm{Ridge})}(V_{J_*},V_{J_*^c}).
    \end{align*}
    Moreover, when $\mathrm{SNR}\to\infty$ and $\sigma=\Omega(\varepsilon)$, $\{\varphi_t^{(\mathrm{Ridge})}\}_{t\in I}\prec_{\cR}\{\varphi_t^{(\mathrm{GF})}\}_{t\in I}$.
\end{Corollary}The proof of Corollary~\ref{coro:saturation} may be found in Section~\ref{sec:proof_coro_saturation}. The quantity $\mathrm{SNR}$ in Corollary~\ref{coro:saturation} can be interpreted as a signal-to-noise ratio, but it is rescaled according to the sample size and the spectrum of $\Sigma$. The lower bound in the condition $4 < \mathrm{SNR} \leq b\frac{\sigma}{\varepsilon}$ is intended to ensure that the signal-to-noise ratio is not too small, while the upper bound is rather mild. For example, if we take $\sigma = 1$, $\varepsilon = (p-k)^{-1}$, and $\|\Sigma_J^{1/2}\bbeta^*\|_2 / \sigma_\xi$ to be a constant, then this condition is satisfied.
This corollary considers the case where the signal $\bbeta^*$ is well aligned with the covariance structure, and shows that the saturation effect occurs over a rather broad range of tuning parameters (which is reasonable, since the tuning parameter is neither too large, causing overfitting, nor too small, leading to underfitting). This illustrates our claim that the saturation effect is a fairly general phenomenon in linear regression problems.

We conclude this subsection with the following observation. By Corollary~\ref{coro:partial_order}, we know that for any $\cR$, the smaller $\psi_t(x)$ is on the interval $\{x: xt > b\}$, the smaller it is in the sense of the pre-order $\preceq_\cR$. Hence, the minimal element of this pre-order should satisfy $\psi_t(x)=0$ for all $x>bt^{-1}$. The PCR method precisely satisfies this condition and therefore should be regarded as the minimal element in the set of spectral methods under the pre-order defined by any linear regression problem. We emphasize, however, that although this formally satisfies Definition~\ref{def:partial_order}, our Theorem~\ref{theo:main} does not support assigning a statistical meaning to this definition. This is because the filter function corresponding to PCR does not satisfy Assumption~\ref{ass:filter_fct}—that is, although PCR is a classical spectral method, its filter function cannot be analytically extended to an open subset of the complex plane containing the entire spectrum, and thus Theorem~\ref{theo:main} does not apply. However, by modifying the definition of the contour and following the same proof strategy as ours, we derive in Theorem~\ref{theo:main_PCR} an upper bound on the population excess risk for PCR. At present, we do not know how to obtain a corresponding lower bound. We conjecture that by incorporating the FSD framework into the classical analysis of the population excess risk of PCR (see, e.g., \cite{blanchard_optimal_2018,zhang_optimality_2023,huckerNotePredictionError2023}), one could extend both Theorem~\ref{theo:main} and Theorem~\ref{theo:main_LB} to spectral methods that are not necessarily analytically continuable, thereby encompassing the analysis of PCR.


\subsection{The efficiency of alignment}\label{sec:no_feature}

An estimator $\hat f_N$ satisfying the alignment property defined in Definition~\ref{def:alignment_property} has the following characteristic: when the target function $f_\cH$ is well aligned $V_{J_*}$, the estimator can exploit this structure and thereby achieve a smaller estimation error $\|\hat f_N - f_\cH\|_{L^2(\mu)}$. Many estimators are known to satisfy this property, for instance, the spectral methods investigated in this paper, as shown by Theorem~\ref{theo:main} and Theorem~\ref{theo:main_RKHS}.

The alignment property characterizes the ability of an estimator to exploit an existing favorable geometric alignment; however, such an alignment is not always inherently present. When this alignment relationship is suboptimal, certain estimators yield large prediction errors. To investigate the impact of the unfavorable alignment between the target function $f_\cH$ and the eigenfunctions of the covariance operator of $\cH$ on the estimation error of all the spectral methods, we introduce the following concept of the efficiency of alignment.

\begin{Definition}\label{def:no_need_for_feature_learning}
    Let $(\mu, f^*, \xi)$ be a supervised regression problem, $\sigma_\xi^2$ be the variance of $\xi$, $\mathcal{H}$ an RKHS, and $N$ be the sample size. Let $T_K:L^2(\mu)\to L^2(\mu)$ be the integral operator of $\cH$, and $\{f_j\}_{j=1}^\infty$ be its eigenfunctions. Let $0 < \alpha < 1$ be a pre-specified threshold. We define the balance dimension $k^\circ(N) = \min\{k \in \bN : \| P_{k+1:\infty}f^*\|_{L^2(\mu)}^2 \leq \sigma_{\xi}^2 \frac{k}{N}\}$, where $P_{k+1:\infty}$ is the projection onto $\overline{\Span(f_j : j > k)}$. If $k^\circ(N) \leq \alpha N$, we say that alignment is efficient; if $k^\circ(N) > \alpha N$, we say that alignment is deficient.
\end{Definition}

Definition~\ref{def:no_need_for_feature_learning} is intended to characterize the relationship between the critical estimation dimension and the sample size $N$. Based on Section~\ref{sec:FSD}, the projection $P_{k+1:\infty}f^*$ is not estimated and thus enters the estimation error as the price for no estimation. Meanwhile, $\sigma_\xi^2\frac{k}{N}$ represents the variance term associated with the estimation subspace. Since both terms are independent of the specific choice of the tuning parameter and filter functions, the balance dimension—by describing the equilibrium between these two components—reveals the impact on the estimation error that depends exclusively on the alignment of $f^*$ with the eigenvectors of the integral operator of $\cH$, regardless of the specific filter function and tuning parameter. 

In particular, by the definition of $k^\circ(N)$, for any $k \in \bN$ (especially for the estimation dimension $k^*$), we always have $\|P_{k^\circ + 1:\infty}f^*\|_{L^2(\mu)}^2 + \sigma_\xi^2 \frac{k^\circ}{N} \leq 2 ( \|P_{k + 1:\infty}f^*\|_{L^2(\mu)}^2 + \sigma_\xi^2 \frac{k}{N} )$. This implies that the optimal rate function of FSD is always subject to a lower bound provided by $\sigma_\xi^2 \frac{k^\circ}{N}$, regardless of which filter function or the tuning parameter is selected. When the balance dimension is excessively large, this lower bound can be quite substantial, leading to a suboptimal estimation error. A key element of the feature learning property introduced in \cite{lecueGeneralizationErrorMean2025} is that feature learning can automatically construct favorable geometric alignments through the autonomous learning of features.

Below, we provide some representative examples where alignment is efficient and deficient respectively. The proof of Proposition~\ref{prop:single_index_problem} and Proposition~\ref{prop:sobolev_regression} may be found in Section~\ref{sec:proof_single_index_problem} and in Section~\ref{sec:proof_sobolev_regression} respectively.
\begin{Proposition}\label{prop:single_index_problem}
    Let $L$ be a positive integer, $\sigma: \mathbb{R} \to \mathbb{R}$ be a polynomial of degree $L$, and $\vw^*$ be a vector on the unit Euclidean sphere. Suppose $f^*(\vx) = \sigma(\langle \vw^*, \vx \rangle)$, where $X$ follows the uniform distribution on the unit Euclidean sphere and $\xi \sim \mathcal{N}(0, \sigma_\xi^2)$ is independent of $X$. Let $K_{\mathrm{ran}}: (\vx_1, \vx_2) \in \Omega_X \times \Omega_X \mapsto \mathbb{E}[\sigma(\langle \vx_1, W \rangle) \sigma(\langle \vx_2, W \rangle)]$, where $W$ is uniformly distributed on the unit Euclidean sphere, and let $\mathcal{H}$ be the RKHS generated by $K_{\mathrm{ran}}$. 
    Let $D_L$ be a positive integer depending only on $d$ and $L$ (see Section~\ref{sec:proof_single_index_problem} for precise definition), and $D_L\sim d^L$. Let $\alpha = \frac{|\langle f^*,\ve_{D_L}\rangle|^2}{2\sigma_\xi^2}$. Then the following conclusions hold.
    \begin{enumerate}
        \item If $N\leq D_L$, alignment is deficient, and $k^\circ(N)\geq \min(\lceil 2\alpha N\rceil, D_L)$.
        \item If $N>\frac{1}{\alpha}D_L$, alignment is efficient, and $k^\circ(N)=D_L$.
    \end{enumerate}
\end{Proposition}

The single-index model serves as a prototypical example in the literature for demonstrating the importance of feature learning, \cite{ba_high-dimensional_2022,bietti_learning_2022,damian_neural_2022}. In this setting, for the same sample size, the estimation error of neural networks possessing the ``feature learning property'' is significantly smaller than that of ridge regression employing rotationally invariant kernels, \cite{donhauser_how_2021}. This phenomenon is thus widely used to illustrate the crucial role of feature learning in improving estimation performance. 
In Proposition~\ref{prop:single_index_problem}, we observe a phase transition. The intuitive explanation is as follows: when the sample size is sufficiently large, i.e., $N \gtrsim d^L \sim D_L$, the learning problem effectively reduces to a low-dimensional one. In this regime, the role of alignment is not significant, as the ample sample size is sufficient to overcome the impact of the non-optimal alignment between $f^*$ and the eigenfunctions. However, when the sample size is small, this unfavorable alignment leads to a substantial price for no estimation. To mitigate this price, the estimator on $\mathcal{H}$ tends to reduce the free subspace, which is equivalent to expanding the estimation subspace; however, this results in an increase in the variance term within the estimation error. A direct consequence of Proposition~\ref{prop:single_index_problem} is that any spectral algorithm satisfying Assumption~\ref{ass:filter_fct} with $K_{\mathrm{ran}}$ cannot overcome the polynomial barrier in the well-regularized regime, that is, when $t^{-1}>100\frac{\Tr(\Sigma_{J_{**}^c})}{N}$. The reason is twofold: first, when $N\leq D_L$, this is due to a deficiency of alignment; second, when $N>\frac{1}{\alpha}D_L$, this is because the alignment provided by the inner product kernel itself makes the balance dimension of spectral methods too large. 

\begin{Proposition}\label{prop:sobolev_regression}
    Assume that $\cH$ is a Sobolev-type RKHS, i.e., it satisfies $\sigma_j \leq C j^{-\alpha}$ where $C$ is an absolute constant and $\alpha > 1$. Suppose $f^*$ satisfies the source condition, namely, there exists $s \geq 0$ such that $\|T_K^{-\frac{s}{2}}f^*\|_{L^2(\mu)} < \infty$. Under this setting, for any $N \in \bN_+$ and any $\sigma_{\xi} > 0$, it holds that $k^\circ(N) \lesssim N^{\frac{1}{1+s\alpha}}$, which implies that the alignment is efficient as long as $s>0$.
\end{Proposition}
Proposition~\ref{prop:sobolev_regression} investigates the classical Sobolev space regression problem in nonparametric statistics, showing that for such problems, the alignment is always favorable provided that $s > 0$. Here, the favorable alignment is guaranteed by the source condition, which ensures good alignment through the rapid decay of the coefficients of $f^*$ with respect to the eigenvectors of $\Sigma$. In fact, $k^\circ(N)$'s upper bound coincides with the estimation dimension of spectral methods with infinite qualification ($\tau = \infty$ in Remark~\ref{assumption:classical}) at the optimal tuning parameter.

\section{Conclusions and Future Work}\label{sec:conclusions}
Corollary~\ref{coro:main_coro} establishes the first matching high-probability upper and lower bounds on the population excess risk under squared loss that hold for \emph{any} linear regression problem $(\Sigma,\bbeta^*,\sigma_\xi)$. This result enables us to define a pre-order over the class of spectral methods according to their rates of convergence in population excess risk for a given regression problem, and, in turn, to extend the notion of the saturation effect. 

Our proof strategy follows the following scheme: we wrap the FSD method around the classical analysis of the statistical properties of spectral methods together with the analysis of the noise absorption part to obtain precise characterizations of the population excess risk of any spectral methods under Assumption~\ref{ass:filter_fct}. This demonstrates that the FSD method may serve as a general tool to sharpen population excess risk bounds for other classical estimators—most notably, upgrading minimax optimality bounds to problem-specific optimality for a given regression problem - that is for a target dependent bounds and not a worst case analysis. The present analysis of the statistical properties of spectral methods constitutes the first application of FSD beyond ridge regression and minimum-norm interpolant estimators. We hope to see future work exploiting FSD to analyze a broader range of estimators. For instance, an interesting research direction is to apply the FSD method to the analysis of the Nadaraya–Watson estimator, aiming to obtain sharp bounds for every $\Sigma$ and $\bbeta^*$, rather than just obtaining a generic convergence rate like $N^{-\gamma}$ for some $\gamma>0$.

Finally, the spectral methods studied in this paper concern scalar-valued supervised regression problems. An interesting future direction is to apply the approach developed here to investigate the population excess risk of spectral methods in vector-valued RKHSs, \cite{alvarezKernelsVectorValuedFunctions2012}, or more generally, in reproducing kernel Hilbert C*-modules, \cite{hashimoto_reproducing_2021}. Such an extension would provide new insights into classical methods used in functional data analysis, kernel mean embedding \cite{muandetKernelMeanEmbedding2017}, and related problems.

\section*{Acknowledgments}

This work originated during ZS's long-term visit to the Center for Statistical Science at Tsinghua University, and ZS gratefully acknowledges Professor Qian Lin for his warm hospitality. Part of the work was completed during ZS's long-term visit to RIKEN–AIP, Japan, where he thanks Taiji Suzuki for his generous hospitality. The authors are grateful to Gilles Blanchard, Kenji Fukumizu, Jaouad Mourtada, and Taiji Suzuki for stimulating discussions, and especially to Martin Wahl for many valuable conversations during ZS's visit to Bielefeld University.

\section{Proof of the upper bound in Theorem~\ref{theo:main}}\label{sec:proof_main}

We abbreviate $J_*$ by $J$ in this section, i.e. $J=[k^*]$ where $k^*$ is the estimation dimension from Definition~\ref{def:estimation_dimension}. Following the FSD method, we recall the risk decomposition of $\hat\bbeta$ given by
\begin{equation}\label{eq:risk_decomposition_origin_FSD}
 \norm{\Sigma^{1/2}\left(\hat\bbeta - \bbeta^*\right)}_2  
    \leq \norm{\Sigma_J^{1/2}\left(\hat\bbeta_J - \bbeta^*_J\right)}_2 + \norm{\Sigma_{J^c}^{1/2}\hat\bbeta_{J^c}}_2 +\norm{ \Sigma_{J^c}^{1/2}\bbeta_{J^c}^*}_2
\end{equation}where $\hat\bbeta_J=P_J \hat\bbeta$ and $\hat\bbeta_{J^c}=P_{J^c} \hat\bbeta$. The next two sections are devoted to show high probability upper bounds on the estimation part $\norm{\Sigma_J^{1/2}\left(\hat\bbeta_J - \bbeta^*_J\right)}_2$ and the noise absorption part $\norm{\Sigma_{J^c}^{1/2}\hat\bbeta_{J^c}}_2$ appearing in \eqref{eq:risk_decomposition_origin_FSD}. 

In multiple occasions, we will use the following relations that follows for instance from SVD: we recall that $P_J:\bR^p\to \bR^p$ is the projection operator onto $V_J$ and $\bX_J^\top := [P_J X_1| \cdots| P_J X_N]$. We have $\bX_J = \bX P_J$, $\bX_J^\top = P_J \bX^\top$ and $\hat \Sigma_J:=\frac{1}{N}\bX_J^\top \bX_J = P_J \hat\Sigma P_J$. Since, $V_J$ is an eigen-space of $\Sigma$, we also have $P_J \varphi_t(\Sigma) \Sigma = \varphi_t(\Sigma_J) \Sigma_J$ where $\Sigma_J := \bE (P_J X)(P_J X)^\top = P_J \Sigma P_J$. We also define $\Sigma_t = \Sigma + t^{-1}I_p$ and $\hat \Sigma_t = \hat \Sigma + t^{-1}I_p$. 

It also follows from the definition of  $k^*$ that $b^{-1}\sigma_{k^*+1} \leq t^{-1} \leq b^{-1}\sigma_{k^*}$. Consequently,
\begin{align}\label{eq:compare_Sigma_t_and_Sigma_J_Jc}
    \begin{aligned}
        &\norm{\Sigma_J^{\frac{1}{2}}\Sigma_t^{-\frac{1}{2}}}_{\op} \leq \norm{\Sigma^{\frac{1}{2}}\Sigma_t^{-\frac{1}{2}}}_{\op}   \leq 1,\, \norm{\Sigma_{J^c}^{\frac{1}{2}}\Sigma_t^{-\frac{1}{2}}}_{\op} \leq \sqrt{\frac{b}{1+b}}\mbox{ and } \norm{\Sigma_J^{-\frac{1}{2}}\Sigma_t^{\frac{1}{2}}}_{\op}\leq \sqrt{\frac{1+b}{b}}.
    \end{aligned}
\end{align}We also have from the definition of $k^*$ that for all $x\in V_J$, 
\begin{equation}\label{eq:Sigma_t_equiv_Sigma_on_VJ}
    \norm{\Sigma_t^{1/2}x}_2^2 = \norm{\Sigma_J^{1/2} x}_2^2 + t^{-1}\norm{x}_2^2 \leq \frac{1+b}{b}\norm{\Sigma_J^{1/2}x}_2^2
\end{equation}because $bt^{-1}\norm{x}_2^2\leq \sigma_{k^*}\norm{x}_2^2\leq \norm{\Sigma_J^{1/2}x}_2^2$.

\subsection{The main property of \texorpdfstring{$\hat \Sigma$}{hat Sigma} required for the analysis and the event \texorpdfstring{$\Omega_t$}{Omega t}.}


The main uniform property we need $\hat \Sigma$ to satisfy for the analysis is the one from the following event: let $0<\square<1/9$ (a typical choice of $\square$ will be of the order of $\log^{-1}(et)$), we consider the event
\begin{align}\label{eq:def_Omega_isomorphy}
    \begin{aligned}
        \Omega_t := \left\{ \norm{ \Sigma_t^{-1/2} (\hat\Sigma - \Sigma) \Sigma_t^{-1/2} }_{\op}\leq \square \right\}.
    \end{aligned}
\end{align} We show in the next result that $\Omega_t$ holds with large probability as long as $\square^2 N$ is larger than the effective rank $\Tr\left[\Sigma(\Sigma+t^{-1}I_p)^{-1}\right]$.
\begin{Lemma}\label{lemma:IP}
    Grant Assumption~\ref{assumption:main_upper}. Let $t\geq1$ and assume that $\square^2 N\gtrsim \Tr\left[\Sigma(\Sigma+t^{-1}I_p)^{-1}\right]$ and $\square^2 N\gtrsim 1$. There exists an absolute constant $c>0$ such that $\Omega_t$ happens with probability at least $1-\exp(-c\square^2 N)$.
\end{Lemma}
\beginproof 
    It follows from Theorem~5.5 in \cite{dirksen_tail_2015} on the control of empirical quadratic processes and the sub-gaussian assumption from Assumption~\ref{assumption:main_upper}   that  there is an absolute constant $C\geq1$ such that for all $r\geq 1$, with probability at least $1-\exp(-r)$,
\begin{equation}\label{eq:sjoerd}
\begin{aligned}
 \sup_{f\in F}\left|\frac{1}{N}\sum_{i=1}^N f^2(X_i) - \bE f^2(X)\right|\leq C&\bigg( \frac{D\gamma_2}{\sqrt{N}}+ \frac{\gamma_2^2}{N}+ D^2\left(\sqrt{\frac{r}{N}} + \frac{r}{N}\right)\bigg)
\end{aligned}
\end{equation} where $\gamma_2 = \gamma_2(F,\norm{\cdot}_{L^2(\mu)})$ is Talagrand's $\gamma_2$-functional of $\cF$ with respect the $L^2(\mu)$-norm \cite[Definition 2.2.19]{talagrand_upper_2014} and $D= \diam(F,L^2(\mu)):=\sup(\norm{f}_{L^2(\mu)}:f\in F)$. Applying \eqref{eq:sjoerd} to $F = \{\langle\cdot,\vv\rangle:\, \vv\in\Sigma_t^{-1/2}S_2^{p-1}\}$ where $S_2^{p-1}$ is the unit $\ell_2^p$-sphere, we have $D=\diam(F,L^2(\mu)) = \norm{\Sigma^{1/2}\Sigma_t^{-1/2}}_{op}\leq 1$ and $\gamma_2(F,\|\cdot\|_{L^2(\mu)})\sim \bE\norm{\Sigma^{1/2}\Sigma_t^{-1/2}G}_2\sim\sqrt{\Tr(\Sigma(\Sigma+t^{-1}I_p)^{-1})}$ where $G\sim\cN(0, I_p)$. As a consequence, it follows from the sample complexity assumption $\square^2 N\gtrsim \Tr\left[\Sigma(\Sigma+t^{-1}I_p)^{-1}\right]$ that for $r = \square^2 N/(16C^2)$ (which is larger than $1$ since we assumed that $\square^2 N\gtrsim 1$), with probability at least $1-\exp(-\square^2 N/(16C^2))$,
\begin{align*}
     &\norm{ \Sigma_t^{-1/2} (\hat\Sigma - \Sigma) \Sigma_t^{-1/2} }_{\op} = \sup_{\vu\in S_2^{p-1}}\bigg|\vu^\top \Sigma_t^{-1/2}(\hat\Sigma-\Sigma)\Sigma_t^{-1/2} \vu\bigg|\\
     &= \sup_{\vu\in S_2^{p-1}}\bigg|\|\hat\Sigma^{\frac{1}{2}}\Sigma_t^{-\frac{1}{2}}\vu\|_2^2 - \|\Sigma^{\frac{1}{2}}\Sigma_t^{-\frac{1}{2}}\vu\|_2^2\bigg| = \sup_{\vu\in S_2^{p-1}}\bigg|\frac{1}{N}\sum_{i=1}^N\langle \Sigma_t^{-\frac{1}{2}} \vu,X_i\rangle^2 - \bE\langle \Sigma_t^{-\frac{1}{2}} \vu,X_i\rangle^2 \bigg|\leq \square.
\end{align*}
\endproof

The sample complexity assumption $\square^2 N\gtrsim \Tr\left[\Sigma(\Sigma+t^{-1}I_p)^{-1}\right]$ is classical in the analysis of spectral methods. It has some consequences on the definition of the estimation dimension $k^*$. Indeed, one has
\begin{equation*}
     \Tr\left[\Sigma(\Sigma+t^{-1}I_p)^{-1}\right] = \sum_j \frac{\sigma_j}{\sigma_j + t^{-1}} = \sum_{j\in J}\frac{\sigma_j}{\sigma_j + t^{-1}} + \sum_{j\notin J}\frac{\sigma_j}{\sigma_j + t^{-1}} 
 \end{equation*} where we recall that $J = \{j:\sigma_j\geq bt^{-1}\}$ is of cardinality $k^*$, by definition of $k^*$ and so 
\begin{equation}\label{eq:sample_complexity_ass_imply}
  \frac{bk^*}{1+b} + \frac{t}{1+b} \Tr[\Sigma_{J^c}] \leq   \Tr\left[\Sigma(\Sigma+t^{-1}I_p)^{-1}\right] \leq k^* + t \Tr[\Sigma_{J^c}].
\end{equation} As a consequence, the sample complexity assumption implies both $\square^2 N\gtrsim b k^*$ - meaning that we require the estimation dimension to be smaller than $N$ - and $\square^2 N \gtrsim t \Tr[\Sigma_{J^c}]$ implying that the estimation dimension of ridge obtained in \cite{gavrilopoulos_geometrical_2025} coincides with the one used here in Definition~\ref{def:estimation_dimension}, i.e. $k^{**} = k^{*}$, for other spectral methods.

In the classical analysis of spectral methods, the property induced by the event $\Omega_t$ is referred as the ``Change-of-Norm argument'' (see, for example, \cite{JMLR:v22:20-358}). From a geometrical perspective, the event $\Omega_t$ is the union of two type of events that are part of the FSD method. Indeed, $\Omega_t$ is equivalent to: for all $\vu\in\bR^p$,
\begin{equation}\label{eq:change_of_norm}
    \left|\norm{\hat\Sigma^{1/2} \vu}_2^2 - \norm{\Sigma^{1/2}\vu}_2^2\right|\leq \square \norm{\Sigma_t^{1/2}\vu}_2^2.
\end{equation} As a consequence, there are two regimes depending on the relative values of $\norm{\Sigma^{1/2}\vu}_2$ and $\norm{\Sigma_t^{1/2}\vu}_2$ that can be described via the following cone
\begin{equation}\label{eq:cone}
    C := \left\{ \vu \in \bR^p: \square\norm{\Sigma_t^{1/2}\vu}_2^2\leq \frac{1}{2}\norm{\Sigma^{1/2}\vu}_2^2 \right\}= \left\{ \vu \in \bR^p: \square t^{-1}\norm{\vu}_2^2\leq \left(\frac{1}{2}-\square\right)\norm{\Sigma^{1/2}\vu}_2^2 \right\}.
\end{equation}Then, we consider the decomposition of $\bR^p$ as the union: $\bR^p=C \cup C^c$. This decomposition is closed to the one of  the FSD $\bR^p = V_J \oplus^\perp V_{J^c}$ since one can see that $C$ contains all singular vectors of $\Sigma$ with singular values such that $\sigma_j \gtrsim \square t^{-1}$ which is, up to the $\square$ term, the inequality appearing in the definition of $k^*$. We see that an isomorphic property restricted to this cone follows from \eqref{eq:change_of_norm}: for all $\vu\in C$, 
\begin{equation*}
    \frac{1}{\sqrt{2}}\norm{\Sigma^{1/2}\vu}_2 \leq \norm{\hat \Sigma^{1/2}\vu}_2  \leq \sqrt{\frac{3}{2}}\norm{\Sigma^{1/2}\vu}_2. 
\end{equation*}This type of 'RIP' (i.e. restricted isomorphic property) is expected in the FSD method on the estimation part of the feature space i.e. $V_J$ or the slightly bigger cone $C$. On the 'noise absorption part' of the feature space, i.e. $V_{J^c}$ - or the slightly bigger cone $C^c$, when $\square$ is of the order of a constant - we don't need such an isomorphic property but only a control of the largest 'restricted' singular value of $\hat\Sigma$: for all $\vu\notin C$, 
\begin{equation*}
    \norm{\hat \Sigma^{1/2}\vu}_2  \leq \sqrt{3\square} \norm{\Sigma_t^{1/2}\vu}_2 = \sqrt{3\square} \left(\norm{\Sigma^{1/2}\vu}_2^2+t^{-1}\norm{\vu}_2^2\right)^{1/2}\leq 3\sqrt{t^{-1}\square}\norm{\vu}_2\leq \sqrt{t^{-1}}\norm{\vu}_2.
\end{equation*}
 In particular, we see that, on the event $\Omega_t$, for all $\vu\in\bR^p$, we have 
\begin{equation*}
   \norm{\hat \Sigma^{1/2}\vu}_2  \leq  \max\left(\sqrt{3/2}\norm{\Sigma^{1/2}\vu}_2, \sqrt{t^{-1}}\norm{\vu}_2\right)
\end{equation*}In particular, the following Lemma holds.

\begin{Lemma}\label{lem:largest_sing_Hat_Sigma}
     On the event $\Omega_t$, $\hat \sigma_1 = \norm{\hat \Sigma}_{op}\leq 4(\sigma_1+t^{-1})$.
 \end{Lemma}

For our proof strategy, it is important to localize the spectrum of $\hat \Sigma$. Indeed, the spectral method $\hat\bbeta$ depends on the filter function via the term $\varphi_t(\hat\Sigma)$ in its definition from  \eqref{eq:def_hat_bbeta}). In particular, we will need to tell how $\varphi_t(\hat\Sigma)$ is close to $\varphi_t(\Sigma)$ . However, it is well-known that for a general non-linear function $f$ (for which the spectral calculus is well-defined), $\bE[f(\hat\Sigma)] \neq f(\Sigma)$; for example, when $f(x) = x^2$. This illustrates that $f(\hat\Sigma)$, as a plug-in estimator for $f(\Sigma)$, is a biased estimator (in fact, this is one of the motivations behind \cite{koltchinskii_asymptotic_2018}).  Methods for handling this bias have been developed in \cite{li_generalization_2024}, they are based on the residue theorem: for any counterclock-wise contour $\cC_t$ surrounding both spectra of $\hat\Sigma$ and  $\Sigma$, we have
\begin{align}\label{eq:residual_theorem}
    \begin{aligned}
    \varphi_t(\hat\Sigma)-\varphi_t(\Sigma)&=-\frac{1}{2 \pi i} \oint_{\cC_t} \varphi_t(z)\left[(\hat\Sigma -z I_p)^{-1}-(\Sigma - zI_p)^{-1}\right] d z\\ 
    &=\frac{1}{2 \pi i} \oint_{\mathcal{C}_t}(\hat\Sigma -z I_p)^{-1}(\hat\Sigma - \Sigma)(\Sigma - zI_p)^{-1} \varphi_t(z) d z.
\end{aligned}
\end{align}In particular, for the choice of contour $\cC_t$ from Section~\ref{sec:choice_contour}, we have $\cC_t$ surrounding both spectra of $\hat\Sigma$ and of $\Sigma$ on the event $\Omega_t$ thanks to Lemma~\ref{lem:largest_sing_Hat_Sigma}. So that the residue theorem applies to both $\varphi_t(\hat\Sigma)$ and  $\varphi_t(\Sigma)$ and the formulae above is valid on $\Omega_t$.
Next, to handle the summand in this integral, we use the following lemma taken from \cite{li_generalization_2024}.
\begin{Lemma}[\cite{li_generalization_2024}]\label{lemma:spectral_calculus_li} There exists an absolute constant $C>1$ such that the following holds. Let $t\geq1$. For the contour $\cC_t$ defined in \eqref{contour} and  for any $z\in\cC_t$, we have
\begin{align*}
    \left\|\Sigma_t^{\frac{1}{2}}(\Sigma-zI_p)^{-1} \Sigma_t^{\frac{1}{2}}\right\|_{\op} \leq C,\, \oint_{\mathcal{C}_t}\left|\varphi_t(z) d z\right|  \leq C \log(t),\mbox{ and } \oint_{\mathcal{C}_{t}}\left|\psi_t(z) d z\right|  \leq C t^{-1}.
\end{align*}
Moreover, on $\Omega_t$, for any $z\in\cC_t$, we further have
    \begin{align*}
        \left\|\Sigma_t^{\frac{1}{2}}\left(\hat\Sigma-zI_p\right)^{-1} \Sigma_t^{\frac{1}{2}}\right\|_{\op} \leq  C.
    \end{align*}
\end{Lemma}For the sake of completeness, we provide the proof of Lemma~\ref{lemma:spectral_calculus_li} in Section~\ref{sec:proof_Lemma_Li}.
On the event $\Omega_t$, other properties that will be useful in our analysis hold. For instance, to obtain an upper bound for $\|\Sigma_J^{1/2}(\hat\bbeta_J-\bbeta^*_J)\|_2$, we will further require the following result.
\begin{Lemma}\label{lemma:operator_norm_Sigma_J}Let $t\geq1$ and recall that $\hat\Sigma_t = \hat\Sigma + t^{-1} I_p$. On the event $\Omega_t$, we have   $\|\Sigma_J^{\frac{1}{2}}\hat{\Sigma}_t^{-\frac{1}{2}}\|_{\op}^2\leq \|\Sigma_t^{\frac{1}{2}}\hat{\Sigma}_t^{-\frac{1}{2}}\|_{\op}^2  \le 2$ and $\|\Sigma_t^{-\frac{1}{2}}\hat{\Sigma}_t^{\frac{1}{2}}\|_{\op}^2 \le 2$.
\end{Lemma}
Lemma~\ref{lemma:operator_norm_Sigma_J} provides the following insight: for a suitably chosen $J$, the (modified) population covariance and the (modified) sample covariance can be interchanged. 
The proof of Lemma~\ref{lemma:operator_norm_Sigma_J} may be found in Section~\ref{sec:proof_operator_norm_Sigma_J}

The event $\Omega_t$ contains all the properties on $\hat \Sigma$ that are enough for our analysis. The only remaining stochastic argument used in the proof from now are only dealing with the noise. As a consequence, if one wants to extend the conclusion from Theorem~\ref{theo:main} beyond Assumption~\ref{assumption:main_upper}, one may only focus on proving that $\Omega_t$ happens with large probability under the new considered setup. Now, that we have dealt with mostly all the stochastic aspect of the proof we can move to the deterministic one, as long as we work on the event $\Omega_t$.

\subsection{The estimation property of \texorpdfstring{$\hat\bbeta_{J}$}{hat beta J}}\label{sec:estimation}

In this subsection, we investigate the estimation properties of $\hat\bbeta_J$, i.e. we obtain a high probability upper bound on $\norm{\Sigma_J^{1/2}(\hat\bbeta_J - \bbeta^*_J)}_2$. In the following analysis, we will see that the estimation error analysis for the estimator on $V_J$, namely $\hat\bbeta_J$, is similar to the classical analysis of spectral methods but performed over $V_J$. This is because on this subspace the problem reduces to standard estimation. 
From this perspective, the FSD method can be viewed as an additional layer around classical analysis only requiring an isomorphic property on the estimation space instead of the entire space, thereby providing better estimation properties under smaller sample complexity.

\subsubsection{Risk decomposition of the estimation part \texorpdfstring{$\hat \bbeta_J$}{hat beta J}.}\label{ssub:risk_decomposition_of_the_estimation_part_hat_bbeta_j}
We start with a risk decomposition of the estimation part $\hat \bbeta_J$ of the spectral method $\hat \bbeta$. Let the ``population'' spectral method be defined as $\tilde\bbeta = \varphi_t(\Sigma)\Sigma\bbeta^*$. It is the 'population version' of $\hat\bbeta(\bX\bbeta^*) = \varphi_t(\hat\Sigma)\hat\Sigma\bbeta^*$ where $\hat \Sigma$ has been replaced by $\Sigma$; we therefore look at $\hat\bbeta(\bX\bbeta^*)$ as a plug-in estimator of  $\tilde\bbeta$ in the noise free case and in the estimation part of the feature space.  Then, by linearity of $\hat\bbeta$, we may decompose $\hat\bbeta_J - \bbeta^*_J$ as follows:
\begin{align*}
    \hat\bbeta_J(\vy) - \bbeta^*_J = \hat\bbeta_J(\bX\bbeta^*_J) - \bbeta^*_J + \hat\bbeta_J(\bX\bbeta_{J^c}^*+\bxi) = \left(\hat\bbeta_J(\bX\bbeta^*_J) - \tilde\bbeta_J\right) + \left(\tilde\bbeta_J - \bbeta^*_J\right) + \hat\bbeta_J(\bX\bbeta_{J^c}^*+\bxi).
\end{align*} Here, $\hat\bbeta_J(\bX\bbeta^*_J) - \tilde\bbeta_J$ plays the role of a bias term of the plug-in estimator $\hat\bbeta_J(\bX\bbeta^*_J)$ in the free noise case, while $\tilde\bbeta_J - \bbeta^*_J$ denotes an approximation error and $\hat\bbeta_J(\bX\bbeta_{J^c}^*+\bxi)$ is considered as a variance term. The following risk decomposition follows from the decomposition above:
\begin{align}\label{eq:risk_decomp_esti_part}
    &\norm{\Sigma_J^{1/2}(\hat\bbeta_J - \bbeta^*_J)}_2 \leq \norm{\Sigma_J^{1/2}(\hat\bbeta_J(\bX\bbeta^*_J) - \tilde\bbeta_J)}_2 + \norm{\Sigma_J^{1/2}(\tilde\bbeta_J - \bbeta^*_J)}_2 + \norm{\Sigma_J^{1/2}\hat\bbeta_J(\bX\bbeta_{J^c}^*+\bxi)}_2.
\end{align}Next, we upper bound the three terms from this sum.

\subsubsection{Upper bound on the approximation term \texorpdfstring{$\norm{\Sigma_J^{1/2}(\tilde\bbeta_J - \bbeta^*_J)}_2$}{tilde beta J}.}  It follows from the definition of the residual function $\psi_t:x\in\bR^+\to 1-x\varphi_t(x)$ that  
$\tilde\bbeta_J - \bbeta^*_J = (\varphi_t(\Sigma)\Sigma - I_p)\bbeta^*_J = -\psi_t(\Sigma)\bbeta^*_J$ and so
\begin{equation}\label{eq:approximation_term_esti_part}
    \norm{\Sigma_J^{1/2}(\tilde\bbeta_J - \bbeta^*_J)}_2= \norm{\Sigma_J^{1/2}\psi_t(\Sigma)\bbeta^*_J}_2.
\end{equation} Next, we move to an upper bound on the bias of the plug-in estimator $\hat\bbeta_J(\bX\bbeta^*_J)$. We will see that the approximation term above is dominating the bias term.

\subsubsection{Upper bound on the bias term \texorpdfstring{$\norm{\Sigma_J^{1/2}(\hat\bbeta_J(\bX\bbeta^*_J) - \tilde\bbeta_J)}_2$}{lower order terms}.} The filter and residual functions satisfy the relation $\varphi_t(x) x + \psi_t(x) = 1$, hence, we have
\begin{align*}
    \begin{aligned}
    &\hat\bbeta_J(\bX\bbeta^*_J) - \tilde\bbeta_J  = 
  P_J\varphi(\hat \Sigma)\hat \Sigma \bbeta^*_J - P_J(\varphi_t(\hat\Sigma)\hat\Sigma+ \psi_t(\hat\Sigma))\tilde\bbeta_J = P_J\varphi_t(\hat\Sigma)\hat \Sigma (\bbeta^*_J - \tilde\bbeta_J) - P_J\psi_t(\hat\Sigma)\tilde\bbeta_J \\
  & = P_J\varphi_t(\hat\Sigma)(\hat\Sigma-\Sigma) (\bbeta^*_J-\tilde\bbeta_J) + P_J\varphi_t(\hat\Sigma)\Sigma (\bbeta^*_J-\tilde\bbeta_J)- P_J\psi_t(\hat\Sigma)\tilde\bbeta_J\\
  &= P_J\varphi_t(\hat\Sigma)(\hat\Sigma-\Sigma) (\bbeta^*_J-\tilde\bbeta_J) + P_J\left(\varphi_t(\hat\Sigma)-\varphi_t(\Sigma)\right)\Sigma\psi_t(\Sigma)\bbeta^*_J + P_J\left(\psi_t(\Sigma)-\psi_t(\hat\Sigma)\right)\Sigma\varphi_t(\Sigma)\bbeta^*_J
\end{aligned}
\end{align*}where we used the fact that $\tilde\bbeta_J :=P_J \tilde \bbeta =\varphi_t(\Sigma_J)\Sigma_J\bbeta^*_J = \varphi_t(\Sigma)\Sigma\bbeta^*_J$ and so $\bbeta^*_J - \tilde\bbeta_J = \psi_t(\Sigma)\bbeta^*_J$  because $V_J$ is an eigenspace of $\Sigma$ and the fact that $\Sigma$, $\varphi_t(\Sigma)$ and $\psi_t(\Sigma)$ commute. Now, by taking $\|\Sigma_J^{1/2}\cdot\|_2$ on both sides, we obtain the following decomposition of the bias term:
\begin{align}\label{eq:def_bias_plug_in}
\begin{aligned}
    \|\Sigma_J^{1/2}(\hat\bbeta_J(\bX\bbeta^*_J)-\tilde\bbeta_J)\|_2\leq & \left\|\Sigma_J^{1/2}\varphi_t(\hat\Sigma)(\hat\Sigma-\Sigma) (\bbeta^*_J-\tilde\bbeta_J)\right\|_2 + \left\|\Sigma_J^{1/2}\left(\varphi_t(\hat\Sigma)-\varphi_t(\Sigma)\right)\Sigma\psi_t(\Sigma)\bbeta^*_J\right\|_2\\ & + \left\|\Sigma_J^{1/2}\left(\psi_t(\Sigma)-\psi_t(\hat\Sigma)\right)\Sigma\varphi_t(\Sigma)\bbeta^*_J\right\|_2.
\end{aligned}
\end{align}Next, we provide upper bounds on the three terms in this sum.

\paragraph{Upper bound for $\left\|\Sigma_J^{1/2}\varphi_t(\hat\Sigma)(\hat\Sigma-\Sigma) (\bbeta^*_J-\tilde\bbeta_J)\right\|_2$.} We recall that $\hat\Sigma_t=\hat\Sigma + t^{-1} I_p$. We have
\begin{equation}\label{eq:bias_21}
    \begin{aligned}
        \left\|\Sigma_J^{1/2}\varphi_t(\hat\Sigma)(\hat\Sigma-\Sigma) (\bbeta^*_J-\tilde\bbeta_J)\right\|_2 
        &\leq \|\Sigma_J^{\frac{1}{2}}\Sigma_t^{-\frac{1}{2}}\|_{\op}\|\Sigma_t^{\frac{1}{2}}\varphi_t(\hat\Sigma)\Sigma_t^{\frac{1}{2}}\|_{\op}\|\Sigma_t^{-\frac{1}{2}}(\hat\Sigma-\Sigma)\Sigma_t^{-\frac{1}{2}}\|_{\op} \|\Sigma_t^{\frac{1}{2}}(\bbeta^*_J-\tilde\bbeta_J)\|_2.
\end{aligned}
\end{equation}
    Under Assumption~\ref{ass:filter_fct}, we know that $\varphi_t(x) \leq \oC{C_filter_1}(x+t^{-1})^{-1}$ hence,  by Lemma~\ref{lemma:operator_norm_Sigma_J}, we have, on $\Omega_t$, 
    \begin{equation}\label{eq:IP_varphi_t_hat_Sigma}
    \begin{aligned}
        \|\Sigma_t^{\frac{1}{2}}\varphi_t(\hat\Sigma)\Sigma_t^{\frac{1}{2}}\|_{\op} \leq 
\norm{\Sigma_t^{\frac{1}{2}} \hat \Sigma_t^{-\frac{1}{2}}}_{\op} \norm{\hat \Sigma_t^{\frac{1}{2}} \varphi_t(\hat \Sigma) \hat \Sigma_t^{\frac{1}{2}}}_{\op} \norm{\hat \Sigma_t^{\frac{1}{2}} \Sigma_t^{-\frac{1}{2}}}_{\op} \leq 2 \oC{C_filter_1}.
    \end{aligned} 
    \end{equation}Moreover, by \eqref{eq:compare_Sigma_t_and_Sigma_J_Jc}, $\|\Sigma_J^{\frac{1}{2}}\Sigma_t^{-\frac{1}{2}}\|_{\op}\leq 1$.
    Plugging \eqref{eq:IP_varphi_t_hat_Sigma} into \eqref{eq:bias_21} together with \eqref{eq:def_Omega_isomorphy}, on $\Omega_t$, we have
    \begin{align}\label{eq:result_upper_bias_1}
     \begin{aligned}
        &\left\|\Sigma_J^{\frac{1}{2}}\varphi_t(\hat\Sigma)(\hat\Sigma-\Sigma) (\bbeta^*_J-\tilde\bbeta_J)\right\|_2 \leq  2\square C_1\|\Sigma_t^{\frac{1}{2}}(\bbeta^*_J-\tilde\bbeta_J)\|_2 \\
        &\leq 2 C_1\left(\frac{1+b}{b}\right) \square\|\Sigma_J^{1/2}(\bbeta^*_J-\tilde\bbeta_J)\|_2\leq 2C_1\left(\frac{1+b}{b}\right) \square \norm{\Sigma_J^{1/2}\psi_t(\Sigma)\bbeta^*_J}_2
    \end{aligned}
    \end{align}where we used \eqref{eq:approximation_term_esti_part} and \eqref{eq:Sigma_t_equiv_Sigma_on_VJ} in the last inequality. 
    
\paragraph{Upper bound for $\left\|\Sigma_J^{1/2}(\varphi_t(\hat\Sigma)-\varphi_t(\Sigma))\Sigma\psi_t(\Sigma)\bbeta^*_J\right\|_2$.} To handle this term, we use \eqref{eq:residual_theorem} which is valid on $\Omega_t$: on $\Omega_t$, we have
    \begin{equation*}
    \begin{aligned}
        & \Sigma_J^{\frac{1}{2}}(\varphi_t(\hat\Sigma)-\varphi_t(\Sigma)) \Sigma \psi_t(\Sigma)\bbeta^*_J = \frac{1}{2 \pi i} \oint_{\mathcal{C}_t} \Sigma_J^{\frac{1}{2}}\left(\hat\Sigma-zI_p\right)^{-1}\left(\hat\Sigma-\Sigma\right)(\Sigma-zI_p)^{-1}  \Sigma \psi_t(\Sigma)\bbeta^*_J \varphi_t(z) d z \\
        = & \frac{1}{2 \pi i} \oint_{\mathcal{C}_t} \Sigma_J^{\frac{1}{2}} \Sigma_t^{-\frac{1}{2}} \Sigma_t^{\frac{1}{2}}   \left(\hat\Sigma-zI_p\right)^{-1} \Sigma_t^{\frac{1}{2}} \Sigma_t^{-\frac{1}{2}}\left(\Sigma-\hat\Sigma\right) \Sigma_t^{-\frac{1}{2}} \Sigma_t^{\frac{1}{2}}(\Sigma-zI_p)^{-1} \Sigma^{\frac{1}{2}} \Sigma^{\frac{1}{2}} \psi_t(\Sigma)\bbeta^*_J\varphi_t(z) d z .
    \end{aligned}
    \end{equation*}
    Taking the $\|\cdot\|_2$ norm on both sides and applying Lemma~\ref{lemma:spectral_calculus_li} yields, on $\Omega_t$,
    \begin{align}\label{eq:result_upper_bias_2}
        \nonumber & \left\|\Sigma_J^{\frac{1}{2}}\left(\varphi_t(\hat\Sigma)-\varphi_t(\Sigma)\right)\Sigma\psi_t(\Sigma)\bbeta^*_J\right\|_2 \\
       \nonumber  \leq &\|\Sigma_J^{\frac{1}{2}}\Sigma_t^{-\frac{1}{2}}\|_{\op} \oint_{\mathcal{C}_t}\left\|\Sigma_t^{\frac{1}{2}}\left(\hat\Sigma-zI_p\right)^{-1} \Sigma_t^{\frac{1}{2}}\right\|_{op} \left\|\Sigma_t^{-\frac{1}{2}}\left(\Sigma-\hat\Sigma\right) \Sigma_t^{-\frac{1}{2}}\right\|_{op}   \left\|\Sigma_t^{\frac{1}{2}}(\Sigma-zI_p)^{-1} \Sigma^{\frac{1}{2}}\right\|_{op}\left\|\Sigma^{\frac{1}{2}} \psi_t(\Sigma)\bbeta^*_J\right\|_{ 2 }\left|\varphi_t(z) d z\right| \\
        \lesssim &  \square \left\|\Sigma^{\frac{1}{2}} \psi_t(\Sigma)\bbeta^*_J\right\|_{ 2 }\oint_{\mathcal{C}_t}\left|\varphi_t(z) d z\right| \lesssim  \square \log(t) \left\|\Sigma^{\frac{1}{2}} \psi_t(\Sigma)\bbeta^*_J\right\|_{ 2 },
    \end{align}where we have used that $\Sigma\preceq \Sigma_t$  to get $\left\|\Sigma_t^{\frac{1}{2}}(\Sigma-zI_p)^{-1} \Sigma^{\frac{1}{2}}\right\|_{op}\leq \left\|\Sigma_t^{\frac{1}{2}}(\Sigma-zI_p)^{-1} \Sigma_t^{\frac{1}{2}}\right\|_{op}\lesssim 1$ from Lemma~\ref{lemma:spectral_calculus_li}.

 \paragraph{Upper bound for $\left\|\Sigma_J^{1/2}\left(\psi_t(\Sigma)-\psi_t(\hat\Sigma)\right)\Sigma\varphi_t(\Sigma)\bbeta^*_J\right\|_2$.} We have on $\Omega_t$ and from \eqref{eq:residual_theorem} 
    \begin{equation*}
    \begin{aligned}
        \Sigma_J^{\frac{1}{2}}&\left(\psi_t(\hat\Sigma)-\psi_t(\Sigma)\right) \Sigma \varphi_t(\Sigma) \bbeta^*_J = \frac{1}{2 \pi i} \oint_{ \mathcal{C} _t} \Sigma_J^{\frac{1}{2}}\left(\hat\Sigma-zI_p\right)^{-1}\left(\hat\Sigma-\Sigma\right)(\Sigma-zI_p)^{-1} \Sigma \varphi_t(\Sigma) \bbeta^*_J \psi_t(z)  d z \\
        = & \frac{1}{2 \pi i} \oint_{ \mathcal{C} _t} \Sigma_J^{\frac{1}{2}}  \left(\hat\Sigma-zI_p\right)^{-1} \Sigma_t^{\frac{1}{2}} \cdot \Sigma_t^{-\frac{1}{2}}\left(\hat\Sigma-\Sigma\right) \Sigma_t^{-\frac{1}{2}} \Sigma_t^{\frac{1}{2}}(\Sigma-zI_p)^{-1} \Sigma_J^{\frac{1}{2}} \cdot \Sigma_J^{\frac{1}{2}} \varphi_t(\Sigma) \bbeta^*_J \psi_t(z) d z .
    \end{aligned}
    \end{equation*}
    Therefore,
    \begin{align}\label{eq:result_upper_bias_3}
    \begin{aligned}
        & \left\|\Sigma_J^{\frac{1}{2}}\left(\psi_{t}(\Sigma)-\psi_{t}(\hat\Sigma)\right)\Sigma\varphi_{t}(\Sigma)\bbeta^*_J\right\|_2  \leq \frac{1}{2\pi} \oint_{\mathcal{C}_t} \left\|\Sigma_t^{\frac{1}{2}}\left(\hat\Sigma-zI_p\right)^{-1} \Sigma_t^{\frac{1}{2}}\right\|_{op} \cdot\left\|\Sigma_t^{-\frac{1}{2}}\left(\hat\Sigma - \Sigma\right) \Sigma_t^{-\frac{1}{2}}\right\|_{op} \\
        &  \cdot\left\|\Sigma_t^{\frac{1}{2}}(\Sigma-zI_p)^{-1} \Sigma_J^{\frac{1}{2}}\right\|_{op} \cdot\left\|\Sigma_J^{\frac{1}{2}} \varphi_t(\Sigma) \bbeta^*_J\right\|_{2}\left|\psi_t(z) d z\right| \\
        &\lesssim   \square \cdot\left\|\Sigma_J^{\frac{1}{2}} \varphi_t(\Sigma) \bbeta^*_J\right\|_{ 2 } \cdot \oint_{\mathcal{C}_t}\left|\psi_t(z) d z\right|\lesssim \square \left\|\Sigma_J^{1/2} \varphi_t(\Sigma) \bbeta^*_J\right\|_{ 2 } t^{-1}.
    \end{aligned}
    \end{align}

Collecting \eqref{eq:result_upper_bias_1}, \eqref{eq:result_upper_bias_2} and \eqref{eq:result_upper_bias_3} all together in \eqref{eq:def_bias_plug_in}, we obtain that, on $\Omega_t$, it holds
\begin{align}\label{eq:expectation_bias_plug_in_result_v0}
    &\norm{\Sigma_J^{1/2}(\hat\bbeta_J(\bX\bbeta^*_J) - \tilde\bbeta_J)}_2 \lesssim \square\left(\log(et)\norm{\Sigma_J^{1/2} \psi_t(\Sigma)\bbeta^*_J}_2 + t^{-1}\norm{\Sigma_J^{1/2}\varphi_t(\Sigma)\bbeta^*_J}_2\right)
\end{align}and since $\varphi_t(\Sigma)\preceq C_1 \Sigma_t^{-1}$, we obtain $\norm{\Sigma_J^{1/2}\varphi_t(\Sigma)\bbeta^*_J}_2\leq C_1 \norm{\Sigma_J^{-1/2}\bbeta^*_J}_2$, we finally get, on $\Omega_t$,
\begin{equation}\label{eq:expectation_bias_plug_in_result}
   \norm{\Sigma_J^{1/2}(\hat\bbeta_J(\bX\bbeta^*_J) - \tilde\bbeta_J)}_2 \lesssim \square\left(\log(et)\norm{\Sigma_J^{1/2} \psi_t(\Sigma)\bbeta^*_J}_2 + t^{-1}\norm{\Sigma_J^{-1/2}\bbeta^*_J}_2\right). 
\end{equation}


\subsubsection{Upper bound on the variance term \texorpdfstring{$\norm{\Sigma_J^{1/2}\hat\bbeta_J(\bX\bbeta_{J^c}^*+\bxi)}_2$}{hat beta J}.}\label{ssub:upper_bound_on_the_variance_term_texorpdfstring_norm_sigma_j_1_2} By linearity of the spectral estimator (see \eqref{eq:def_hat_bbeta}), we have
\begin{equation*}
    \|\Sigma_J^{1/2}\hat\bbeta_J(\bX\bbeta_{J^c}^*+\bxi)\|_2 \leq \|\Sigma_J^{1/2}\varphi_t(\hat\Sigma)\hat\Sigma\bbeta_{J^c}^*\|_2 + \frac{1}{N}\|\Sigma_J^{1/2}\varphi_t(\hat\Sigma)\bX^\top\bxi\|_2.
\end{equation*}Now, we prove high probability upper bounds on the two terms from the sum above.

\paragraph{Upper bound for \texorpdfstring{$\|\Sigma_J^{1/2}\varphi_t(\hat\Sigma)\hat\Sigma\bbeta_{J^c}^*\|_2$}{variance caused by not estimating beta Jc *}.} We have 
\begin{equation*}
    \|\Sigma_J^{1/2}\varphi_t(\hat\Sigma)\hat\Sigma\bbeta_{J^c}^*\|_2 = \frac{1}{N}\norm{\Sigma_J^{1/2}\varphi_t(\hat\Sigma)\Sigma_t^{\frac12} \Sigma_{t}^{-\frac12}\bX^\top \bX \bbeta_{J^c}^*}_2\leq \frac{1}{\sqrt{N}}\norm{\Sigma_J^{1/2}\varphi_t(\hat\Sigma)\Sigma_t^{1/2}}_{\op} \norm{\Sigma_t^{-1/2}\bX^\top}_{\op} \frac{\|\bX\bbeta_{J^c}^*\|_2}{\sqrt{N}}.
\end{equation*}It follows from \eqref{eq:IP_varphi_t_hat_Sigma}, \eqref{eq:compare_Sigma_t_and_Sigma_J_Jc} and Lemma~\ref{lemma:operator_norm_Sigma_J}, that on the event $\Omega_t$,
\begin{equation}\label{eq:upper_operator_norm_Sigma_t_inverse_X_transpose}
    \frac{1}{\sqrt{N}}\norm{\Sigma_t^{-1/2}\bX^\top}_{\op} =  \norm{\Sigma_t^{-1/2}\hat \Sigma^{1/2}}_{\op} \leq \sqrt{2}
\end{equation}and 
\begin{align*}
    \norm{\Sigma_J^{1/2}\varphi_t(\hat\Sigma)\Sigma_t^{1/2}}_{\op}\leq \norm{\Sigma_J^{1/2}\Sigma_t^{-1/2}}_{\op} \norm{\Sigma_t^{1/2}\varphi_t(\hat\Sigma)\Sigma_t^{1/2}}_{\op}\leq 2 C_1.
    \end{align*}Next, it follows from the sub-gaussian property of the design vector $X$ from Assumption~\ref{assumption:main_upper} and Lemma~\ref{lem:sum_squares} that, for some absolute constant $c>0$, with probability at least $1-\exp(-cN)$, 
    \begin{equation*}
        \frac{1}{N}\|\bX\bbeta_{J^c}^*\|_2^2 = \frac{1}{N}\sum_{i=1}^N \inr{X_i,  \bbeta_{J^c}^*}^2 \leq 2\|\Sigma_{J^c}^{1/2}\bbeta_{J^c}^*\|_2^2.
    \end{equation*}
    As a result, there exist an absolute constants $c>0$ such that  with probability at least $1-\exp(-c|J|)-\bP[\Omega_t^c]$,
    \begin{equation}\label{eq:first_ineq_var}
         \|\Sigma_J^{1/2}\varphi_t(\hat\Sigma)\hat\Sigma\bbeta_{J^c}^*\|_2\leq 16 C_1 \|\Sigma_{J^c}^{1/2}\bbeta_{J^c}^*\|_2.
     \end{equation} 
\paragraph{Upper bound for \texorpdfstring{$(1/N)\|\Sigma_J^{1/2}\varphi_t(\hat\Sigma)\bX^\top\bxi\|_2^2$}{true variance of hat beta J}.} We first work conditionally on $\bX$ and consider the randomness coming only from the Gaussian vector $\bxi$ so that we can apply the Borel-TIS inequality (see Theorem~7.1 in \cite{dobrushin_isoperimetry_1996} or p.56-57 in \cite{ledoux_probability_1991}) in order to get: for almost all $\bX$, for all $t\geq1$ with probability at least $1-\exp(-t/2)$, $\|A\bxi\|_2\leq \sigma_\xi \sqrt{\Tr[AA^\top]} + \sigma_\xi\norm{A}_{op}\sqrt{t}$ where $A=\Sigma_J^{1/2}\varphi_t(\hat\Sigma)\bX^\top$. This implies that for almost all $\bX$, with probability at least $1-\exp(-|J|/2)$,
\begin{equation*}
    \frac{1}{N}\|\Sigma_J^{1/2}\varphi_t(\hat\Sigma)\bX^\top\bxi\|_2^2\leq 2\sigma_\xi^2 \Tr\left[\Sigma_J^{1/2}\varphi_t(\hat \Sigma)\hat \Sigma \varphi_t(\hat \Sigma)\Sigma_J^{1/2}\right] 
      + \frac{2\sigma_\xi^2}{N}\norm{\Sigma_J^{1/2}\varphi_t(\hat\Sigma)\bX^\top}_{op}^2|J|.
\end{equation*} For the weak variance term in the inequality above, we have $\hat \Sigma_t^{1/2}\varphi_t(\hat\Sigma)\hat\Sigma \varphi_t(\hat\Sigma) \hat \Sigma_t^{1/2} \preceq \oC{C_filter_1}^2 I_p$ and so by Lemma~\ref{lemma:operator_norm_Sigma_J} we get, on $\Omega_t$,
\begin{align*}
    &\frac{1}{N}\norm{\Sigma_J^{1/2}\varphi_t(\hat\Sigma)\bX^\top}_{\op}^2 \leq \norm{\Sigma_J^{1/2} \varphi_t(\hat\Sigma)\hat\Sigma \varphi_t(\hat\Sigma) \Sigma_J^{1/2}}_{\op}\\ 
    & \leq \norm{\Sigma_J^{1/2} \hat\Sigma_t^{-1/2}}_{\op}\norm{\hat \Sigma_t^{1/2}\varphi_t(\hat\Sigma)\hat\Sigma \varphi_t(\hat\Sigma) \hat \Sigma_t^{1/2}}_{\op} \norm{\hat \Sigma_t^{-1/2} \Sigma_J^{1/2}}_{\op}
     \leq 2 \oC{C_filter_1}^2. 
\end{align*}For the strong variance term in the inequality above, we use that $\varphi_t(\hat\Sigma)\hat\Sigma \varphi_t(\hat\Sigma) \preceq \oC{C_filter_1}^2 \hat\Sigma _t^{-1}$ and apply Lemma~\ref{lemma:operator_norm_Sigma_J} to get, on $\Omega_t$,
\begin{align*}
  & \Tr\left[\Sigma_J^{1/2}\varphi_t(\hat \Sigma)\hat \Sigma \varphi_t(\hat \Sigma)\Sigma_J^{1/2}\right]\leq \oC{C_filter_1}^2 \Tr\left[\Sigma_J^{1/2} \hat \Sigma_t^{-1} \Sigma_J^{1/2}\right] = \oC{C_filter_1}^2 \left(\Tr\left[\hat \Sigma_t^{-1} (\Sigma_J - \hat\Sigma_J) \right] + \Tr\left[\hat \Sigma_t^{-1}\hat\Sigma_J\right]\right)\\
   &\leq \oC{C_filter_1}^2 \left(\Tr\left[\hat \Sigma_t^{-1/2} (\Sigma_J - \hat\Sigma_J) \hat \Sigma_t^{-1/2}  \right] + |J|\right) \leq \oC{C_filter_1}^2\left(|J|\norm{\hat \Sigma_t^{-1/2} (\Sigma_J - \hat\Sigma_J) \hat \Sigma_t^{-1/2} }_{op}   + |J|\right)\leq 2 \oC{C_filter_1}^2 |J|.
\end{align*}As a consequence, we obtain that with probability at least $1-2\exp(-c|J|)-\bP[\Omega_t^c]$, $(1/N)\|\Sigma_J^{1/2}\varphi_t(\hat\Sigma)\bX^\top\bxi\|_2^2\lesssim \sigma_\xi^2 |J|$. 

Finally, gathering the last inequality together with \eqref{eq:first_ineq_var} we obtain that  with probability at least $1-2\exp(-c|J|) - \bP[\Omega_t^c]$,
\begin{equation*}
     \|\Sigma_J^{1/2}\hat\bbeta_J(\bX\bbeta_{J^c}^*+\bxi)\|_2 \lesssim \|\Sigma_{J^c}^{1/2}\bbeta_{J^c}^*\|_2 + \sigma_\xi \sqrt{\frac{|J|}{N}}.
\end{equation*}

\subsubsection{Conclusion on the estimation property of \texorpdfstring{$\hat\bbeta_J$}{hat beta J}}It follows from the results obtained in the previous sections, that with probability at least $1-2\exp(-c|J|)- \bP[\Omega_t^c]$,
\begin{equation}\label{eq:final_esti_part}
        \norm{\Sigma_J^{1/2}(\hat\bbeta_J - \bbeta^*_J)}_2\lesssim \sigma_\xi \sqrt{\frac{|J|}{N}}+ \norm{\Sigma_{J^c}^{1/2}\bbeta_{J^c}^*}_2 + \left(\square \log(t) + 1\right) \norm{\Sigma_J^{1/2}\psi_t(\Sigma)\bbeta^*}_2 + \frac{\square}{t}\norm{\Sigma_J^{-1/2}\bbeta^*_J}_2.
\end{equation}  
This result finishes our analysis of the statistical property of the estimation part $\hat\bbeta_J$ of the spectral method $\hat\bbeta$. The next step of the FSD method is to handle the 'noise absorption part' of $\hat\bbeta$.


\subsection{Control of the noise absorption part \texorpdfstring{$\hat\bbeta_{{J}^c}$.}{hat beta Jc}}\label{sec:noise_absorption}
In this section, we derive an upper bound for $\|\Sigma_{J^c}^{1/2}\hat\bbeta_{J^c}\|_2$, where $\hat\bbeta_{J^c} = P_{J^c}\hat\bbeta$. We recall that $\hat\bbeta = N^{-1}\varphi_t(\hat\Sigma)\bX^\top\vy$ and $\vy = \bX\bbeta^* + \bxi = \bX\bbeta^*_J + \bX\bbeta_{J^c}^*+\bxi$. Therefore, we have
\begin{equation}\label{eq:first_noise_absorption_part_risk}
   \|\Sigma_{J^c}^{1/2}\hat\bbeta_{J^c}\|_2\leq  \norm{\Sigma_{J^c}^{1/2}\varphi_t(\hat\Sigma)\hat\Sigma \bbeta^*_J }_2 + \norm{\Sigma_{J^c}^{1/2}\varphi_t(\hat\Sigma)\hat\Sigma \bbeta^*_{J^c} }_2+ \norm{\Sigma_{J^c}^{1/2}\varphi_t(\hat\Sigma)[N^{-1}\bX^\top]\bxi }_2.
\end{equation}Next, we prove high probability upper bounds on the three terms in the sum above.

\paragraph{Upper bound for \texorpdfstring{$\norm{\Sigma_{J^c}^{1/2}\varphi_t(\hat\Sigma)\hat\Sigma \bbeta^*_J }_2$}{bias caused by beta J *}.}
By definition of the residual function, we have $\varphi_t(\hat \Sigma)\hat \Sigma = I_p - \psi_t(\hat\Sigma)$ and so  $\Sigma_{J^c}^{1/2}\varphi_t(\hat\Sigma)\hat\Sigma \bbeta^*_J = -\Sigma_{J^c}^{1/2}\psi_t(\hat\Sigma)\bbeta^*_J$ where we have used the fact that $\Sigma_{J^c}^{1/2}\bbeta^*_J=0$.  Next, we take the $\ell_2^p$-norm on both sides and use the fact that $\Sigma_{J^c}^{1/2}\psi_t(\Sigma)\bbeta^*_J=0$ to get 
\begin{equation*}
    \norm{\Sigma_{J^c}^{1/2}\varphi_t(\hat\Sigma)\hat\Sigma \bbeta^*_J }_2 = \|\Sigma_{J^c}^{1/2}\big(\psi_t(\hat\Sigma)-\psi_t(\Sigma)\big)\bbeta^*_J\|_2.
\end{equation*}
Next, on $\Omega_t$, we can apply the residual theorem to $\psi_t(\hat\Sigma)$ and $\psi_t(\Sigma)$ and get a result similar to the one of \eqref{eq:residual_theorem} where $\varphi_t$ is replaced by $\psi_t$. Thanks to this result we get (on  $\Omega_t$)
\begin{align*}
    &\norm{\Sigma_{J^c}^{1/2}\big(\psi_t(\hat\Sigma)-\psi_t(\Sigma)\big)\bbeta^*_J}_2 = \norm{\Sigma_{J^c}^{\frac{1}{2}}\Sigma_t^{-\frac{1}{2}}\Sigma_t^{\frac{1}{2}} \big(\psi_t(\hat\Sigma)-\psi_t(\Sigma)\big)\bbeta^*_J }_2\\
    &\leq \norm{\Sigma_{J^c}^{\frac{1}{2}}\Sigma_t^{-\frac{1}{2}}}_{\op} \norm{\Sigma_t^{\frac{1}{2}} \big(\psi_t(\hat\Sigma)-\psi_t(\Sigma)\big)\bbeta^*_J}_2\\
    &\leq \sqrt{\frac{b}{1+b}} \norm{\oint_{\cC_t} \Sigma_t^{\frac{1}{2}} (\hat\Sigma - zI_p)^{-1}(\hat\Sigma-\Sigma)(\Sigma-zI_p)^{-1} \bbeta^*_J\psi_t(z)dz }_2\\
    &\leq \sqrt{\frac{b}{1+b}}\oint_{\cC_t} \norm{\Sigma_t^{\frac{1}{2}}(\hat\Sigma - zI_p)^{-1}\Sigma_t^{\frac{1}{2}}}_{\op}\norm{\Sigma_t^{-\frac{1}{2}}(\hat\Sigma-\Sigma)\Sigma_t^{-\frac{1}{2}}}_{\op} \norm{\Sigma_t^{\frac{1}{2}}(\Sigma-zI_p)^{-1}\Sigma_t^{\frac{1}{2}}}_{\op}\norm{\Sigma_t^{-\frac{1}{2}}\bbeta^*_J}_2 |\psi_t(z)dz|
\end{align*}and so, on $\Omega_t$, by applying Lemma~\ref{lemma:spectral_calculus_li} we obtain
\begin{align}\label{eq:result_upper_bias_1_in_V_Jc}
        \norm{\Sigma_{J^c}^{1/2}\varphi_t(\hat\Sigma)\hat\Sigma \bbeta^*_J }_2 =\norm{\Sigma_{J^c}^{1/2}\big(\psi_t(\hat\Sigma)-\psi_t(\Sigma)\big)\bbeta^*_J}_2\lesssim \frac{\square}{t}\norm{\Sigma_J^{-\frac{1}{2}}\bbeta^*_J}_2.
\end{align}

\paragraph{Upper bound for \texorpdfstring{$\norm{\Sigma_{J^c}^{1/2}\varphi_t(\hat\Sigma)\hat\Sigma \bbeta^*_{J^c} }_2$}{bias caused by beta J c*}.} It follows from the 'upper side of Dvoretsky-Milman' theorem (see for instance Section~2.2.0.3 in \cite{gavrilopoulos_geometrical_2025}) that under Assumption~\ref{assumption:main_upper}, there are absolute constants $C,c>0$ such that with probability at least $1-\exp(-cN)$,
\begin{equation}\label{eq:noise_part_DM}
    \bP\left(\norm{\Sigma_{J^c}^{1/2}\bX^\top}_{\op}\leq C\left(\sqrt{\Tr(\Sigma_{J^c}^2)}+\sqrt{N}\norm{\Sigma_{J^c}}_{\op}\right)\right)\geq 1-\exp(-cN).
\end{equation}
Moreover, we have $\|\bX\bbeta_{J^c}^*\|_2\leq C\sqrt{N}\|\Sigma_{J^c}^{1/2}\bbeta_{J^c}^*\|_2$ with probability at least $1-\exp(-cN)$. Next, we observe that thanks to Assumption~\ref{ass:filter_fct}, $\varphi_t(x)\leq C_1(x+t^{-1})^{-1}\leq C_1 t$ so that we have $\varphi_t(\hat \Sigma)\leq C_1 t I_p$ and (since $\hat \Sigma$ and $I_p$ commute) for all $x\in\bR^p, \norm{\varphi_t(\hat \Sigma)x}_2\leq C_1 t \norm{x}_2$. It follows that with probability at least $1-2\exp(-cN)$,
\begin{align}\label{eq:result_upper_bias_2_in_V_Jc}
    \begin{aligned}
        \norm{\Sigma_{J^c}^{1/2}\varphi_t(\hat\Sigma)\hat\Sigma \bbeta^*_{J^c} }_2 = &\frac{1}{N}\norm{\Sigma_{J^c}^{1/2}\bX^\top\varphi_t(\hat \Sigma)\bX\bbeta_{J^c}^*}_2 \leq C C_1\frac{\sqrt{\Tr(\Sigma_{J^c}^2)}+\sqrt{N}\norm{\Sigma_{J^c}}_{\op}}{\sqrt{N}t^{-1}}\norm{\Sigma_{J^c}^{1/2}\bbeta_{J^c}^*}_2.
    \end{aligned}
\end{align}Finally, it follows from the definition of $k^*$ that $\sigma_{k^*+1} = \norm{\Sigma_{J^c}}_{\op}\leq b t^{-1}$ and from the sample complexity assumption (i.e. $\square^2 N\gtrsim \Tr\left[\Sigma(\Sigma+t^{-1}I_p)^{-1}\right]$) - see the discussion below \eqref{eq:sample_complexity_ass_imply} - that $\square^2 N \gtrsim t \Tr[\Sigma_{J^c}]$ so that
\begin{equation}\label{eq:noise_part_sample_comp_imply}
    \frac{\sqrt{\Tr(\Sigma_{J^c}^2)}+\sqrt{N}\norm{\Sigma_{J^c}}_{\op}}{\sqrt{N}t^{-1}}\leq \sqrt{\frac{\norm{\Sigma_{J^c}}_{\op}}{t^{-1}}} \sqrt{\frac{\Tr(\Sigma_{J^c})}{N t^{-1}}} + \frac{\norm{\Sigma_{J^c}}_{\op}}{t^{-1}}\leq \sqrt{b\square} + b\leq 2b
\end{equation}as long as $\square\leq b$. We conclude that with probability at least $1-2\exp(-cN)$,
\begin{equation}\label{eq:noise_part_2}
    \norm{\Sigma_{J^c}^{1/2}\varphi_t(\hat\Sigma)\hat\Sigma \bbeta^*_{J^c} }_2\leq C C_1 b \norm{\Sigma_{J^c}^{1/2}\bbeta_{J^c}^*}_2.
\end{equation}

\paragraph{Upper bound for \texorpdfstring{$\norm{\Sigma_{J^c}^{1/2}\varphi_t(\hat\Sigma)[N^{-1}\bX^\top]\bxi }_2$}{bias caused by bxi}.}
As in the previous section we first condition on $\bX$ and apply the Borell-TIS inequality: for almost all $\bX$, for all $r>0$, with probability at least $1-\exp(-r/2)$, $\norm{A\bxi}_2\leq \sigma_\xi \sqrt{\Tr[AA^\top]} + \sigma_\xi \norm{A}_{op}\sqrt{r}$ where $A = \Sigma_{J^c}^{1/2}\varphi_t(\hat\Sigma)[N^{-1}\bX^\top]$. Hence, we have with probabiltity at least $1-\exp(-|J|/2)$,
\begin{align}\label{eq:first_xi_noise_absorp}
   \nonumber \norm{\Sigma_{J^c}^{1/2}\varphi_t(\hat\Sigma)[N^{-1}\bX^\top]\bxi }_2  & \leq \sigma_\xi \sqrt{\frac{\Tr[\Sigma_{J^c}\hat \Sigma\varphi_t^2(\hat\Sigma)]}{N}} + \sigma_\xi \norm{\Sigma_{J^c}^{1/2} \hat\Sigma^{1/2} \varphi_t(\hat\Sigma) }_{op}\sqrt{\frac{|J|}{N}} \\
   &\leq \sigma_\xi\oC{C_filter_1}t \sqrt{\frac{\Tr[\Sigma_{J^c}\hat \Sigma]}{N}} + \sigma_\xi\oC{C_filter_1}t \norm{\Sigma_{J^c}^{1/2} \hat\Sigma^{1/2}}_{op}\sqrt{\frac{|J|}{N}}
\end{align}where in the last inequality we used that $\varphi_t(x)\leq \oC{C_filter_1} (x+t^{-1})^{-1}\leq \oC{C_filter_1}t$.  Next, it follows from Lemma~\ref{lem:sum_norm_square_psi_2} that there exists an absolute constant $c>0$ such that with probability at least $1-\exp(-cN)$,
\begin{align*}
    \Tr[\Sigma_{J^c}\hat \Sigma] = \frac{1}{N}\Tr(\bX\Sigma_{J^c}\bX^\top) = \frac{1}{N}\sum_{i=1}^N\norm{\Sigma_{J^c}^{1/2}X_i}_2^2 \leq 2\Tr(\Sigma_{J^c}^2).
\end{align*}Then, it follows from \eqref{eq:noise_part_DM} that there are absolute constants $C, c>0$ such that with probability at least $1-\exp(-cN)$,
\begin{align}\label{eq:DM}
    \norm{\Sigma_{J^c}^{1/2} \hat\Sigma^{1/2}}_{op} = \frac{1}{\sqrt{N}}\norm{\Sigma_{J^c}^{1/2}\bX^\top }_{op}\leq C\left(\sqrt{\frac{\Tr(\Sigma_{J^c}^2)}{N}}+\norm{\Sigma_{J^c}}_{\op}\right).
\end{align}Finally, collecting the last two results together with \eqref{eq:noise_part_sample_comp_imply} in the Borell-TIS inequality above, we get that with probability at least $1-2\exp(-c|J|)$,
\begin{align}\label{eq:final_xi_noise_absorp}
   \norm{\Sigma_{J^c}^{1/2}\varphi_t(\hat\Sigma)[N^{-1}\bX^\top]\bxi }_2 \lesssim \sigma_\xi t \sqrt{\frac{\Tr(\Sigma_{J^c}^2)}{N}} + \sigma_\xi t \left(\sqrt{\frac{\Tr(\Sigma_{J^c}^2)}{N}}+\norm{\Sigma_{J^c}}_{\op}\right) \sqrt{\frac{|J|}{N}} \lesssim \sigma_\xi \sqrt{\frac{|J|}{N}} + \sigma_\xi t \sqrt{\frac{\Tr(\Sigma_{J^c}^2)}{N}}. 
\end{align}


\paragraph{Concluding on the noise absorption property.} Combining \eqref{eq:result_upper_bias_1_in_V_Jc}, \eqref{eq:noise_part_2} and \eqref{eq:final_xi_noise_absorp}, we obtain that with probability at least $1-2\exp(-c|J|) - \bP[\Omega_t^c]$,
\begin{equation}\label{eq:final_noise_absorption}
    \norm{\Sigma_{J^c}^{1/2}\hat\bbeta_{J^c}}_2 \lesssim \frac{\square}{t}\norm{\Sigma_J^{-\frac{1}{2}}\bbeta^*_J}_2 + \norm{\Sigma_{J^c}^{1/2}\bbeta_{J^c}^*}_2 +  \sigma_\xi \sqrt{\frac{|J|}{N}} + \sigma_\xi t \sqrt{\frac{\Tr(\Sigma_{J^c}^2)}{N}}   
\end{equation}

\subsection{End of the proof of the upper bound from Theorem~\ref{theo:main}.}
Going back to the original risk decomposition from the FSD method in \eqref{eq:risk_decomposition_origin_FSD} and collecting both results on the estimation part and the noise absorption part from  \eqref{eq:final_esti_part} and \eqref{eq:final_noise_absorption}, we obtain that with probability at least $1-\exp(-c|J|) - \bP[\Omega_t^c]$,
\begin{align*}
&\norm{\Sigma^{1/2}\left(\hat\bbeta - \bbeta^*\right)}_2  \leq  \norm{\Sigma_J^{1/2}\left(\hat\bbeta_J - \bbeta^*_J\right)}_2 + \norm{\Sigma_{J^c}^{1/2}\hat\bbeta_{J^c}}_2 +\norm{ \Sigma_{J^c}^{1/2}\bbeta_{J^c}^*}_2\\
    \lesssim  & \left( \sigma_\xi\sqrt{\frac{|J|}{N}} + \norm{\Sigma_{J^c}^{1/2}\bbeta_{J^c}^*}_2 + \left(\square \log(et) + 1\right) \norm{\Sigma_J^{1/2}\psi_t(\Sigma)\bbeta^*}_2 + \frac{\square}{t}\norm{\Sigma_J^{-\frac{1}{2}}\bbeta^*_J}_2 \right) \\  
    & +  \left(\frac{\square}{t}\norm{\Sigma_J^{-\frac{1}{2}}\bbeta^*_J}_2 + \sigma_\xi \sqrt{\frac{|J|}{N}} + \norm{\Sigma_{J^c}^{1/2}\bbeta_{J^c}^*}_2 + \sigma_\xi t \sqrt{\frac{\Tr(\Sigma_{J^c}^2)}{N}}\right) + \norm{\Sigma_{J^c}^{1/2}\bbeta_{J^c}^*}_2\\
    &\lesssim \sigma_\xi\sqrt{\frac{|J|}{N}} + \norm{\Sigma_{J^c}^{1/2}\bbeta_{J^c}^*}_2 + \left(\square \log(et) + 1\right) \norm{\Sigma_J^{1/2}\psi_t(\Sigma)\bbeta^*}_2 + \sigma_\xi t \sqrt{\frac{\Tr(\Sigma_{J^c}^2)}{N}}  + \frac{\square}{t}\norm{\Sigma_J^{-\frac{1}{2}}\bbeta^*_J}_2
\end{align*}and the result follows if one takes $\square\lesssim \log^{-1}(et)$.


\section{Proof of the lower bound result from Theorem~\ref{theo:main_LB}}\label{sec:optimal_FSD}
In this section, we prove the lower bound result from Theorem~\ref{theo:main_LB}. We first work conditionally to $\bX$ so that we can use the concentration inequality of a Lipschitz function of the Gaussian vector $\bxi$ (see Eq.(2.35) in \cite{ledoux_concentration_2005} or Theorem~5.2.2 in \cite{vershynin_high-dimensional_2018}): for almost all $\bX$, for all $r>0$, with probability at least $1-\exp(-r)$, $\phi(\bxi) \geq  \bE_\bxi \phi(\bxi) -  \sigma_\xi\norm{\phi}_{Lip}\sqrt{2r}$ where $\phi(\bxi) = \norm{\Sigma^{1/2}(\hat \bbeta(\bX \bbeta^* + \bxi) - \bbeta^*)}_2$ and $\norm{\phi}_{Lip}$ is the Lipschitz constant of $\phi$ with respect to the Euclidean norm. Moreover, thanks to the concentration of Lipschitz functions of Gaussian vectors recalled above we have: for almost all $\bX$,
\begin{align*}
   \bE_\bxi \phi(\bxi)^2 - [\bE_\bxi \phi(\bxi)]^2 = \bE_\bxi\left[\left(\phi(\bxi) - \bE_\bxi \phi(\bxi)\right)^2\right] = \int_0^\infty \bP_\bxi\left[|\phi(\bxi) - \bE_\bxi \phi(\bxi)|\geq \sqrt{r}\right]dr\leq 2 \sigma_\xi^2\norm{\phi}_{Lip}^2.
\end{align*}As a consequence, $[\bE_\bxi \phi(\bxi)]^2 \geq \bE_\bxi [\phi(\bxi)^2] - 2 \sigma_\xi^2\norm{\phi}_{Lip}^2$ and so, for almost all $\bX$, with $\bP_\bxi$-probability at least $1-\exp(-r)$, 
\begin{equation}\label{eq:lower_bound_borell}
    \phi(\bxi) \geq  \bE_\bxi \phi(\bxi) - \sigma_\xi\norm{\phi}_{Lip}\sqrt{2r}\geq \sqrt{\frac{\bE_\bxi[\phi(\bxi)^2]}{2}} - \sigma_\xi\norm{\phi}_{Lip}\sqrt{2r}
\end{equation}when $\bE_\bxi \phi(\bxi)^2 \geq 4 \sigma_\xi^2\norm{\phi}_{Lip}^2$. We note that \eqref{eq:lower_bound_borell} also holds when  $\bE_\bxi \phi(\bxi)^2 \leq 4 \sigma_\xi^2\norm{\phi}_{Lip}^2$ as long as $r\geq 4 \sqrt{2}$ since $\phi(\bxi)\geq0$ a.s.. As a consequence, we (always) have for all $r\geq 4 \sqrt{2}$, 
\begin{equation*}
   \phi(\bxi) \geq  \sqrt{\frac{\bE_\bxi[\phi(\bxi)^2]}{2}} - \sigma_\xi\norm{\phi}_{Lip}\sqrt{2r}. 
\end{equation*}Next, thanks to the linearity of the estimator $\hat\bbeta$ we have for all $\xi_1, \xi_2 \in\bR^p, |\phi(\xi_1) - \phi(\xi_2)|\leq \norm{\Sigma^{1/2}\hat\bbeta(\xi_1-\xi_2)}_2$ and so $\norm{\phi}_{Lip}\leq \norm{A}_{op}$ where $A=\Sigma^{1/2}\varphi_t(\hat\Sigma) [N^{-1}\bX^\top]$ and 
\begin{align*}
 \bE_\bxi[\phi(\bxi)^2] =    \bE_\bxi\norm{\Sigma^{\frac{1}{2}}(\hat\bbeta-\bbeta^*)}_2^2 = \norm{\Sigma^{\frac{1}{2}}(\hat\bbeta(\bX\bbeta^*)-\bbeta^*)}_2^2 + \bE_\bxi\norm{\Sigma^{\frac{1}{2}}\hat\bbeta(\bxi)}_2^2 = \norm{\Sigma^{\frac{1}{2}}(\hat\bbeta(\bX\bbeta^*)-\bbeta^*)}_2^2 + \sigma_\xi \Tr[AA^\top].
\end{align*}Finally, we have for almost all $\bX$ and all $r\geq 4 \sqrt{2}$, with probability at least $1-\exp(-r)$,
\begin{equation}\label{eq:first_lower_bound}
  \norm{\Sigma^{\frac{1}{2}}(\hat\bbeta-\bbeta^*)}_2\geq \frac{1}{\sqrt{2}}\norm{\Sigma^{\frac{1}{2}}(\hat\bbeta(\bX\bbeta^*)-\bbeta^*)}_2 + \frac{\sigma_\xi}{\sqrt{2}} \sqrt{\frac{\Tr[\Sigma \varphi^2_t(\hat\Sigma)\hat\Sigma]}{N}} -  \sigma_\xi \norm{\Sigma^{1/2} \varphi_t(\hat\Sigma)\frac{\bX^\top}{\sqrt{N}} }_{op} \sqrt{\frac{2r}{N}}.
\end{equation}
In the next two sections, we obtain lower bounds on the three main terms appearing in the right hand side of \eqref{eq:first_lower_bound}.

\subsection{A lower bound for the bias term \texorpdfstring{$\norm{\Sigma^{\frac{1}{2}}(\hat\bbeta(\bX\bbeta^*)-\bbeta^*)}_2$}{bias}.}\label{sec:proof_lower_bound_bias}

As before, we decompose the feature space as $\bR^p= V_{J}\oplus^\perp V_{{J}^c}$ where $J=J_*$ is the optimal decomposition, so that the bias term can be decomposed as
\begin{align*}
    &\norm{\Sigma^{\frac{1}{2}}(\hat\bbeta(\bX\bbeta^*)-\bbeta^*)}_2^2 = \norm{\Sigma_{J}^{\frac{1}{2}}(\hat\bbeta(\bX\bbeta^*)-\bbeta^*)}_2^2 + \norm{\Sigma_{{J}^c}^{\frac{1}{2}}(\hat\bbeta(\bX\bbeta^*)-\bbeta^*)}_2^2.
\end{align*}

\paragraph{A lower bound for the bias term on $V_{J}$.} In Section~\ref{ssub:risk_decomposition_of_the_estimation_part_hat_bbeta_j}, we introduced $\tilde\bbeta = \varphi_t(\Sigma)\Sigma\bbeta^*$ and proved in \eqref{eq:approximation_term_esti_part} that
\begin{equation}\tag{\ref{eq:approximation_term_esti_part}}
    \norm{\Sigma_J^{1/2}(\tilde\bbeta_J - \bbeta^*_J)}_2= \norm{\Sigma_J^{1/2}\psi_t(\Sigma)\bbeta^*_J}_2
\end{equation} and in \eqref{eq:expectation_bias_plug_in_result_v0} that, for some absolute constant $C>0$, on $\Omega_t$,
\begin{equation}\tag{\ref{eq:expectation_bias_plug_in_result_v0}}
    \norm{\Sigma_J^{1/2}(\hat\bbeta_J(\bX\bbeta^*_J) - \tilde\bbeta_J)}_2 \leq C\square\left(\log(et)\norm{\Sigma_J^{1/2} \psi_t(\Sigma)\bbeta^*_J}_2 + t^{-1}\norm{\Sigma_J^{1/2}\varphi_t(\Sigma)\bbeta^*_J}_2\right).
\end{equation}We also recall
\begin{equation}\tag{\ref{eq:first_ineq_var}}
         \norm{\Sigma_J^{\frac{1}{2}}\hat\bbeta_J(\bX\bbeta_{J^c}^*)}_2 = \|\Sigma_J^{1/2}\varphi_t(\hat\Sigma)\hat\Sigma\bbeta_{J^c}^*\|_2\leq 16 C_1 \|\Sigma_{J^c}^{1/2}\bbeta_{J^c}^*\|_2.
     \end{equation}
By triangular inequality,
\begin{align*}
    \norm{\Sigma_J^{\frac{1}{2}}\left( \hat\bbeta_J(\bX\bbeta_J^*) - \bbeta_J^* \right)}_2 &\geq \norm{\Sigma_J^{\frac{1}{2}}(\tilde\bbeta_J - \bbeta_J^*)}_2 - \norm{\Sigma_J^{\frac{1}{2}}\left( \hat\bbeta_J(\bX\bbeta_J^*) - \tilde\bbeta_J \right)}_2\\
    &= \norm{\Sigma_J^{\frac{1}{2}}\psi_t(\Sigma)\bbeta_J^*}_2 - \norm{\Sigma_J^{\frac{1}{2}}\left( \hat\bbeta_J(\bX\bbeta_J^*) - \tilde\bbeta_J \right)}_2.
\end{align*}By \eqref{eq:expectation_bias_plug_in_result_v0},
\begin{align*}
    \norm{\Sigma_J^{\frac{1}{2}}\left( \hat\bbeta_J(\bX\bbeta_J^*) - \bbeta_J^* \right)}_2 &\geq \left(1-C\square\log(et)\right)\norm{\Sigma_J^{\frac{1}{2}}\psi_t(\Sigma)\bbeta_J^*}_2 - C\frac{\square}{t}\norm{\Sigma_J^{\frac{1}{2}}\varphi_t(\Sigma)\bbeta_J^*}_2.
\end{align*}
Next, it follows from Assumption~\ref{ass:filter_fct} that $\varphi_t(\Sigma)\preceq C_1\Sigma_t^{-1}$ and so $\norm{\Sigma_J^{1/2}\varphi_t(\Sigma)\bbeta^*_J}_2 \leq C_1\norm{\Sigma_J^{-1/2}\bbeta^*_J}_2$.  As a consequence, as long as $\square\log(e^2t)\lesssim 1$, the following lower bound holds on $\Omega_t$:
\begin{align*}
    \|\Sigma_{J}^{\frac{1}{2}}(\hat\bbeta(\bX\bbeta_J^*)-\bbeta_J^*)\|_2 &\geq \frac12 \|\Sigma_{J}^{\frac{1}{2}}\psi_t(\Sigma)\bbeta_J^*\|_2-\frac{CC_1\square }{t} \norm{\Sigma_J^{-1/2}\bbeta^*_J}_2.
\end{align*}Now, $\hat\bbeta_J(\bX\bbeta^*) - \bbeta_J^* = \hat\bbeta_J(\bX\bbeta_J^*)-\bbeta_J^* + \hat\bbeta_J(\bX\bbeta_{J^c}^*)$, hence by \eqref{eq:first_ineq_var}, together with the triangular inequality,
\begin{align}\label{eq:bias_lower_1_result}
    \begin{aligned}
        &\norm{\Sigma_J^{\frac{1}{2}}\left(\hat\bbeta_J(\bX\bbeta^*) - \bbeta_J^*\right)}_2 \geq \left( \frac{1}{2}\norm{\Sigma_J^{\frac{1}{2}}\psi_t(\Sigma)\bbeta_J^*}_2 - 16C_1\norm{\Sigma_{J^c}^{\frac{1}{2}}\bbeta_{J^c}^*}_2 - CC_1\frac{\square}{t}\norm{\Sigma_J^{-\frac{1}{2}}\bbeta_J^*}_2\right)_+ \\
        &\geq \left( \frac{1}{2}\norm{\Sigma_J^{\frac{1}{2}}\psi_t(\Sigma)\bbeta_J^*}_2 - 16C_1\norm{\Sigma_{J^c}^{\frac{1}{2}}\bbeta_{J^c}^*}_2 \right)_+ - CC_1\frac{\square}{t}\norm{\Sigma_J^{-\frac{1}{2}}\bbeta_J^*}_2,
    \end{aligned}
\end{align}where we have used the fact that $(r-s)_+\geq r_+-s$ for any $r\in\bR$ and $s\geq 0$.
\paragraph{A lower bound for the bias on $V_{{J}^c}$.} 
We recall
\begin{align}\tag{\ref{eq:result_upper_bias_1_in_V_Jc}}
        \norm{\Sigma_{J^c}^{\frac{1}{2}}\hat\bbeta_{J^c}(\bX\bbeta_J^*)}_2 = \norm{\Sigma_{J^c}^{1/2}\varphi_t(\hat\Sigma)\hat\Sigma \bbeta^*_J }_2 \lesssim \frac{\square}{t}\norm{\Sigma_J^{-\frac{1}{2}}\bbeta^*_J}_2,
\end{align}as well as 
\begin{equation}\tag{\ref{eq:noise_part_2}}
    \norm{\Sigma_{J^c}^{\frac{1}{2}}\hat\bbeta_{J^c}(\bX\bbeta_{J^c}^*)}_2 = \norm{\Sigma_{J^c}^{1/2}\varphi_t(\hat\Sigma)\hat\Sigma \bbeta^*_{J^c} }_2\leq C C_1 b \norm{\Sigma_{J^c}^{1/2}\bbeta_{J^c}^*}_2.
\end{equation}Take $b$ small enough such that $CC_1b\leq \frac{1}{4}$.
We have
    \begin{align*}
        &\norm{\Sigma_{{J}^c}^{\frac{1}{2}}(\hat\bbeta(\bX\bbeta^*)-\bbeta^*)}_2 \geq \norm{\Sigma_{{J}^c}^{\frac{1}{2}}\bbeta_{{J}^c}^*}_2 - \norm{\Sigma_{{J}^c}^{\frac{1}{2}}\hat\bbeta(\bX\bbeta^*)}_2
    \end{align*}and using that $\hat\bbeta(\bX\bbeta^*) = \hat\bbeta(\bX\bbeta^*_J) + \hat\bbeta(\bX\bbeta_{J^c}^*) = \varphi_t(\hat\Sigma)\hat\Sigma \bbeta^*_J + \varphi_t(\hat\Sigma)\hat\Sigma \bbeta^*_{J^c}$ we get
    \begin{equation*}
   \|\Sigma_{J^c}^{1/2}\hat\bbeta(\bX\bbeta^*)\|_2\leq  \norm{\Sigma_{J^c}^{1/2}\varphi_t(\hat\Sigma)\hat\Sigma \bbeta^*_J }_2 + \norm{\Sigma_{J^c}^{1/2}\varphi_t(\hat\Sigma)\hat\Sigma \bbeta^*_{J^c} }_2.
\end{equation*}Therefore,
\begin{align}\label{eq:bias_lower_2_result}
    \begin{aligned}
        \norm{\Sigma_{J^c}^{\frac{1}{2}}\left( \hat\bbeta_{J^c}(\bX\bbeta^*) - \bbeta_{J^c}^* \right)}_2 &\geq \norm{\Sigma_{J^c}^{\frac{1}{2}}\bbeta_{J^c}^*}_2 - \norm{\Sigma_{J^c}^{\frac{1}{2}}\hat\bbeta_{J^c}(\bX\bbeta_J^*)}_2 - \norm{\Sigma_{J^c}^{\frac{1}{2}}\hat\bbeta_{J^c}(\bX\bbeta_{J^c}^*)}_2\\
        &\geq \frac{1}{2}\norm{\Sigma_{J^c}^{\frac{1}{2}}\bbeta_{J^c}^*}_2 - C\frac{\square}{t}\|\Sigma_J^{-\frac{1}{2}}\bbeta_J^*\|_2.
    \end{aligned}
\end{align}

\paragraph{End of the proof of the lower bound for the bias term.} Since $V_J\perp V_{J^c}$, we obtain
\begin{align*}
    \norm{\Sigma^{\frac{1}{2}}(\hat\bbeta(\bX\bbeta^*) - \bbeta^*)}_2 \geq \frac{1}{\sqrt{2}}\norm{\Sigma_J^{\frac{1}{2}}(\hat\bbeta_J(\bX\bbeta^*) - \bbeta_J^*)}_2 + \frac{1}{\sqrt{2}}\norm{\Sigma_{J^c}^{\frac{1}{2}}(\hat\bbeta_{J^c}(\bX\bbeta^*) - \bbeta_{J^c}^*)}_2.
\end{align*} Combining \eqref{eq:bias_lower_1_result} and \eqref{eq:bias_lower_2_result}, there exists an absolute constant $C$ such that
\begin{align*}
    &\norm{\Sigma^{\frac{1}{2}}(\hat\bbeta(\bX\bbeta^*) - \bbeta^*)}_2\geq \frac{1}{\sqrt{2}}\left( \frac{1}{2}\norm{\Sigma_J^{\frac{1}{2}}\psi_t(\Sigma)\bbeta_J^*}_2 - 16 C_1\norm{\Sigma_{J^c}^{\frac{1}{2}}\bbeta_{J^c}^*}_2 \right)_+ + \frac{1}{2\sqrt{2}}\norm{\Sigma_{J^c}^{\frac{1}{2}}\bbeta_{J^c}^*}_2 - C\frac{\square}{t}\norm{\Sigma_J^{-\frac{1}{2}}\bbeta_J^*}_2.
\end{align*}Now we apply Lemma~\ref{lemma:positive_part} to
\begin{align*}
    x = \norm{\Sigma_{J^c}^{\frac{1}{2}}\bbeta_{J^c}^*}_2,\, y = \norm{\Sigma_J^{\frac{1}{2}}\psi_t(\Sigma)\bbeta_J^*}_2,\, A = \frac{1}{2\sqrt{2}},\, B = \frac{1}{2\sqrt{2}},\mbox{ and } D = \frac{16C_1}{\sqrt{2}},
\end{align*}then there holds
\begin{align*}
    &\frac{1}{\sqrt{2}}\left( \frac{1}{2}\norm{\Sigma_J^{\frac{1}{2}}\psi_t(\Sigma)\bbeta_J^*}_2 - 16 C_1\norm{\Sigma_{J^c}^{\frac{1}{2}}\bbeta_{J^c}^*}_2 \right)_+ + \frac{1}{2\sqrt{2}}\norm{\Sigma_{J^c}^{\frac{1}{2}}\bbeta_{J^c}^*}_2\\
    &\geq \frac{\sqrt{2}}{8+128C_1}\left( \norm{\Sigma_{J^c}^{\frac{1}{2}}\bbeta_{J^c}^*}_2 + \norm{\Sigma_J^{\frac{1}{2}}\psi_t(\Sigma)\bbeta_J^*}_2 \right).
\end{align*}
As a result,
\begin{align}\label{eq:LB_final_bias_term}
    \norm{\Sigma^{\frac{1}{2}}(\hat\bbeta(\bX\bbeta^*)-\bbeta^*)}_2 \geq \frac{\sqrt{2}}{8+128C_1}\|\Sigma_{J}^{\frac{1}{2}}\psi_t(\Sigma)\bbeta^*_J\|_2 + \frac{\sqrt{2}}{8+128C_1}\norm{\Sigma_{J^c}^{1/2}\bbeta_{J^c}^*}_2 - \frac{C \square}{t} \norm{\Sigma_J^{-\frac{1}{2}}\bbeta^*_J}_2.
\end{align}

\subsection{Lower bound for the conditional variance term \texorpdfstring{$\bE_\bxi\|\Sigma^{1/2}\hat\bbeta(\bxi)\|_2^2$}{variance}.}\label{sec:variance_hat_bbeta}

In this section, we obtain a lower bound on the conditional (with respect to $\bX$) variance of $\hat{\bbeta}$: $\bE_\bxi\|\Sigma^{1/2}\hat\bbeta(\bxi)\|_2^2$. It follows from Assumption~\ref{ass:filter_fct} that for all $t\geq 1$ and $x\in[0,8]$ , we have
\begin{align}\label{eq:equivalence_varphi}
   \varphi_t(x) \geq   \frac{\oc{c_filter_1}}{x+t^{-1}}:= \oc{c_filter_1}\varphi_t^{(\mathrm{Ridge})}(x)
\end{align}where we recall (see \eqref{eq:def_ridge}) that $\varphi_t^{(\mathrm{Ridge})}(x)=(x+t^{-1})^{-1}$ is the filter function of ridge regression with regularization parameter $t^{-1}$.

\begin{Lemma}\label{lemma:universality_variance}
    Grant Assumption~\ref{assumption:main_upper} and assume that $X$ has independent and centered coordinates with respect to $\{\ve_1,\cdots,\ve_p\}$.
    Let $\hat\bbeta$ be a spectral algorithm defined in Definition~\ref{def:spec_algo} with filter function $\varphi_t$ satisfying \eqref{eq:equivalence_varphi}. Then,  there exists absolute constants $c,\nc\label{c_variance}>0$ such that with probability at least $1-c\exp(-N/c)-\bP[\Omega_t^c]$,
    \begin{align*}
       \sigma_\xi^2 \frac{\Tr[\Sigma \varphi^2_t(\hat\Sigma)\hat\Sigma]}{N} =  \bE_\bxi\|\Sigma^{1/2}\hat\bbeta(\bxi)\|_2^2 \geq \oc{c_variance} \oc{c_filter_1} \sigma_\xi^2  \bigg( \frac{|J|}{N} + t^2 \frac{\Tr(\Sigma_{{J}^c}^2)}{N}\bigg).
    \end{align*}
\end{Lemma}

\begin{proof} 

Let $\sum_{j=1}^p\hat\sigma_i^{\frac{1}{2}}\hat\vu_i\otimes\hat\ve_i$ be the singular value decomposition of $\frac{1}{\sqrt{N}}\bX$, where $\hat\sigma_j=0$ if $j>N$, $\{\hat\vu_i\}_{i=1}^N$ is an orthonormal basis of $\bR^N$ and $\{\hat\ve_j\}_{j=1}^p$ is an orthonormal basis of $\bR^p$. It follows from  \eqref{eq:def_hat_bbeta} that
\begin{align*}
    \hat\bbeta(\bxi) = \frac{1}{N}\varphi(\hat\Sigma)\bX^\top\bxi = \frac{1}{\sqrt{N}}\varphi_t(\hat\Sigma) \sum_{i=1}^N \hat\ve_i\sqrt{\hat\sigma_i}\langle\hat\vu_i,\bxi\rangle = \frac{1}{\sqrt{N}}\sum_{i=1}^N \sqrt{\hat\sigma_i}\varphi_t(\hat\sigma_i)\langle\hat\vu_i,\bxi\rangle\hat\ve_i
\end{align*}and by taking $\|\Sigma^{1/2}\cdot\|_2^2$, we obtain
\begin{align*}
    & \norm{\Sigma^{1/2}\hat\bbeta(\bxi)}_2^2 = \frac{1}{N}\norm{\sum_{i=1}^N\sqrt{\hat\sigma_i}\varphi_t(\hat\sigma_i)\langle\hat\vu_i,\bxi\rangle\Sigma^{1/2}\hat\ve_i}_2^2 = \frac{1}{N}\sum_{i,j=1}^N \sqrt{\hat\sigma_i\hat\sigma_j}\varphi_t(\hat\sigma_i)\varphi_t(\hat\sigma_j)\langle\hat\vu_i,\bxi\rangle\langle\hat\vu_j,\bxi\rangle\langle\Sigma^{1/2}\hat\ve_i,\Sigma^{1/2}\hat\ve_j\rangle.
\end{align*}Taking expectation with respect to $\bxi$ and using that $\bE_\bxi[\langle\hat\vu_i,\bxi\rangle\langle\hat\vu_j,\bxi\rangle] = \sigma_\xi^2 \inr{\hat\vu_i,\hat\vu_j}=\sigma_\xi^2\1_{\{i=j\}}$, we obtain that for almost all $\bX$,
\begin{align*}
    \bE_\bxi\norm{\Sigma^{1/2}\hat\bbeta(\bxi)}_2^2 = \frac{\sigma_\xi^2}{N}\sum_{i=1}^N \hat\sigma_i\varphi_t^2(\hat\sigma_i)\norm{\Sigma^{1/2}\hat\ve_i}_2^2.
\end{align*}The latter result is actually true for any filter function. By applying it to the filter function from ridge regression and using \eqref{eq:equivalence_varphi}, we have on the event $\Omega_t$ (where we know, thanks to Lemma~\ref{lem:largest_sing_Hat_Sigma}, that the spectrum of $\hat\Sigma$ is in $[0,8]$ because $\sigma_1, t^{-1}\leq1$) that
\begin{align*}
    \bE_\bxi\norm{\Sigma^{1/2}\hat\bbeta(\bxi)}_2^2  \geq \oc{c_filter_1}^2\frac{\sigma_\xi^2}{N}\sum_{i=1}^N\hat\sigma_i\big(\varphi_t^{(\mathrm{Ridge})}(\hat\sigma_i)\big)^2\|\Sigma^{1/2}\hat\ve_i\|_2^2 = \oc{c_filter_1}^2\bE_\bxi\norm{\Sigma^{1/2}\hat\bbeta^{(Ridge)}(\bxi)}_2^2. 
\end{align*}

Finally, by \cite[Lemma~7 and Theorem~2]{tsigler_benign_2023}, there exists an absolute constant $0<\nc\label{c_lower_variance}<1$ such that with probability at least $1-c\exp(-N/c)$,
\begin{align*}
    & \bE_\bxi\norm{\Sigma^{1/2}\hat\bbeta^{(\mathrm{Ridge})}(\bxi)}_2^2 \geq \oc{c_lower_variance}\sigma_\xi^2 \bigg(\frac{|J|}{N} + \frac{N\Tr(\Sigma_{{J}^c}^2)}{(Nt^{-1} + \Tr(\Sigma_{{J}^c}))^2} \bigg).
\end{align*}Lemma~\ref{lemma:universality_variance} then follows since $\Tr(\Sigma_{{J}^c})\lesssim \square t^{-1}N\lesssim t^{-1}N$ thanks to the sampling complexity assumption (see the discussion below \eqref{eq:sample_complexity_ass_imply}).
\end{proof}

\subsection{An upper bound for the weak variance term and the conclusion. } 
\label{sub:an_upper_bound_for_the_weak_variance_term_}
In this section, we provide a high probability upper bound on the weak variance term coming from Borell's inequality in \eqref{eq:first_lower_bound} i.e. $ \sigma_\xi \norm{\Sigma^{1/2} \varphi_t(\hat\Sigma)(\bX^\top/ \sqrt{N})}_{op}$. It follows from \eqref{eq:compare_Sigma_t_and_Sigma_J_Jc} and Lemma~\ref{lemma:operator_norm_Sigma_J} that, on the event $\Omega_t$, we have
\begin{equation}\label{eq:LB_weak_var_term}
  \norm{\Sigma^{1/2} \varphi_t(\hat\Sigma)(\bX^\top/ \sqrt{N})}_{op} \leq \norm{\Sigma^{1/2} \Sigma_t^{-1/2}}_{\op}  \norm{\Sigma_t^{1/2} \hat \Sigma_t^{-1/2}}_{\op} \norm{\hat \Sigma_t\varphi_t(\hat\Sigma)^{2}\hat \Sigma}^{1/2}_{op}\lesssim 1 
\end{equation}where we used Assumption~\ref{ass:filter_fct} to get $\hat \Sigma_t\varphi_t(\hat\Sigma)^2\hat \Sigma\preceq C_1 \hat \Sigma_t\hat\Sigma_t^{-2}\hat \Sigma\preceq C_1 I_p$.

Finally, plugging \eqref{eq:LB_weak_var_term} and \eqref{eq:LB_final_bias_term} together with Lemma~\ref{lemma:universality_variance} in \eqref{eq:first_lower_bound}, we get that for all $r\geq 4 \sqrt{2}$, with probability at least $1-\exp(-r) - c\exp(-N/c)-\bP[\Omega_t^c]$, 
\begin{align*}
    \norm{\Sigma^{1/2}(\hat \bbeta-\bbeta^*)}_2 & \geq \left( \|\Sigma_{J}^{\frac{1}{2}}\psi_t(\Sigma)\bbeta^*_J\|_2 + \frac{1}{2}\norm{\Sigma_{J^c}^{1/2}\bbeta_{J^c}^*}_2 - \frac{C \square}{t} \norm{\Sigma_J^{-\frac{1}{2}}\bbeta^*_J}_2\right) +\oc{c_variance} \bigg(\sigma_\xi\sqrt{\frac{|J|}{N}} + \sigma_\xi t \sqrt{\frac{\Tr(\Sigma_{{J}^c}^2)}{N}}\bigg) - c_0\sigma_\bxi\sqrt{\frac{r}{N}}\\
    &\geq c r(V_J, V_{J^c}) - \frac{C \square}{t} \norm{\Sigma_J^{-\frac{1}{2}}\bbeta^*_J}_2 - c_0\sigma_\bxi\sqrt{\frac{ r}{N}}
\end{align*} as long as $b\lesssim1$. Finally, the result follows by taking $r\sim k^*$ in the inequality above.



\section{Proof of Theorem~\ref{theo:main_RKHS}}\label{sec:proof_main_RKHS}

\subsection{Upper bound of Theorem~\ref{theo:main_RKHS}}
The proof of the upper bound is almost identical to that in the sub-Gaussian setting.  We divide the proof into an estimation error and an approximation error. For clarity of exposition, we only point out the differences from the proof of Theorem~\ref{theo:main} in Section~\ref{sec:proof_main}. Here, for notational simplicity, we similarly write $J = J_*$ and $J^c=J_*^c$. Since, when $f^*$ does not necessarily belong to $\cH$, for instance when $f^*\in\overline{\cH}\backslash\cH$, the quantity $\bX f^*$ is not well-defined, we use $\boldsymbol{f^*}$ in place of $\bX f^*$ in the original proof, where $\boldsymbol{f^*} = (f^*(X_i))_{i=1}^N\in\bR^N$. Similarly, let $\vf_J^* = (f_J^*(X_i))_{i=1}^N$ and $\vf_{J^c}^* = (f_{J^c}^*(X_i))_{i=1}^N$.

If we expand $f^*$ with respect to $\{f_j\}_{j=1}^\infty$ as $f^* = \sum_{j\geq 1}a_j^* f_j$, where $a_j^* = \langle f^*,f_j\rangle_{L^2(\mu)}$, then we have $f_J^* = P_Jf^* = \sum_{j\in J}a_j^* f_j$, and $f_{J^c}^* = \sum_{j\in J^c}a_j^* f_j$. Note here that $V_J$ is a finite-dimensional Hilbert space, and therefore $f_J^*$ admits an $\cH$-lift, namely $h_J^* = \sum_{j\in J}\frac{a_j^*}{\sqrt{\sigma_j}}\ve_j \in\cH$, which satisfies $S h_J^* = f_J^*$. Here we recall that $S:g\in \cH\mapsto g\in L^2(\mu)$ satisfies $S\ve_j = \sqrt{\sigma_j}f_j$. On the other hand, $V_{J^c}$ is an infinite-dimensional Hilbert space, and therefore $f_{J^c}^*$ does not necessarily admit a lift. Hence, in the original proof of Theorem~\ref{theo:main} in Section~\ref{sec:proof_main}, every occurrence involving $\bX\bbeta_{J^c}^*$ is replaced by $\vf_{J^c}^*$, whereas every occurrence involving $\bbeta_J^*$ is replaced by $h_J^*$. For instance, $\tilde\bbeta_J = \varphi_t(\Sigma)\Sigma \bbeta_J^*$ in the original proof is now replaced, in the proof for kernel regression, by $\tilde h_J = \varphi_t(\Sigma)\Sigma h_J^* \in \cH$. Then we still have $h_J^* - \tilde h_J = \psi_t(\Sigma)h_J^*$, and $\| S \tilde h_J - f_J^* \|_{L^2(\mu)} = \|\Sigma_J^{\frac{1}{2}}\psi_t(\Sigma)h_J^*\|_\cH = \|\psi_t(T_K)f_J^*\|_{L^2(\mu)}$.

\subsubsection{Estimation error}
For the stochastic argument, we list the differences in the following.
\begin{enumerate}
    \item For the noise, we replace the Borel--TIS inequality for $\bxi$ by the Hanson--Wright inequality, c.f., \cite[Section 6.2]{vershynin_high-dimensional_2018} in the bounded case. In the Gaussian case, we continue using the Borel-TIS inequality.
    
    \item For the random event $\Omega_t$, we replace the proof of Lemma~\ref{lemma:IP} by the classical proof based on a change-of-norm argument; see, for instance, \cite[Lemma 5]{huckerNotePredictionError2023}. Since this step is standard, we do not repeat its proof here. Therefore, in the following proof, we may always continue to work on the set $\Omega_t$. 

    \item For $\|\vf_{J^c}^*\|_2\leq C\sqrt{N}\|f_{J^c}^*\|_{L^2(\mu)}$, we use Markov's inequality, so that $\bP(\|\vf_{J^c}^*\|_2\leq C\sqrt{N}\|f_{J^c}^*\|_{L^2(\mu)})\geq \frac{9999}{10000}$.

    \item We replace the ``upper side of Dvoretzky-Milman theorem'', namely \eqref{eq:noise_part_DM} by \eqref{eq:alternative_upper_DM} below.

\end{enumerate}
We summarize the stochastic argument as below.
\begin{Proposition}\label{prop:stochastic_argument_upper_theo_RKHS}
    Under the assumptions of Theorem~\ref{theo:main_RKHS}, there exists an absolute constant $C>1$ such that with probability at least $\frac{9998}{10000}$, the following random event holds.
    \begin{align*}
        \Omega_{\mathrm{upper}} = \left\{ \|\vf_{J^c}^*\|_2\leq C\sqrt{N}\|f_{J^c}^*\|_{L^2(\mu)}\right\} \cap \Omega_t.
    \end{align*}
\end{Proposition}
Repeating the proof of Theorem~\ref{theo:main}, we obtain
\begin{Proposition}\label{prop:upper_RKHS}
    Grant the assumption of Theorem~\ref{theo:main_RKHS}. There exists an absolute constant $C$ such that with probability at least $\frac{9997}{10000}$, there holds $\|\hat f_N(\vf^*+\bxi) - f^*\|_{L^2(\mu)}\leq Cr(V_{J_*},V_{J_*^c}) + C\frac{\square}{t}\|T_K^{-1}f_{J_*}^*\|_{L^2(\mu)}$.
\end{Proposition}
In the proof of Proposition~\ref{prop:upper_RKHS}, there are several differences from the proof of Theorem~\ref{theo:main}.
\begin{enumerate}
    \item Upper bound for $\|\Sigma_{J^c}^{\frac{1}{2}}\varphi_t(\hat\Sigma)\frac{1}{N}\bX^\top\vf_{J^c}^*\|_\cH$.

    On $\Omega_t$, $\hat\Sigma\preceq \Sigma + \square\Sigma_t$. For any $f\in\cH$, there holds $\frac{1}{N}\|\bX\Sigma_{J^c}^{\frac{1}{2}}f\|_2^2 = \|\hat\Sigma^{\frac{1}{2}}\Sigma_{J^c}^{\frac{1}{2}}f\|_\cH^2\leq \|\Sigma^{\frac{1}{2}}\Sigma_{J^c}^{\frac{1}{2}}f\|_\cH^2 + \square\|\Sigma_t^{\frac{1}{2}}\Sigma_{J^c}^{\frac{1}{2}}f\|_\cH^2 = (1+\square)\|\Sigma_{J^c}f\|_\cH^2 + \square t^{-1}\|\Sigma_{J^c}f\|_\cH^2$. Therefore,
    \begin{align*}
        \frac{1}{\sqrt{N}}\norm{\Sigma_{J^c}^{\frac{1}{2}}\bX^\top}_{\op} = \frac{1}{\sqrt{N}}\norm{\bX\Sigma_{J^c}^{\frac{1}{2}}}_{\op} \leq \left((1+\square)\|\Sigma_{J^c}\|_{\op}^2 + \square t^{-1}\|\Sigma_{J^c}\|_{\op}\right)^{\frac{1}{2}}.
    \end{align*}Since $\square\lesssim 1$ and $t^{-1}\geq \frac{1}{b}\|\Sigma_{J^c}\|_{\op}$, there exists some absolute constant $C$ depending only on $b$ such that
    \begin{align}\label{eq:alternative_upper_DM}
        \frac{1}{\sqrt{N}}\norm{\Sigma_{J^c}^{\frac{1}{2}}\bX^\top}_{\op}\leq Ct^{-1}.
    \end{align}
    Recall that $\varphi_t(\hat\Sigma)\frac{1}{N}\bX^\top = \frac{1}{N}\bX^\top\varphi_t(\frac{1}{N}\bX\bX^\top)$, we have
    \begin{align*}
        &\|\Sigma_{J^c}^{\frac{1}{2}}\varphi_t(\hat\Sigma)\frac{1}{N}\bX^\top\vf_{J^c}^*\|_\cH \leq \frac{1}{N}\|\Sigma_{J^c}^{\frac{1}{2}}\bX^\top\|_{\op}\|\varphi_t(\frac{1}{N}\bX\bX^\top)\|_{\op}\|\vf_{J^c}^*\|_2 \leq Ct^{-1}\|\varphi_t(\frac{1}{N}\bX\bX^\top)\|_{\op}\frac{1}{\sqrt{N}}\|\vf_{J^c}^*\|_2.
    \end{align*}Now, by Assumption~\ref{ass:filter_fct}, together with $\Omega_{\mathrm{upper}}$, $\|\varphi_t(\frac{1}{N}\bX\bX^\top)\|_{\op}\leq\oC{C_filter_1} t$ and $\|\vf_{J^c}^*\|_2\lesssim \sqrt{N}\|f_{J^c}^*\|_{L^2(\mu)}$. In summary, we still have $\|\Sigma_{J^c}^{\frac{1}{2}}\varphi_t(\hat\Sigma)\frac{1}{N}\bX^\top\vf_{J^c}^*\|_\cH\lesssim \|f_{J^c}^*\|_{L^2(\mu)}$.

    \item Upper bound for $\|\Sigma_{J^c}^{\frac{1}{2}}\varphi_t(\hat\Sigma)[N^{-1}\bX^\top]\bxi\|_\cH$.

    Applying \eqref{eq:alternative_upper_DM} to the original analysis, we still obtain $\|\Sigma_{J^c}^{\frac{1}{2}}\varphi_t(\hat\Sigma)[N^{-1}\bX^\top]\bxi\|_\cH\lesssim \sigma_\xi t\sqrt{\frac{\Tr(\Sigma_{J^c}^2)}{N}} + \sigma_\xi\sqrt{\frac{|J|}{N}}$.
\end{enumerate}

\subsubsection{Approximation error}
For the deterministic argument, by the definition of $\hat f_N$,
\begin{align*}
    \hat f_N - f^\circ = \bigg[ \bigg( \big( \hat f_N(\vf^*) - f^*\big) + \frac{1}{N}\varphi_t(\hat\Sigma)\bX^\top\bxi \bigg) + \frac{1}{N}\varphi_t(\hat\Sigma)\bX^\top(\vf^\circ - \vf^*) \bigg] + (f^* - f^\circ),
\end{align*}where $\vf^\circ - \vf^* = (f^\circ(X_i)-f^*(X_i))_{i=1}^N$.
Note that, since $f^\circ-f^*\perp\overline{\cH}$ in $L^2(\mu)$ inner product, taking $L^2(\mu)$ norm on both sides gives
\begin{align}\label{eq:decomposition_misspecified}
    \|\hat f_N - f^\circ\|_{L^2(\mu)}^2 &= \| f^\circ - f^*\|_{L^2(\mu)}^2 + \norm{\bigg( \big( \hat f_N(\vf^*) - f^*\big) + \frac{1}{N}\varphi_t(\hat\Sigma)\bX^\top\bxi \bigg) + \frac{1}{N}\varphi_t(\hat\Sigma)\bX^\top(\vf^\circ - \vf^*)}_{L^2(\mu)}^2.
\end{align}By triangular inequality,
\begin{align*}
    &\norm{\bigg( \big( \hat f_N(\vf^*) - f^*\big) + \frac{1}{N}\varphi_t(\hat\Sigma)\bX^\top\bxi \bigg) + \frac{1}{N}\varphi_t(\hat\Sigma)\bX^\top(\vf^\circ - \vf^*)}_{L^2(\mu)}^2\leq 2\norm{\hat f_N(\vf^*) - f^* + \frac{1}{N}\varphi_t(\hat\Sigma)\bX^\top\bxi}_{L^2(\mu)}^2\\
    &+ 2\norm{\frac{1}{N}\varphi_t(\hat\Sigma)\bX^\top(\vf^\circ - \vf^*)}_{L^2(\mu)}^2 = 2\norm{\hat f_N(\vf^*+\bxi) - f^*}_{L^2(\mu)}^2 + 2\norm{\frac{1}{N}\varphi_t(\hat\Sigma)\bX^\top(\vf^\circ - \vf^*)}_{L^2(\mu)}^2,
\end{align*}where we have used the definition $\hat f_N:\vy\in\bR^N\mapsto \frac{1}{N}\varphi_t(\hat\Sigma)\bX^\top\vy$. By repeating the proof in Section~\ref{sec:proof_main}, one immediately obtains a high-probability upper bound for $2\norm{\hat f_N(\vf^*+\bxi) - f^*}_{L^2(\mu)}^2$. In the following, we only study a high-probability upper bound for $2\norm{\frac{1}{N}\varphi_t(\hat\Sigma)\bX^\top(\vf^\circ - \vf^*)}_{L^2(\mu)}^2$. Since $\varphi_t(\hat\Sigma)\bX^\top(\vf^\circ - \vf^*)\in\cH$, we have
\begin{align*}
    &2\norm{\frac{1}{N}\varphi_t(\hat\Sigma)\bX^\top(\vf^\circ - \vf^*)}_{L^2(\mu)}^2 = \frac{2}{N^2}\norm{\Sigma^{1/2}\varphi_t(\hat\Sigma)\bX^\top(\vf^\circ - \vf^*)}_\cH^2\leq \frac{2}{N^2}\|\Sigma^{\frac{1}{2}}\Sigma_t^{-\frac{1}{2}}\|_{\op}^2\norm{\Sigma_t^{\frac{1}{2}}\varphi_t(\hat\Sigma)\bX^\top(\vf^\circ - \vf^*)}_\cH^2\\
    &\leq \frac{2}{N^2}\|\Sigma^{\frac{1}{2}}\Sigma_t^{-\frac{1}{2}}\|_{\op}^2\norm{\Sigma_t^{\frac{1}{2}}\varphi_t(\hat\Sigma)\Sigma_t^{\frac{1}{2}}}_{\op}^2\norm{\Sigma_t^{-\frac{1}{2}}\bX^\top}_{\op}^2\norm{\vf^\circ - \vf^*}_2^2.
\end{align*}By \eqref{eq:IP_varphi_t_hat_Sigma}, \eqref{eq:upper_operator_norm_Sigma_t_inverse_X_transpose}, and $\Sigma \preceq \Sigma_t$, on $\Omega_t$, there holds
\begin{align*}
    2\norm{\frac{1}{N}\varphi_t(\hat\Sigma)\bX^\top(\vf^\circ - \vf^*)}_{L^2(\mu)}^2\leq \frac{16\oC{C_filter_1}}{N}\sum_{i=1}^N\left( f^\circ(X_i) - f^*(X_i) \right)^2.
\end{align*}By Markov's inequality, there exists an absolute constant $C^2$ such that
\begin{align}\label{eq:upper_approximation_error_empirical}
    \bP\left(2\|\frac{1}{N}\varphi_t(\hat\Sigma)\bX^\top(\vf^\circ - \vf^*)\|_{L^2(\mu)}^2\leq C^2\| f^\circ - f^* \|_{L^2(\mu)}^2\right)\geq \frac{9999}{10000}.
\end{align}
In conclusion, the proof of the upper bound is complete.

\subsection{Lower bound of Theorem~\ref{theo:main_RKHS}}

We use the following lemma.
\begin{Lemma}\label{lemma:comparision}
    Let $\vu, \vm$ be vectors in a normed space and $A \ge 0$. Suppose $\|\vu\| \ge \rho$ and $\|\vm\| \le K A$ for some $\rho, K > 0$. Let $E := (\|\vu+\vm\|^2 + A^2)^{1/2}$. Then $E \ge \frac{\rho + A}{K + 2}$.
\end{Lemma}
\beginproof (of Lemma~\ref{lemma:comparision})
To prove this lemma, we note that $A \le E$ is obvious. By the triangle inequality,
\begin{align*}
\rho \le \|\vu\| \le \|\vu+\vm\| + \|\vm\| \le E + K A \le E + K E = (K + 1)E.
\end{align*}
Adding the two inequalities yields $\rho + A \le (K + 1)E + E = (K + 2)E$, which implies $E \ge \frac{\rho + A}{K + 2}$.
\endproof
Let $\vu = \big( \hat f_N(\vf^*) - f^*\big) + \frac{1}{N}\varphi_t(\hat\Sigma)\bX^\top\bxi$, $\vm = \frac{1}{N}\varphi_t(\hat\Sigma)\bX^\top(\vf^\circ - \vf^*)$, let $\|\cdot\| = \|\cdot\|_{L^2(\mu)}$, let $K = C$, let $\rho = \|\vu\|_{L^2(\mu)}$, and let $A = \| f^\circ - f^*\|_{L^2(\mu)}$. Then combining Lemma~\ref{lemma:comparision}, \eqref{eq:decomposition_misspecified} and \eqref{eq:upper_approximation_error_empirical} gives that with probability at least $\frac{9999}{10000}$,
\begin{align}\label{eq:lower_estimation_error_bounded_0}
    \norm{\hat f_N - f^\circ}_{L^2(\mu)} \geq c\left( \norm{f^\circ - f^*}_{L^2(\mu)} + \norm{\hat f_N(\vf^*+\bxi) - f^*}_{L^2(\mu)} \right).
\end{align}
Moreover, in Section~\ref{sec:optimal_FSD}, we proved a lower bound for the term $\|\hat f_N(\vf^*+\bxi) - f^* \|_{L^2(\mu)}$ but in the case where $X$ has independent coordinates. In the following, we derive a similar lower bound but without this assumption. 

Let $\psi:\bR^N\ni\vx\mapsto  \|\hat{f}_N(\vf^* + \vx) - f^*\|_{L^2(\mu)}$. Since this function is the composition of a norm and an affine mapping, $\psi$ is clearly convex. When $\bxi$ is Gaussian, the proof is the same as the linear regression case (we use Borel-TIS inequality). When $\boldsymbol{\xi}$ has independent and identically distributed coordinates and $\|\xi\|_{L^\infty} < \infty$, we apply Talagrand's concentration inequality for convex-Lipschitz functions of independent and bounded variables (c.f., \cite[Theorem 7.12]{boucheron_concentration_2013}): there exists an absolute constant $c > 0$, such that for any $t > 0$,
\begin{align*}
\mathbb{P}_{\boldsymbol{\xi}}\left(|\psi(\boldsymbol{\xi}) - \mathbb{E}_{\boldsymbol{\xi}}\psi(\boldsymbol{\xi})| \ge t\right) \le 2 \exp\left(-\frac{c t^2}{\|\xi\|_{L^\infty}^2 \|\psi\|_{Lip}^2}\right).
\end{align*}
By integrating the tail probability above, we directly obtain the variance bound:
\begin{align*}
\mathbb{E}_{\boldsymbol{\xi}}[\psi(\boldsymbol{\xi})^2] - [\mathbb{E}_{\boldsymbol{\xi}}\psi(\boldsymbol{\xi})]^2 \le \frac{2}{c} \|\xi\|_{L^\infty}^2 \|\psi\|_{Lip}^2.
\end{align*}
Meanwhile, the lower tail probability of Talagrand's inequality shows that, for any $r > 0$, with probability at least $1 - \exp(-r)$,
\begin{align*}
\psi(\boldsymbol{\xi}) \ge \mathbb{E}_{\boldsymbol{\xi}}\psi(\boldsymbol{\xi}) - \frac{\|\xi\|_{L^\infty} \|\psi\|_{Lip}}{\sqrt{c}} \sqrt{r}.
\end{align*}
In the following, we prove that for any $r \ge 2$, with probability at least $1 - \exp(-r)$, we always have:
\begin{align*}
\psi(\boldsymbol{\xi}) \ge \sqrt{\frac{\mathbb{E}_{\boldsymbol{\xi}}[\psi(\boldsymbol{\xi})^2]}{2}} - \frac{\|\xi\|_{L^\infty} \|\psi\|_{Lip}}{\sqrt{c}} \sqrt{r}.
\end{align*}
\begin{enumerate}
    \item If $\mathbb{E}_{\boldsymbol{\xi}}[\psi(\boldsymbol{\xi})^2] \ge \frac{4}{c} \|\xi\|_{L^\infty}^2 \|\psi\|_{Lip}^2$, the aforementioned variance bound implies $[\mathbb{E}_{\boldsymbol{\xi}}\psi(\boldsymbol{\xi})]^2 \ge \frac{1}{2}\mathbb{E}_{\boldsymbol{\xi}}[\psi(\boldsymbol{\xi})^2]$. Substituting this into the lower tail probability inequality yields the desired conclusion.

    \item Conversely, if $\mathbb{E}_{\boldsymbol{\xi}}[\psi(\boldsymbol{\xi})^2] < \frac{4}{c} \|\xi\|_{L^\infty}^2 \|\psi\|_{Lip}^2$, as long as we choose $r \ge 2$, the right-hand side of the target lower bound, $\sqrt{\frac{1}{2}\mathbb{E}_{\boldsymbol{\xi}}[\psi(\boldsymbol{\xi})^2]} - c^{-1/2}\|\xi\|_{L^\infty} \|\psi\|_{Lip} \sqrt{r}$, is strictly non-positive. Since $\psi(\boldsymbol{\xi}) \ge 0$ holds almost surely by the definition of a norm, the lower bound holds trivially in this case.
\end{enumerate}

Similar to the derivation of \eqref{eq:first_lower_bound}, there holds
\begin{align}\label{eq:lower_estimation_error_bounded}
\begin{aligned}
    &\|\hat{f}_N(\vf^* + \boldsymbol{\xi}) - f^*\|_{L^2(\mu)} \\
    &\ge \frac{1}{2}\|\hat{f}_N(\vf^*) - f^*\|_{L^2(\mu)} + \frac{\sigma_\xi}{2} \sqrt{\frac{\operatorname{Tr}[\Sigma \varphi_t^2(\hat{\Sigma})\hat{\Sigma}]}{N}} - \frac{\|\xi\|_{L^\infty}}{\sqrt{c}} \left\|\Sigma^{1/2}\varphi_t(\hat{\Sigma})\frac{\mathbb{X}^\top}{\sqrt{N}}\right\|_{op} \sqrt{\frac{r}{N}}.
\end{aligned}
\end{align}

Here, the lower bound for the bias term $\frac{1}{2}\|\hat{f}_N(\vf^*) - f^*\|_{L^2(\mu)}$ still holds, namely the derivation in Section~\ref{sec:proof_lower_bound_bias} remains valid. Therefore, with probability at least $\frac{9998}{10000}$, there still holds
\begin{align}\tag{\ref{eq:LB_final_bias_term}}
    \norm{\hat f_N(\vf^*)-f^*}_{L^2(\mu)} \geq \frac12\|\psi_t(T_K)f^*_J\|_{L^2(\mu)} + \frac{1}{2}\norm{f_{J^c}^*}_{L^2(\mu)} - \frac{C \square}{t} \norm{T_K^{-1}f^*_J}_{L^2(\mu)}.
\end{align}In addition, the derivation in Section~\ref{sub:an_upper_bound_for_the_weak_variance_term_} remains valid, namely
\begin{equation}\tag{\ref{eq:LB_weak_var_term}}
  \norm{\Sigma^{1/2} \varphi_t(\hat\Sigma)(\bX^\top/ \sqrt{N})}_{op} \leq \norm{\Sigma^{1/2} \Sigma_t^{-1/2}}_{\op}  \norm{\Sigma_t^{1/2} \hat \Sigma_t^{-1/2}}_{\op} \norm{\hat \Sigma_t\varphi_t(\hat\Sigma)^{2}\hat \Sigma}^{1/2}_{op}\lesssim 1. 
\end{equation}
Therefore, we may take $r = \frac{c_0^2c^2}{16} \frac{\sigma_\xi^2}{\|\xi\|_{L^\infty}^2} k^*$ for some absolute constant $c_0$ that appears later in Proposition~\ref{prop:lower_variance_general_determinisitc}, and the weak variance term $\frac{\|\xi\|_{L^\infty}}{\sqrt{c}} \left\|\Sigma^{1/2}\varphi_t(\hat{\Sigma})\frac{\mathbb{X}^\top}{\sqrt{N}}\right\|_{op} \sqrt{\frac{r}{N}}\leq \frac{c_0}{4}\sigma_\xi\sqrt{\frac{k^*}{N}}$. So far, we have proved that
\begin{Proposition}\label{prop:lower_bias_weak_variance_general}
    For $c_0$ that will appear later in Proposition~\ref{prop:lower_variance_general_determinisitc}, suppose $k^*\geq \frac{32}{c_0^2c^2}\frac{\|\xi\|_{L^\infty}^2}{\sigma_\xi^2}$. There exist absolute constants $C,c'$ depending only on $c,c_0$, such that with probability at least $\frac{9998}{10000} - \exp(-Ck^*)$,
    \begin{align*}
        &\frac{1}{2}\|\hat{f}_N(\vf^*) - f^*\|_{L^2(\mu)}  - \frac{\|\xi\|_{L^\infty}}{\sqrt{c}} \left\|\Sigma^{1/2}\varphi_t(\hat{\Sigma})\frac{\mathbb{X}^\top}{\sqrt{N}}\right\|_{op} \sqrt{\frac{r}{N}} \\
        &\geq c'\|\psi_t(T_K)f^*_J\|_{L^2(\mu)} + c'\norm{f_{J^c}^*}_{L^2(\mu)} - \frac{C \square}{t} \norm{T_K^{-1}f^*_J}_{L^2(\mu)} - \frac{c_0}{4}\sigma_\xi\sqrt{\frac{k^*}{N}}.
    \end{align*}
\end{Proposition}
We now prove that the variance lower bound indeed contains the term $\sigma_\xi\sqrt{\frac{k^*}{N}}$.
For the lower bound of the variance term $\frac{\sigma_\xi}{2} \sqrt{\frac{\operatorname{Tr}[\Sigma \varphi_t^2(\hat{\Sigma})\hat{\Sigma}]}{N}}$, we need to use a method different from that of \cite{tsigler_benign_2023}.

\subsubsection{Stochastic argument of the variance lower bound}

Let $\fU(J^c,N) = Ct^{-2}$, and $\fL(J^c,N)=c_0'\sum_{j\in J_*^c}\sigma_j^2$ for some absolute constant $c_0'$ that will be determined in Lemma~\ref{lemma:lower_trace_lower_bound_variance}. In this section, slightly abusing notation, we write $P_J = \sum_{j\in J}\ve_j\otimes_\cH\ve_j$, $P_{J^c} = \sum_{j\in J^c}\ve_j\otimes_\cH\ve_j$, and denote $\Sigma_{J^c} = P_{J^c}\Sigma P_{J^c}$, $\bX_J = \bX P_J$, and $\bX_{J^c} = \bX P_{J^c}$. Here, we do not use $f_j$, because for the variance term $\frac{1}{N}\varphi_t(\hat\Sigma)\bX^\top\bxi$, all arguments can be carried out within $\cH$. We consider the following random event $\Omega_{\mathrm{lower}}$:
\begin{align*}
    \Omega_{\mathrm{lower}} =& \left\{ \frac{1}{N}\bX_{J^c}\bX_{J^c}^\top + t^{-1}I_N\preceq C\left( \frac{\Tr(\Sigma_{J^c})}{N} + t^{-1}\right)I_N \right\}\\
    &\cap \left\{ \forall f_J\in V_J,\, c\|\Sigma_J^{\frac{1}{2}}f_J\|_\cH\leq \frac{1}{\sqrt{N}}\|\bX_J f_J\|_2 \leq C'\|\Sigma_J^{\frac{1}{2}}f_J\|_\cH \right\}\\
    &\cap \left\{ \norm{\frac{1}{N}\bX_{J^c}\Sigma_{J^c}\bX_{J^c}^\top}_{\op}\leq \mathfrak{U}(J^c,N)\right\} \\
    &\cap \left\{ \frac{1}{N}\Tr(\bX_{J^c}\Sigma_{J^c}\bX_{J^c}^\top) \geq \mathfrak{L}(J^c,N) \right\}.
\end{align*}
Here, we call the event $\left\{ \frac{1}{N}\bX_{J^c}\bX_{J^c}^\top + t^{-1}I_N\preceq C\left( \frac{\Tr(\Sigma_{J^c})}{N} + t^{-1} \right)I_N \right\}$ the upper side of the Dvoretzky--Milman theorem; see \cite{gavrilopoulos_geometrical_2025}. We call the event $\left\{ \forall f_J\in V_J,\, c\|\Sigma_J^{\frac{1}{2}}f_J\|_\cH\leq \frac{1}{\sqrt{N}}\|\bX_J f_J\|_2 \leq C'\|\Sigma_J^{\frac{1}{2}}f_J\|_\cH \right\}$ the isomorphic property on $V_J$.
We now prove that the event $\Omega_{\mathrm{lower}}$ holds with high probability.

We first prove that both the upper side of the Dvoretzky--Milman theorem and the isomorphic property are consequences of the change-of-norm argument. This only requires the condition $\|K\|_{L^\infty(\mu\otimes\mu)}<\infty$.
\begin{Lemma}\label{lemma:DM_RIP_on_Omega_t}
    \begin{align*}
        \Omega_t\subset &\Bigg(\left\{ \frac{1}{N}\bX_{J^c}\bX_{J^c}^\top + t^{-1}I_N\preceq C\left( \frac{\Tr(\Sigma_{J^c})}{N} + t^{-1}I_N \right) \right\}\\
    &\cap \left\{ \forall f_J\in V_J,\, c\|\Sigma_J^{\frac{1}{2}}f_J\|_\cH\leq \frac{1}{\sqrt{N}}\|\bX_J f_J\|_2 \leq C'\|\Sigma_J^{\frac{1}{2}}f_J\|_\cH \right\}\cap \left\{ \norm{\frac{1}{N}\bX_{J^c}\Sigma_{J^c}\bX_{J^c}^\top}_{\op}\leq \mathfrak{U}(J^c,N)\right\}\Bigg).
    \end{align*}
\end{Lemma}
\beginproof
Let us first consider the event $\left\{ \frac{1}{N}\bX_{J^c}\bX_{J^c}^\top + t^{-1}I_N\preceq C\left( \frac{\Tr(\Sigma_{J^c})}{N} + t^{-1}I_N \right) \right\}$. On $\Omega_t$, we have $P_{J^c}\hat\Sigma P_{J^c} \preceq (1+\square)\Sigma_{J^c} + \square t^{-1}I_{V_{J^c}}$. Since $\|\Sigma_{J^c}\|_{\op}\leq bt^{-1}$, we have $P_{J^c}\hat\Sigma P_{J^c} \preceq Ct^{-1}I_{V_{J^c}}$, that is, for all $f_{J^c}\in V_{J^c}$, the corresponding empirical quadratic form admits the same upper bound. Consequently, $\frac{1}{N}\bX_{J^c}\bX_{J^c}^\top\preceq C t^{-1} I_N\preceq C\left(t^{-1}+\frac{\Tr(\Sigma_{J^c})}{N}\right)I_N$.
    
Next, we show that the event $\left\{ \forall f_J\in V_J,\, c\|\Sigma_J^{\frac{1}{2}}f_J\|_\cH\leq \frac{1}{\sqrt{N}}\|\bX_J f_J\|_2 \leq C'\|\Sigma_J^{\frac{1}{2}}f_J\|_\cH \right\}$ contains $\Omega_t$ (see \eqref{eq:change_of_norm}). Indeed, since $\sigma_j\geq bt^{-1}$ for every $j\in J$, we have $t^{-1}\|f_J\|_\cH^2 \leq \frac{1}{b}\|\Sigma_J^{1/2}f_J\|_\cH^2$ for every $f_J\in V_J$. Therefore, by the definition $\Sigma_t = \Sigma + t^{-1}I$, we have $\|\Sigma_t^{\frac{1}{2}}f_J\|_\cH^2\leq\frac{1+b}{b}\|\Sigma_J^{\frac{1}{2}}f_J\|_\cH^2$. Substituting this into $\Omega_t$, we obtain $\left|\frac{1}{N}\|\bX_J f_J\|_2^2 - \|\Sigma_J^{1/2}f_J\|_\cH^2\right| \leq \square\frac{1+b}{b}\|\Sigma_J^{\frac{1}{2}}f_J\|_\cH^2$. Hence, when $\square\leq \frac{b}{2(1+b)}$, one may take $c = \frac{1}{\sqrt{2}}$ and $C' = \sqrt{\frac{3}{2}}$. For the third random event, by equation \eqref{eq:alternative_upper_DM}, we have $\norm{\frac{1}{N}\bX_{J^c}\Sigma_{J^c}\bX_{J^c}^\top}_{\op}\leq Ct^{-2}$ on $\Omega_t$. 
\endproof

For the last event, we have the following lemma:
\begin{Lemma}\label{lemma:lower_trace_lower_bound_variance}
    Grant the assumptions of Theorem~\ref{theo:main_RKHS}. Let $0<\delta'<\frac{1}{2}$ be some real number. If $N\gtrsim \log(1/\delta')$, then there exists an absolute constant $c_0'$ depending on $\delta$, such that
    \begin{align*}
        \bP\left\{ \frac{1}{N}\Tr(\bX_{J^c}\Sigma_{J^c}\bX_{J^c}^\top) \geq c_0'\Tr(\Sigma_{J^c}^2) \right\}\geq 1-\delta'.
    \end{align*}
\end{Lemma}
\beginproof Recall that $S:g\in\cH\mapsto Sg\in L^2(\mu)$. As a result, $\phi_{J^c}(\vx) = (\sum_{j\in J^c}\ve_j\otimes_\cH\ve_j)\phi(\vx) = \sum_{j\in J^c}\ve_j(\vx)\ve_j$. Moreover, by, $S\ve_j = \sqrt{\sigma_j}f_j$, and $\ve_j(\vx) = \sqrt{\sigma_j}f_j(\vx)$, there holds $S\phi_{J^c}(\vx) = \sum_{j\in J^c}\sigma_j f_j(\vx)f_j$. Consequently, $\|S\phi_{J^c}(\vx)\|_{L^2(\mu)}^2 = \sum_{j\in J^c}\sigma_j^2 f_j^2(\vx)$.
This random event follows by observing that $$\frac{1}{N}\Tr(\bX_{J^c}\Sigma_{J^c}\bX_{J^c}^\top) = \frac{1}{N}\sum_{i=1}^N\|S\phi_{J^c}(X_i)\|_{L^2(\mu)}^2 = \frac{1}{N}\sum_{i=1}^N\left[\sum_{j\in J^c}\sigma_j^2 f_j^2(X_i)\right]$$ is an average of i.i.d. random variables, and hence follows immediately from Bernstein's inequality applied to selectors $\{\delta_i\}_{i=1}^N$, where $\delta_i = \1_{\{ \|S\phi_{J^c}(X_i)\|_{L^2(\mu)}^2 \geq \frac{1}{10}\Tr(\Sigma_{J^c}^2) \}}$. Since $\|S\phi_{J^c}(X_i)\|_{L^2(\mu)}^2\geq 0$, we have $\|S\phi_{J^c}(X_i)\|_{L^2(\mu)}^2\geq \frac{1}{10}\delta_i\Tr(\Sigma_{J^c}^2)$ and hence $\frac{1}{N}\|S\phi_{J^c}(X_i)\|_{L^2(\mu)}^2\geq\frac{1}{10}\Tr(\Sigma_{J^c}^2)\frac{1}{N}\sum_{i=1}^N\delta_i$. By the Bernstein's inequality applied to $\frac{1}{N}\sum_{i=1}^N\delta_i$, the proof is complete.
\endproof

In summary, we obtain the following proposition.
\begin{Proposition}\label{prop:stochastic_argument_lower_bound_variance}
    Under the assumptions of Theorem~\ref{theo:main_RKHS}, $\Omega_{\mathrm{lower}}$ happens with probability at least $\frac{9998}{10000}$.
\end{Proposition}

\subsubsection{Deterministic argument of the variance lower bound}
In the following, we work on $\Omega_{\mathrm{lower}}$. Note that that $\bX\Sigma\bX^\top = \bX_J\Sigma_J\bX_J^\top + \bX_{J^c}\Sigma_{J^c}\bX_{J^c}^\top$ and $\bX^\top\varphi_t(\frac{1}{N}\bX\bX^\top) = \varphi_t(\frac{1}{N}\bX^\top\bX)\bX^\top$. Therefore, we have
\begin{align}\label{eq:lower_variacne_decomposition_bounded}
    \begin{aligned}
        &\Tr\left[\Sigma \varphi_t^2(\hat{\Sigma})\hat{\Sigma}\right] = \frac{1}{N}\Tr\left[ \Sigma\varphi_t\left(\frac{1}{N}\bX^\top\bX\right)\bX^\top\bX\varphi_t\left(\frac{1}{N}\bX^\top\bX\right) \right]\\
        &= \frac{1}{N}\Tr\left[ \Sigma\bX^\top\varphi_t\left( \frac{1}{N}\bX\bX^\top \right)\varphi_t\left( \frac{1}{N}\bX\bX^\top \right)\bX \right] = \frac{1}{N}\Tr\left[ \varphi_t\left( \frac{1}{N}\bX\bX^\top \right)(\bX\Sigma\bX^\top)\varphi_t\left( \frac{1}{N}\bX\bX^\top \right) \right]\\
        &=\frac{1}{N}\Tr\left[\varphi_t\left(\frac{1}{N}\bX\bX^\top\right)\left(\bX_J\Sigma_J\bX_J^\top\right)\varphi_t\left(\frac{1}{N}\bX\bX^\top\right)\right] \\
        &+\frac{1}{N} \Tr\left[\varphi_t\left(\frac{1}{N}\bX\bX^\top\right)\left(\bX_{J^c}\Sigma_{J^c}\bX_{J^c}^\top\right)\varphi_t\left(\frac{1}{N}\bX\bX^\top\right)\right].
    \end{aligned}
\end{align}
\paragraph{Lower bound for \texorpdfstring{$\frac{1}{N}\Tr\left[\varphi_t\left(\frac{1}{N}\bX\bX^\top\right)\left(\bX_J\Sigma_J\bX_J^\top\right)\varphi_t\left(\frac{1}{N}\bX\bX^\top\right)\right]$}{Tr 1}.}
By Assumption~\ref{ass:filter_fct}, $\varphi_t(x)\geq \oc{c_filter_1}(x+t^{-1})^{-1}$, there holds
\begin{align*}
    &\frac{1}{N}\Tr\left[\varphi_t\left(\frac{1}{N}\bX\bX^\top\right)\left(\bX_J\Sigma_J\bX_J^\top\right)\varphi_t\left(\frac{1}{N}\bX\bX^\top\right)\right]\geq \oc{c_filter_1}^2\Tr\left(\frac{1}{N}\bX_J\Sigma_J\bX_J^\top\left(\frac{1}{N}\bX\bX^\top + t^{-1}I_N\right)^{-2}\right)
\end{align*}Recall that $\Sigma_J = \sum_{j\in J}\sigma_j\ve_j\otimes\ve_j$. Therefore, $\frac{1}{N}\bX_J\Sigma_J\bX_J^\top = \sum_{j\in J}\sigma_j^2\left(\frac{\bX_J\ve_j}{\sqrt{N\sigma_j}}\right)\otimes\left(\frac{\bX_J\ve_j}{\sqrt{N\sigma_j}}\right)$. Therefore,
\begin{align}\label{eq:lower_variance_pre_0}
    &\frac{1}{N}\Tr\left[\varphi_t\left(\frac{1}{N}\bX\bX^\top\right)\left(\bX_J\Sigma_J\bX_J^\top\right)\varphi_t\left(\frac{1}{N}\bX\bX^\top\right)\right]\geq\oc{c_filter_1}^2\sum_{j\in J}\sigma_j^2\left< \left(\frac{1}{N}\bX\bX^\top + t^{-1}I_N\right)^{-2}\frac{\bX_J\ve_j}{\sqrt{N\sigma_j}},\frac{\bX_J\ve_j}{\sqrt{N\sigma_j}}\right>_{\ell_2^N}.
\end{align}By Cauchy-Schwartz, for any $j\in J$,
\begin{align}\label{eq:lower_variance_pre_1}
    &\left< \left(\frac{1}{N}\bX\bX^\top + t^{-1}I_N\right)^{-2}\frac{\bX_J\ve_j}{\sqrt{N\sigma_j}},\frac{\bX_J\ve_j}{\sqrt{N\sigma_j}}\right>_{\ell_2^N} \norm{\frac{\bX_J\ve_j}{\sqrt{N\sigma_j}}}_2^2 \geq \left< \left(\frac{1}{N}\bX\bX^\top + t^{-1}I_N\right)^{-1}\frac{\bX_J\ve_j}{\sqrt{N\sigma_j}},\frac{\bX_J\ve_j}{\sqrt{N\sigma_j}}\right>_{\ell_2^N}^2.
\end{align}Moreover, on $\Omega_{\mathrm{lower}}$, we have
\begin{align*}
    \left(\frac{1}{N}\bX\bX^\top + t^{-1}I_N\right)^{-1}\succeq \left(\frac{1}{N}\bX_J\bX_J^\top + C\left(t^{-1} + \frac{\Tr(\Sigma_{J^c})}{N}\right)I_N\right)^{-1}
\end{align*}so that we finally obtain
\begin{align*}
    &\frac{1}{N}\Tr\left[\varphi_t\left(\frac{1}{N}\bX\bX^\top\right)\left(\bX_J\Sigma_J\bX_J^\top\right)\varphi_t\left(\frac{1}{N}\bX\bX^\top\right)\right]\\ &\geq\oc{c_filter_1}^2\sum_{j\in J}\sigma_j^2\left< \left(\frac{1}{N}\bX_J\bX_J^\top + C\left(t^{-1} + \frac{\Tr(\Sigma_{J^c})}{N}\right)I_N\right)^{-1}\frac{\bX_J\ve_j}{\sqrt{N\sigma_j}},\frac{\bX_J\ve_j}{\sqrt{N\sigma_j}}\right>_{\ell_2^N}^2 \norm{\frac{\bX_J\ve_j}{\sqrt{N\sigma_j}}}_2^{-2}.
\end{align*}

Next, we make use of the isomorphic property on $V_J$ to establish two results:
\begin{enumerate}
    \item First, rewriting the isomorphic property on $V_J$, we have that $\bigg(\vu_j = \frac{\bX_J\ve_j}{\sqrt{N\sigma_j}}\bigg)_{j\in J}$ forms a frame on $\ell_2^J$, that is, for any $\va\in \ell_2^J$, there holds $c\|\va\|_2\leq \|\sum_{j\in J}a_j\vu_j\|_2 \leq C\|\va\|_2$, see \cite[Definition 3.3.8]{vershynin_high-dimensional_2018}.


    \item We prove that for every $j\in J$, there exists $\vv_j\in\Span(\ve_\ell:\ell\in J)$, such that for every $\ell\in J$,  $\langle\vv_j,\vu_\ell\rangle = \1_{\ell=j}$, and $\|\vv_j\|_2\leq \frac{1}{\sqrt{c}}$. Indeed, let $T = (T_{j_1,j_2})_{j_1,j_2\in J} = (\langle \vu_{j_1},\vu_{j_2}\rangle)_{j_1,j_2\in J}$ be the Gram matrix of $(\vu_j)_{j\in J}$. By the frame property proved above, we know that, for any non-zero $\va = (a_j)_{j\in J}\in \ell_2^J$, $\va^\top T \va = \| \sum_{j\in J}a_j\vu_j\|_2^2>c^2\|\va\|_2^2 >0$. Therefore, $T$ is invertible. Hence, for any $j\in J$, let $\vdelta_j = (0,\cdots,0,1,0,\cdots,0)\in\ell_2^J$, namely the vector whose $j$-th coordinate is $1$ and whose other coordinates are $0$. Then there exists $\vb_j = T^{-1}\vdelta_j$ such that, for this $\vb_j$, the vector $\vv_j = \sum_{j\in J}b_j\vu_j$ satisfies $\langle \vv_j,\vu_\ell\rangle = \1_{\ell=j}$ for every $\ell\in J$. Here, we used the identity $\langle \vv_j,\vu_\ell\rangle = \sum_{k\in J}b_k\langle\vu_k,\vu_\ell\rangle = (T\vb_j)_\ell$.

    Moreover, for this $\vv_j$, we have $\|\vv_j\|_2^2 = \langle \sum_{j_1\in J}b_{j_1}\vu_{j_1},\sum_{j_2\in J}b_{j_2}\vu_{j_2}\rangle = \vb_j^\top T\vb_j$. Since $\vb_j = T^{-1}\vdelta_j$, it follows that $\|\vv\|_2^2 = \vdelta_j^\top T^{-1}\vdelta_j$. Since $T\succeq cI$, we have $T^{-1}\preceq c^{-1}I$. Hence, this $\vv$ satisfies $\|\vv\|_2 \leq 1/\sqrt{c}$.
\end{enumerate}
Since $\frac{1}{N}\bX_J\bX_J^\top = \frac{1}{N}\bX_J(\sum_{j\in J}\ve_j\otimes\ve_j)\bX_J^\top = \sum_{j\in J}\sigma_j\left(\frac{\bX_J\ve_j}{\sqrt{N\sigma_j}}\right)\otimes\left(\frac{\bX_J\ve_j}{\sqrt{N\sigma_j}}\right) = \sum_{j\in J}\sigma_j\vu_j\otimes\vu_j$, for each $j\in J$, there exists $\vv_j\in\Span(\vu_\ell:\ell\in J)$ such that $\langle\vv_j,\vu_\ell\rangle=\1_{\ell=j}$ and $\|\vv_j\|_2\leq \frac{1}{\sqrt{c}}$. Therefore, $\langle\vv_j,\frac{1}{N}\bX_J\bX_J^\top\vv_j\rangle = \sigma_j$. Therefore,
\begin{align*}
    \left< \left[ \frac{1}{N}\bX_J\bX_J^\top + C\left(t^{-1} + \frac{\Tr(\Sigma_{J^c})}{N}\right)I_N \right]\vv,\vv \right> \leq \sigma_j + \frac{C}{\sqrt{c}}\left(t^{-1}+\frac{\Tr(\Sigma_{J^c})}{N}\right).
\end{align*}We use the following formula\footnote{This can be proved by Cauchy's inequality $\langle\vu,\vz\rangle^2 = \langle A^{-1/2}\vu,A^{1/2}\vz\rangle^2\leq \langle A^{-1}\vu,\vu\rangle\langle A\vz,\vz\rangle$ and equality holds for $\vz=A^{-1}\vu.$}: for any positive definite matrix $A$ and vector $\vu$,
\begin{align*}
    \langle A^{-1}\vu,\vu\rangle = \sup_{\vz\neq\vzero}\frac{\langle\vu,\vz\rangle^2}{\langle A\vz,\vz\rangle}
\end{align*}applied to $\vu=\vu_j$, $\vz = \vv_j$ and
\begin{align*}
    A =  \frac{1}{N}\bX_J\bX_J^\top + C\left(t^{-1} + \frac{\Tr(\Sigma_{J^c})}{N}\right)I_N.
\end{align*}Then
\begin{align*}
    \left< \left[\frac{1}{N}\bX_J\bX_J^\top + C\left(t^{-1} + \frac{\Tr(\Sigma_{J^c})}{N}\right)I_N\right]^{-1}\frac{\bX_J\ve_j}{\sqrt{N\ve_j}}, \frac{\bX_J\ve_j}{\sqrt{N\ve_j}} \right> \geq \frac{1}{\sigma_j + \frac{C}{\sqrt{c}}\left(t^{-1}+\frac{\Tr(\Sigma_{J^c})}{N}\right)}.
\end{align*}Plug it back to \eqref{eq:lower_variance_pre_1} together with the fact that $\|\frac{\bX_J\ve_j}{\sqrt{N\sigma_j}}\|_2\leq C'$, we obtain that
\begin{align*}
    \left< \left(\frac{1}{N}\bX\bX^\top + t^{-1}I_N\right)^{-2}\frac{\bX_J\ve_j}{\sqrt{N\sigma_j}},\frac{\bX_J\ve_j}{\sqrt{N\sigma_j}}\right>_{\ell_2^N} \geq \frac{c}{C'^2C^2}\frac{1}{\sigma_j + \left(t^{-1}+\frac{\Tr(\Sigma_{J^c})}{N}\right)^2}.
\end{align*}Plug back to \eqref{eq:lower_variance_pre_0}, since the above inequality holds for any $j\in J$, we have
\begin{align*}
    &\frac{1}{N}\Tr\left[\varphi_t\left(\frac{1}{N}\bX\bX^\top\right)\left(\bX_J\Sigma_{J}\bX_J^\top\right)\varphi_t\left(\frac{1}{N}\bX\bX^\top\right)\right]\geq\frac{c\oc{c_filter_1}^2}{C'^4C^2}\sum_{j\in J}\frac{\sigma_j^2}{\left(\sigma_j + \left(t^{-1}+\frac{\Tr(\Sigma_{J^c})}{N}\right)\right)^2}
\end{align*}Now take $J=J_{*}=[k^{*}]$. Since we have assumed $t^{-1}>100\frac{\Tr(\Sigma_{J_{**}^c})}{N}$, by the fact that $k^*\geq k^{**}$, there holds $\frac{\Tr(\Sigma_{J_*^c})}{N}\leq \frac{\Tr(\Sigma_{J_{**}^c})}{N}<\frac{1}{100}t^{-1}$. Therefore, $\frac{1}{N}\Tr(\Sigma_{J_{*}^c})+t^{-1}\leq \frac{1.01}{b}\sigma_{k^{*}}\leq \frac{1.01}{b}\sigma_j$ for any $j\in J_{*}$, hence
\begin{align}\label{eq:lower_variance_1}
    &\frac{1}{N}\Tr\left[\varphi_t\left(\frac{1}{N}\bX\bX^\top\right)\left(\bX_{J_*}\Sigma_{J_*}\bX_{J_*}^\top\right)\varphi_t\left(\frac{1}{N}\bX\bX^\top\right)\right] \geq \frac{cc_1^2b^2}{C'^4C^2(b+1.01)^2}|J_{*}|.
\end{align}

\paragraph{Lower bound for \texorpdfstring{$\frac{1}{N} \Tr\left[\varphi_t\left(\frac{1}{N}\bX\bX^\top\right)\left(\bX_{J^c}\Sigma_{J^c}\bX_{J^c}^\top\right)\varphi_t\left(\frac{1}{N}\bX\bX^\top\right)\right]$.}{Trace 2}}
By Assumption~\ref{ass:filter_fct}, $\varphi_t(x)\geq \oc{c_filter_1}(x+t^{-1})^{-1}$ and so we have
\begin{align*}
    &\frac{1}{N}\Tr\left[\varphi_t\left(\frac{1}{N}\bX\bX^\top\right)\left(\bX_{J^c}\Sigma_{J^c}\bX_{J^c}^\top\right)\varphi_t\left(\frac{1}{N}\bX\bX^\top\right)\right]\geq \oc{c_filter_1}^2\Tr\left(\frac{1}{N}\bX_{J^c}\Sigma_{J^c}\bX_{J^c}^\top\left(\frac{1}{N}\bX\bX^\top + t^{-1}I_N\right)^{-2}\right).
\end{align*}We now study a lower bound, in the Loewner order, for $\left( \frac{1}{N}\bX\bX^\top + t^{-1} I_N \right)^{-2}$.

On the event $\Omega_{\mathrm{lower}}$, we have $\frac{1}{N}\bX_{J^c}\bX_{J^c}^\top + t^{-1} I_N \preceq C\left(\frac{\Tr(\Sigma_{J^c})}{N}+t^{-1}\right)I_N$. Moreover, we also know that $\rank(\frac{1}{N}\bX_J\bX_J^\top)\leq |J|$, and
$\frac{1}{N}\bX\bX^\top + t^{-1}I_N = \frac{1}{N}\bX_{J^c}\bX_{J^c}^\top + t^{-1} I_N + \frac{1}{N}\bX_J\bX_J^\top$. That is, $\frac{1}{N}\bX\bX^\top + t^{-1}I_N$ is decomposed as $ \frac{1}{N}\bX_{J^c}\bX_{J^c}^\top + t^{-1} I_N$ and a rank $|J|$ perturbation. Let $\frac{1}{N}\bX\bX^\top + t^{-1}I_N = \sum_{i=1}^N \lambda_i \vv_i\otimes\vv_i$ be its spectral decomposition. Then we know that, almost surely, $\lambda_N>0$. Furthermore,
$\left( \frac{1}{N}\bX\bX^\top + t^{-1}I_N\right)^{-2} = \sum_{i=1}^{|J|} \lambda_i^{-2} \vv_i\otimes\vv_i + \sum_{i>|J|}^{N} \lambda_i^{-2} \vv_i\otimes\vv_i \succeq \sum_{i>|J|}^{N} \lambda_i^{-2} \vv_i\otimes\vv_i$. It follows from the interlacing property (see \cite[Weyl's inequality in Example 7.5.2]{meyer_matrix_2023}) that, on $\Omega_t$, for any $i\geq |J|+1$, 
$$\lambda_i \leq \lambda_{|J|+1} \leq \sigma_1\left(\frac{1}{N}\bX_{J^c}\bX_{J^c}^\top + t^{-1} I_N\right) \leq C\left(\frac{\Tr(\Sigma_{J^c})}{N}+t^{-1}\right).$$ Let $\mathrm{Proj}_{|J|+1:N} = \sum_{i>|J|}^N \vv_i\otimes\vv_i$. Then $\left( \frac{1}{N}\bX\bX^\top + t^{-1}I_N\right)^{-2}\succeq \frac{1}{C^2}\left(\frac{\Tr(\Sigma_{J^c})}{N}+t^{-1}\right)^{-2}\mathrm{Proj}_{|J|+1:N}$. In conclusion, we have
\begin{align}\label{eq:variance_lower_2_pre}
    &\oc{c_filter_1}^2\Tr\left(\frac{1}{N}\bX_{J^c}\Sigma_{J^c}\bX_{J^c}^\top\left(\frac{1}{N}\bX\bX^\top + t^{-1}I_N\right)^{-2}\right)\geq \frac{\oc{c_filter_1}^2}{C^2}\frac{\Tr\left(\frac{1}{N}\bX_{J^c}\Sigma_{J^c}\bX_{J^c}^\top\mathrm{Proj}_{|J|+1:N}\right)}{\left(t^{-1} + \frac{\Tr(\Sigma_{J^c})}{N}\right)^2}.
\end{align}We now study a lower bound for $\Tr\left(\frac{1}{N}\bX_{J^c}\Sigma_{J^c}\bX_{J^c}^\top\mathrm{Proj}_{|J|+1:N}\right)$. By $\Omega_{\mathrm{lower}}$, together with the fact that $\rank(I_N - \mathrm{Proj}_{|J|+1:N})\leq |J|$, there holds
\begin{align*}
    &\Tr\left(\frac{1}{N}\bX_{J^c}\Sigma_{J^c}\bX_{J^c}^\top\mathrm{Proj}_{|J|+1:N}\right) = \Tr\left(\frac{1}{N}\bX_{J^c}\Sigma_{J^c}\bX_{J^c}^\top\right) - \Tr\left(\frac{1}{N}\bX_{J^c}\Sigma_{J^c}\bX_{J^c}^\top\left( I_N - \mathrm{Proj}_{|J|+1:N}\right)\right)\\
    &\geq \left[ \mathfrak{L}(J^c,N) - \rank(I_N - \mathrm{Proj}_{|J|+1:N})\norm{\frac{1}{N}\bX_{J^c}\Sigma_{J^c}\bX_{J^c}^\top}_{\op}\right]_+ \geq \left[ \mathfrak{L}(J^c,N) - |J|\mathfrak{U}(J^c,N) \right]_+.
\end{align*}Now, for $J = J_{*} = [k^{*}]$ defined in \eqref{eq:def_estimation_dimension}, plug it back into \eqref{eq:variance_lower_2_pre}, we have
\begin{align}\label{eq:lower_variance_2}
    &\frac{1}{N}\Tr\left[\varphi_t\left(\frac{1}{N}\bX\bX^\top\right)\left(\bX_{J^c}\Sigma_{J^c}\bX_{J^c}^\top\right)\varphi_t\left(\frac{1}{N}\bX\bX^\top\right)\right]\geq \frac{\oc{c_filter_1}^2}{C^2}\frac{\left( \mathfrak{L}(J_{*}^c,N) - k^{**}\mathfrak{U}(J_{*}^c,N) \right)_+}{\left(t^{-1} + \frac{\Tr(\Sigma_{J_{*}^c})}{N}\right)^2}.
\end{align}
Combining \eqref{eq:lower_variacne_decomposition_bounded}, \eqref{eq:lower_variance_1} and \eqref{eq:lower_variance_2}, when $J = J_{*}$ and when $t^{-1}>100\frac{\Tr(\Sigma_{J_{**}^c})}{N}$, there holds
\begin{align*}
    \frac{\sigma_\xi^2}{4}\Tr[\Sigma\varphi_t^2(\hat\Sigma)\hat\Sigma] \geq \frac{cc_1^2b^2}{4C'^4C^2(b+1)^2} \sigma_\xi^2k^{*} + \sigma_\xi^2\frac{\oc{c_filter_1}^2}{4C^2} \frac{\left( \mathfrak{L}(J_{*}^c,N) - k^{*}\mathfrak{U}(J_{*}^c,N) \right)_+}{\left(t^{-1} + \frac{\Tr(\Sigma_{J_{*}^c})}{N}\right)^2}. 
\end{align*}Now, we use the following lemma.
\begin{Lemma}\label{lemma:positive_part}
  For any $x,y\geq0$, $A,B,D>0$, there holds
  \begin{align}\label{eq:positive_part}
      Ax+\left(By-Dx\right)_+ \geq \frac{AB}{A+B+D}(x+y).
  \end{align}
\end{Lemma}
\beginproof (Proof of Lemma~\ref{lemma:positive_part})
We separate into two cases.
\begin{enumerate}
    \item When $y\leq \frac{D}{B}x$. In this case, the left-hand-side of \eqref{eq:positive_part} is $Ax$. By $y\leq \frac{D}{B}x$, there holds $x+y\leq \frac{D+B}{B}x$, hence $Ax\geq \frac{AB}{B+D}(x+y)>\frac{AB}{A+B+D}(x+y)$, which is the right-hand-side of \eqref{eq:positive_part}.

    \item When $y>\frac{D}{B}x$, the left-hand-side of \eqref{eq:positive_part} is $(A-D)x+By$. We only need to prove that $h(x,y):=\frac{(A-D)x+By}{x+y}\geq \frac{AB}{A+B+D}$. If $x=y=0$, \eqref{eq:positive_part} holds trivially. If $x=0$, then from $y>\frac{D}{B}x$, we know $y>0$, therefore $h(0,y)=B>\frac{A}{A+B+D}B$. It is impossible in this case that $y=0$ but $x>0$. Therefore in the following, we assume $x,y>0$. Let $z=\frac{y}{x}$, and $g(z) = \frac{(A-D)+Bz}{1+z}$, where $z>\frac{D}{B}$. We only need to prove that $\inf(g(z):z>\frac{D}{B})\geq \frac{AB}{A+B+D}$. We have $g'(z) = \frac{B+D-A}{(1+z)^2}$.
    \begin{enumerate}
        \item When $B+D-A=0$. Then $g(z)$ is a constant, which equals to $B>\frac{A}{A+B+D}B$.
        \item When $B+D-A>0$. Then $\inf(g(z):z>\frac{D}{B}) = g(\frac{D}{B}) = \frac{AB}{B+D}>\frac{AB}{A+B+D}$.
        \item When $B+D-A<0$. Then $\inf(g(z):z>\frac{D}{B}) = \lim_{z\to\infty}g(z) = B>\frac{A}{A+B+D}B$.
    \end{enumerate}
\end{enumerate}In summary, \eqref{eq:positive_part} is verified.
\endproof

Applying Lemma~\ref{lemma:positive_part} to
\begin{align*}
    A = \frac{c\oc{c_filter_1}^2b^2}{4C'^4C^2(b+1.01)^2},\, B = \frac{\oc{c_filter_1}^2}{4C^2},\, D = \frac{\oc{c_filter_1}^2\mathfrak{U}(J_{*}^c,N)}{4C^2\left(t^{-1}+\frac{\Tr(\Sigma_{J_{*}^c})}{N}\right)^2}, x = \sigma_\xi^2 k^{*} \mbox{ and } y = \sigma_\xi^2\frac{\mathfrak{L}(J_{*}^c,N)}{\left(t^{-1}+\frac{\Tr(\Sigma_{J_{*}^c})}{N}\right)^2}.
\end{align*}
 By $\mathfrak{U}(J_{*}^c, N) = Ct^{-2}$, $D \lesssim 1$. Thus $A, B$ and $D$ are bounded by absolute constants, and $\frac{AB}{A+B+D}$ is bounded from below by an absolute constant $c_0>0$. We obtain
\begin{align*}
    \frac{\sigma_\xi^2}{4}\frac{\Tr[\Sigma\varphi_t^2(\hat\Sigma)\hat\Sigma]}{N}\geq 4c_0^2\sigma_\xi^2\frac{k^{*}}{N} + 4c_0^2\sigma_\xi^2\frac{\mathfrak{L}(J_{*}^c,N)}{N\left(t^{-1}+\frac{\Tr(\Sigma_{J_{*}^c})}{N}\right)^2}.
\end{align*}
So far, we have proved the following proposition:
\begin{Proposition}\label{prop:lower_variance_general_determinisitc}
    Grant the same assumptions as Theorem~\ref{theo:main_RKHS}.
    On $\Omega_{\mathrm{lower}}$, there exists an absolute constant $c''>0$ such that
    \begin{align*}
        \frac{\sigma_\xi^2}{4}\frac{\Tr[\Sigma\varphi_t^2(\hat\Sigma)\hat\Sigma]}{N}\geq 4c_0^2\sigma_\xi^2 \left( \frac{k^{*}}{N} + \frac{\mathfrak{L}(J_{*}^c,N)}{N\left(t^{-1}+\frac{\Tr(\Sigma_{J_{*}^c})}{N}\right)^2}\right).
    \end{align*}
\end{Proposition}
The variance term of the lower bound of Theorem~\ref{theo:main_RKHS} follows by using again $\frac{\Tr(\Sigma_{J_*^c})}{N}\leq\frac{\Tr(\Sigma_{J_{**}^c})}{N}<\frac{1}{100}t^{-1}$ together with Lemma~\ref{lemma:lower_trace_lower_bound_variance}

\paragraph{End of the proof of the lower bound of Theorem~\ref{theo:main_RKHS}.}
Combining \eqref{eq:lower_estimation_error_bounded_0}, \eqref{eq:lower_estimation_error_bounded}, Proposition~\ref{prop:lower_bias_weak_variance_general}, and Proposition~\ref{prop:lower_variance_general_determinisitc}, the proof of the lower bound in Theorem~\ref{theo:main_RKHS} is complete.

\section{Auxiliaries results}

We start with some results on the concentration of sum of independent sub-exponential variables. We first start with the definition of $\psi$-norm (see for instance, Chapter~1 in \cite{ChafaiGuedonLecuePajor2012}). Let $\psi$ be an Orlicz function. We define the Orlicz norm of a random variable $Z$ as 
\begin{equation*}
    \norm{Z}_\psi = \inf\left(c: \bE \psi(|Z|/c)\leq \psi(1)\right).
\end{equation*}Orlicz functions that are of particular interest to us are, for all $\alpha\geq1$, $\psi_\alpha:t\geq 0\to \exp(t^\alpha)-1$. It follows from Theorem~1.1.5 in \cite{ChafaiGuedonLecuePajor2012} that for all $\alpha\geq1$, there is equivalence between:
\begin{itemize}
    \item[(a)] there is a constant $K_1>0$ such that $\norm{Z}_{\psi_\alpha}\leq K_1$
    \item[(b)] there is a constant $K_2>0$ such that for all $p\geq \alpha$, $\norm{Z}_{L_p}\leq K_2 p^{1/\alpha}$
    \item[(c)] there exists $K_3, K_3^\prime$ such that for all $t\geq K_3^\prime$, with probability at least $1-\exp(-t^\alpha/K_3^\alpha), |Z|\leq t$.
\end{itemize}Moreover, $K_2\leq 2e K_1$, $K_3\leq e K_2$, $K_3^\prime\leq e^2 K_2$ and $K_1\leq 2\max(K_3, K_3^\prime)$. It follows from these equivalence that
\begin{equation*}
    \norm{Z}_{\psi_\alpha} \sim \sup_{p\geq \alpha} \frac{\norm{Z}_{L_p}}{p^{1/\alpha}}.
\end{equation*}In particular, if $X$ is a sub-gaussian vector as defined in Assumption~\ref{assumption:main_upper} then there exists some absolute constant $C>0$ such that for all $\vv\in\bR^p$, $\norm{\inr{X, \vv}}_{\psi_2}\leq C \norm{\Sigma^{\frac12}\vv}_2$. It is also clear from the definition of the $\psi_1$ and $\psi_2$ norm that for all $\vv$ we have $\norm{\inr{X, \vv}^2}_{\psi_1} = \norm{\inr{X, \vv}}_{\psi_2}^2\leq C^2 \norm{\Sigma^{\frac12}\vv}_2^2$. Finally, the last tool we need is Bernstein's inequality for the sum of independent $\psi_1$ variable (see for instance Theorem~1.2.7 in \cite{ChafaiGuedonLecuePajor2012}): if $Z_1, \ldots, Z_N$ are independent $\psi_1$ random variables then for all $t\geq1$, with probability at least $1-\exp(-ct)$,
\begin{equation*}
    \left|\frac{1}{N}\sum_{i=1}^N Z_i-\bE Z_i\right|\leq \sigma_1 \sqrt{\frac{t}{N}} + M_1 \frac{t}{N}
\end{equation*}where $M_1 = \max_{1\leq i\leq N}\norm{Z_i-\bE Z_i}_{\psi_1}$ and $\sigma_1^2 = (1/N)\sum_{i=1}^N \norm{Z_i-\bE Z_i}_{\psi_1}^2$. In particular, if we apply this result for $Z_i = \inr{X_i, \vv}^2$ (which is a $\psi_1$ random variable according to the argument above), we get that with probability at least $1-\exp(-ct)$,
\begin{equation}\label{eq:Bernstein_ineq_psi1}
    \left|\frac{1}{N}\sum_{i=1}^N \inr{X_i, \vv}^2-\bE \inr{X_i, \vv}^2\right|\leq \norm{\inr{X, \vv}^2 - \bE\inr{X, \vv}^2}_{\psi_1}\left( \sqrt{\frac{t}{N}} + \frac{t}{N}\right)
\end{equation}and  
\begin{equation*}
    \norm{\inr{X, \vv}^2 - \bE\inr{X, \vv}^2}_{\psi_1}\leq \norm{\inr{X, \vv}^2}_{\psi_1} + \norm{\bE\inr{X, \vv}^2}_{\psi_1}\leq \norm{\inr{X, \vv}}_{\psi_2}^2 + \norm{1}_{\psi_1} \bE\inr{X, \vv}^2\leq C \norm{\Sigma^{\frac12}\vv}_2^2
\end{equation*}where we used the subgaussian property of $X$. As a consequence, we proved the following result.

\begin{Lemma}\label{lem:sum_squares}
    There is some absolute constant $c>0$ such that the following holds. Let $X$ be a sub-gaussian vector in $\bR^p$ and denote $\Sigma=\bE X X^\top$ ($X$ is not necessarily centered).  Let $\vv\in\bR^p$. With probability at least $1-\exp(-cN)$, we have
\begin{equation}\label{eq:concentration_sum_square_psi2}
    \frac{1}{2}\norm{\Sigma^{\frac12}\vv}_2^2  \leq \frac{1}{N}\sum_{i=1}^N \inr{X_i, \vv}^2 \leq \frac{3}{2}\norm{\Sigma^{\frac12}\vv}_2^2. 
\end{equation}
\end{Lemma}

Next we use the classical generic chaining bound for sub-gaussian processes that follows from Theorem~2.2.27 in \cite{talagrand_majorizing_1996}. Note that the following result requires less assumptions than the one required in Hanson-Wright inequality from Theorem~6.2.1 in \cite{vershynin_high-dimensional_2018}.

\begin{Lemma}\label{lem:HW_better}
There is an absolute constant $c>0$ such that the following holds.
    Let $X$ be a sub-gaussian vector in $\bR^p$ and denote $\Sigma=\bE X X^\top$ ($X$ is not necessarily centered). Let $A$ be a matrix in $\bR^{p\times d}$. We have for all $t>0$, with probability at least $1-\exp(-t)$,
    \begin{align*}
       \norm{AX}_2 \leq c\left(\norm{\Sigma^{1/2}A^\top}_{HS} + \norm{\Sigma^{1/2}A^\top}_{op}\sqrt{t}\right).
    \end{align*}We also have
    \begin{equation*}
        \bE \norm{AX}_2^2  = \norm{\Sigma^{1/2}A^\top}_{HS}^2.
    \end{equation*}
\end{Lemma}
\begin{proof}
We first note that 
\begin{equation*}
    \norm{AX}_2 \leq \norm{A(X-\bE X)}_2 + \norm{A\bE X}_2. 
\end{equation*}Then, we write $\norm{A(X-\bE X)}_2$ as the supremum of a centered sub-gaussian process: $\norm{A(X-\bE X)}_2 = \sup(Z_x: x \in A^\top B_2^d)$ where $Z_x=\inr{X-\bE X, x}$. The canonical metric associated with this process is $(u,v) \to \left(\bE(Z_u - Z_v)^2\right)^{1/2} = \norm{\Sigma_0^{1/2}(u-v)}_2$ where $\Sigma_0=\bE\left[(X-\bE X)(X-\bE X)^\top\right]$. It follows from Theorem~2.2.27 in \cite{talagrand_majorizing_1996}, that for all $t>0$, with probability at least $1-\exp(-t)$, $\norm{A(X-\bE X)}_2\lesssim \gamma_2+ \sqrt{t} D$ where $\gamma_2 = \gamma_2(\Sigma_0^{1/2}A^\top B_2^d, \ell_2^p)$ is Talagrand's $\gamma_2$-functional and $D$ is the diameter of $\Sigma_0^{1/2}A^\top B_2^d$ with respect to $\ell_2^p$. It follows from Talagrand's majorizing measure that
\begin{equation*}
    \gamma_2(\Sigma_0^{1/2}A^\top B_2^d, \ell_2^p) \lesssim \bE \norm{\Sigma_0^{1/2}A^\top G}_2\lesssim \Tr[A\Sigma_0A^\top]^{1/2} = \norm{\Sigma_0^{1/2}A^\top}_{HS}
\end{equation*}and $D= \norm{\Sigma_0^{1/2}A^\top}_{op}$. We conclude the proof of the exponential bound by using that $\norm{A\bE X}_2 + \norm{\Sigma_0^{1/2}A^\top}_{HS} \lesssim \norm{\Sigma^{1/2}A^\top}_{HS}$. The result in expectation follows from $\bE \norm{AX}_2^2 = \Tr[\bE[AXX^\top A^\top]]  = \norm{\Sigma^{1/2}A^\top}_{HS}^2$.
\end{proof}
Under the same assumptions as in Lemma~\ref{lem:HW_better}, we get that for all $t\geq \norm{\Sigma^{1/2}A^\top}_{HS}$ with probability at least $1-\exp\left(-c t^2 / \norm{\Sigma^{1/2}A^\top}_{op}^2\right)$, $\norm{AX}_2\leq t$. The latter statement coincides with point \textit{(c)} above for $\alpha=2$, $K_3 \sim \norm{\Sigma^{1/2}A^\top}_{op}$ and $K_3^\prime\sim \norm{\Sigma^{1/2}A^\top}_{HS}$ meaning that $\norm{AX}_2$ is a subgaussian variable with subgaussian norm satisfying
\begin{equation*}
    \norm{\norm{AX}_2}_{\psi_2}\lesssim \norm{\Sigma^{1/2}A^\top}_{HS}.
\end{equation*}This follows from the equivalence between \textit{(a)} and \textit{(c)} above. As a consequence, $\norm{\norm{AX}_2^2}_{\psi_1}\lesssim \norm{\Sigma^{1/2}A^\top}_{HS}^2 = \bE\norm{A X}_2^2$ and so it follows from \eqref{eq:Bernstein_ineq_psi1} that for all $t>0$, with probability at least $1-\exp(-ct)$, 
\begin{equation*}
    \left|\frac{1}{N}\sum_{i=1}^N \norm{A X_i}_2^2 - \bE\norm{A X}_2^2 \right|\leq c\bE\norm{A X}_2^2 \left(\sqrt{\frac{t}{N}} + \frac{t}{N}\right).
\end{equation*}(Note that this results holds even if $X$ is not centered and does not have independent coordinates unlike the Hansen-Wright inequality from Theorem~6.2.1 in \cite{vershynin_high-dimensional_2018}). For $t\sim N$ we just proved the following result.

\begin{Lemma}\label{lem:sum_norm_square_psi_2}
    There exists an absolute constant $c>0$ such that the following holds. Let $X$ be a sub-gaussian vector in $\bR^p$ and denote $\Sigma=\bE X X^\top$ ($X$ is not necessarily centered). Let $A$ be a matrix in $\bR^{p\times d}$.
    With probability at least $1-\exp(-cN)$, 
    \begin{equation*}
        \frac{1}{2} \norm{\Sigma^{1/2}A^\top}_{HS}^2 \leq \frac{1}{N}\sum_{i=1}^N \norm{A X_i}_2^2 \leq \frac{3}{2} \norm{\Sigma^{1/2}A^\top}_{HS}^2.
    \end{equation*}
\end{Lemma}


\subsection{Proof of Corollary~\ref{coro:Sobo}}\label{sec:proof_Coro_Sobo}

By the proof of Proposition 7 of \cite{gavrilopoulos_geometrical_2025}, if $t^{-1}\sim N^{-\frac{\alpha }{1+s\alpha}}$, regardless of the relationship between $s$ and $2$, we always have
\begin{align*}
    \sigma_\xi^2\frac{|J_*|}{N} + \sigma_\xi^2\frac{N\Tr(\Sigma_{J_*^c}^2)}{(Nt^{-1})^2}\sim \sigma_\xi^2 N^{-\frac{\alpha s}{1+\alpha s}},\quad \mbox{ and }\quad\|\Sigma_{J_*^c}^{\frac{1}{2}}\bbeta_{J_*^c}^*\|_2^2\sim N^{-\frac{\alpha s}{1+\alpha s}}.
\end{align*}The difference is that for ridge, $\psi_t^{(B)}(x)=\frac{1}{xt+1}$, hence by the proof of Proposition 7 of \cite{gavrilopoulos_geometrical_2025},
\begin{align*}
    \norm{\Sigma_{J_*}^{\frac{1}{2}}\psi_t^{(B)}(\Sigma)\bbeta_{J_*}^*}_2^2 \sim N^{-\frac{\alpha \tilde s}{1+\alpha\tilde s}}.
\end{align*}On the other hand, by Definition~\ref{assumption:main_upper}, item \emph{2.}, we have
\begin{align*}
    &\norm{\Sigma_{J_*}^{\frac{1}{2}}\psi_t^{(A)}(\Sigma)\bbeta_{J_*}^*}_2^2 = \sum_{j\leq k_{t^{-1},b}^*}\sigma_j^s(\psi_t^{(A)}(\sigma_j))^2\sigma_j^{1-s}\langle\bbeta^*,\ve_j\rangle^2 \leq \oC{C_filter_2}^2t^{-s}\norm{\Sigma_{J_*}^{\frac{1-s}{2}}\bbeta_{J_*}^*}_2^2\lesssim N^{-\frac{\alpha s}{1+\alpha s}}.
\end{align*}
As the choice of $t$ is optimal over the class $\fR_{\mathrm{Sob}}(s,\alpha)$ (see, for instance, \cite{li_generalization_2024}), we conclude that $\{\varphi^{(A)}\}_{t\geq1}\preceq_{\cR}\{\varphi^{(B)}\}_{t\geq1}$ for any $\cR\in\fR_{\mathrm{Sob}}(s,\alpha)$.




\subsection{Proof of Corollary~\ref{coro:saturation}}\label{sec:proof_coro_saturation}

For any $t$ in the interval $I = \{t: b^{-1}\varepsilon\leq t^{-1}<\sigma\}$, it is easy to verify that $k_{t^{-1},b}^* = k$. Moreover, since we have assumed that for any $1\leq j\leq k$, there holds $|\langle\bbeta^*,\ve_j\rangle|=\alpha_*$, and for any $j>k$, $\langle\bbeta^*,\ve_j\rangle=0$, we have $\|\Sigma_{J_*^c}^{1/2}\bbeta_{J_*^c}^*\|_2=0$. Moreover, $\|\Sigma_{J_*}^{1/2}\psi_t(\Sigma)\bbeta_{J_*}^*\|_2 = (\sum_{j\leq k}\sigma\psi_t^2(\sigma)\alpha_*^2)^{1/2} = \alpha_*\psi_t(\sigma)\sqrt{k\sigma}$, and $\sigma_\xi t\sqrt{\Tr(\Sigma_{J_*^c}^2)/N}=\sigma_\xi\varepsilon t\sqrt{(p-k)/N}$. We compute that
\begin{align*}
    R = \frac{\alpha_*}{\sigma_\xi}\frac{\sigma^{3/2}}{\varepsilon}\sqrt{\frac{kN}{p-k}}.
\end{align*}
\begin{enumerate}
    \item When $\psi_t(x)=\psi_t^{(\mathrm{Ridge})}(x) = \frac{1}{xt + 1}$. Then
    \begin{align*}
        \min_{t\in I} r^{(\mathrm{Ridge})}(V_{J_*},V_{J_*^c}) = \sigma_\xi\sqrt{\frac{k}{N}} + \min_{t\in I}\left( \sigma_\xi\varepsilon t\sqrt{\frac{p-k}{N}} + \alpha_*\frac{\sqrt{k\sigma}}{\sigma t + 1} \right).
    \end{align*}Under the assumption that
    \begin{align*}
        4< \frac{\alpha_*}{\sigma_\xi}\frac{\sigma^{3/2}}{\varepsilon}\sqrt{\frac{kN}{p-k}}< \frac{\sigma}{\varepsilon}b \leq \big(1+\frac{\sigma}{\varepsilon}b\big)^2,
    \end{align*}the minimum is given by
    \begin{align}\label{eq:optimal_ridge}
        \min_{t\in I}r^{(\mathrm{Ridge})}(V_{J_*},V_{J_*^c}) = \sigma_\xi\sqrt{\frac{k}{N}} + \frac{\sigma_\xi}{\sigma}\varepsilon\sqrt{\frac{p-k}{N}}\left(2\sqrt{R}-1\right).
    \end{align}

    \item When $\psi_t(x)=\psi_t^{(\mathrm{GF})}(x) = \exp(-tx)$. Then
    \begin{align*}
        \min_{t\in I} r^{(\mathrm{GF})}(V_{J_*},V_{J_*^c}) = \sigma_\xi\sqrt{\frac{k}{N}} + \min_{t\in I}\left( \sigma_\xi\varepsilon t\sqrt{\frac{p-k}{N}} + \alpha_* \sqrt{k\sigma}\exp(-t\sigma) \right).
    \end{align*}Under the assumption that
    \begin{align*}
        e<\frac{\alpha_*}{\sigma_\xi}\frac{\sigma^{3/2}}{\varepsilon}\sqrt{\frac{kN}{p-k}}< \frac{\sigma}{\varepsilon}b\leq \exp\bigg(\frac{\sigma}{\varepsilon}b\bigg),
    \end{align*}the minimum is given by
    \begin{align}\label{eq:optimal_GF}
        \min_{t\in I}r^{(\mathrm{GF})}(V_{J_*},V_{J_*^c}) = \sigma_\xi\sqrt{\frac{k}{N}} + \frac{\sigma_\xi}{\sigma}\varepsilon\sqrt{\frac{p-k}{N}}\left(1+\log(R)\right).
    \end{align}
\end{enumerate}
Combining \eqref{eq:optimal_ridge} and \eqref{eq:optimal_GF} and using the fact that $1+\log(R)\leq 2\sqrt{R}-1$ for any $R\geq 1$, we know that 
\begin{align*}
    \min_{t\in I}r^{(\mathrm{GF})}(V_{J_*},V_{J_*^c})\leq \min_{t\in I} r^{(\mathrm{Ridge})}(V_{J_*},V_{J_*^c}).
\end{align*}
Moreover, when $R\to\infty$, $\{\varphi_t^{(\mathrm{Ridge})}\}_{t\in I}\prec_{\cR}\{\varphi_t^{(\mathrm{GF})}\}_{t\in I}$.

\subsection{Definition of the  contour \texorpdfstring{$\cC_t$}{C} and proof of Lemma~\ref{lemma:spectral_calculus_li}}\label{sec:choice_contour}
In this section, we construct the family of  contours $(\cC_t)_{t\geq1}$ used in the formulae \eqref{eq:residual_theorem}. This formulae follows from the residue theorem, but, in order to apply this theorem, we need the contour $\cC_t$ to surround both spectra of $\Sigma$ and $\hat \Sigma$. By definition, the spectrum of $\Sigma$ lies in $[0,\sigma_1]$ and the one of $\hat\Sigma$ lies in $[0, \hat \sigma_1]$. Moreover,  thanks to Lemma~\ref{lem:largest_sing_Hat_Sigma}, we know that on $\Omega_t$, we have $\hat\sigma_1\leq 4(\sigma_1+t^{-1})$. As a consequence, formulae \eqref{eq:residual_theorem} is valid on $\Omega_t$ if we construct a contour $\cC_t$ in such a way that it contains $[0, 4(\sigma_1+t^{-1})]$. Moreover, we also need to choose $\cC_t$ so that Lemma~\ref{lemma:spectral_calculus_li} and \ref{lemma:operator_norm_Sigma_J} hold on $\Omega_t$.

We follow \cite{li_generalization_2024} for the construction of such a contour: for all $t\geq1$, define $\mathcal{C}_{t} =\mathcal{C}_{t,1} \cup \mathcal{C}_{t,2} \cup \mathcal{C}_{t,3}$ where $\cC_{t,k}, k=1,2,3$ are defined now. We let $L:x\in\bR\to\alpha x + \beta$, where
\begin{align*}
    \alpha = \frac{5(\sigma_1+t^{-1})}{\sigma_1+t^{-1}/2},\,\mbox{ and } \beta = \frac{\alpha}{2t}.
\end{align*}Note that $L(-1/(2t))=0$ and $L(\sigma_1)= 5(\sigma_1+t^{-1})$ so that by setting
\begin{equation}\label{contour}
   \begin{aligned}
\mathcal{C}_{t,1} & =\left\{x + L(x) i: x \in[-1/(2t), \sigma_1]\right\}, \\
\mathcal{C}_{t,2} & =\left\{x - L(x) i : x \in[-1/(2t), \sigma_1]\right\}, \\
\mathcal{C}_{t,3} & =\left\{z \in \mathbb{C}: | z-\sigma_1 |=5(\sigma_1+t^{-1}) , \operatorname{Re}(z) \geq \sigma_1\right\},
\end{aligned} 
\end{equation}the union $\cup_{k=1,2,3}\cC_{t,k}$ is well defining a contour in $\mathbb{C}$; this is the one we call $\cC_t$ depicted in Figure~\ref{fig:contour}.


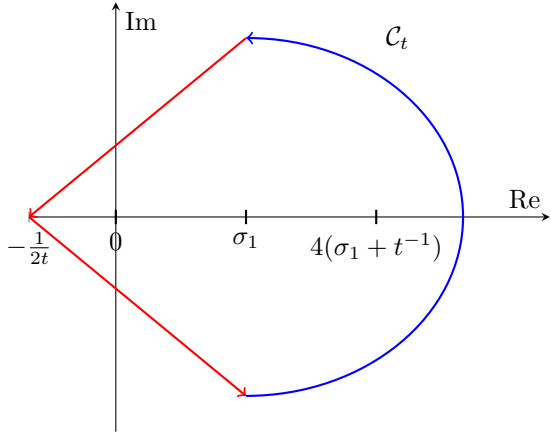
\begin{figure}[!ht]
    \centering
\begin{tikzpicture}
    \begin{axis}[
        axis lines = middle,
        xlabel = {Re},
        ylabel = {Im},
        xmin=-1, xmax=5,
        ymin=-3, ymax=3,
        xtick={-1,0},
        ytick = \empty,
        xticklabels={$-\frac{1}{2t}$,$0$},
    ]
    \draw[<-, thick, color=blue] (1.5,2.5) arc (90:-90:2.5);
    \draw[<-, thick, color=red] (-1,0) -- (1.5,2.5) ;
    \draw[->, thick, color=red] (-1,0) -- (1.5,-2.5);
    \draw[thick] (0,-0.1) -- (0,0.1) node[below=5pt] {$0$};
    \draw[thick] (1.5,-0.1) -- (1.5,0.1) node[below=5pt] {$\sigma_1$};
    \draw[thick] (3,-0.1) -- (3,0.1) node[below=5pt] {$4(\sigma_1+t^{-1})$};
    \node at (3.25,2.5) {$\cC_t$};
    \end{axis}
\end{tikzpicture}
\caption{The contour $\mathcal{C}_{t}$ defined in \eqref{contour} surrounds both spectra of $\Sigma$ and of $\hat\Sigma$ on $\Omega_t$ since, on that event, $\hat\sigma_1\leq 4(\sigma_1+t^{-1})$ thanks to Lemma~\ref{lem:largest_sing_Hat_Sigma}.}\label{fig:contour}
\end{figure}



\subsubsection{Proof of Lemma~\ref{lemma:spectral_calculus_li}}\label{sec:proof_Lemma_Li}

\beginproof Let $z\in\cC_t$. We first show that $\left\|\Sigma_t^{\frac{1}{2}}\left(\hat\Sigma-zI_p\right)^{-1} \Sigma_t^{\frac{1}{2}}\right\|_{\op} \leq 3 C$. To that end, we first bound $\|\hat\Sigma_{t}^{\frac{1}{2}}(\hat\Sigma-zI_p)^{-1} \hat\Sigma_{t}^{\frac{1}{2}}\|_{\op}$ from above and then we will conclude using Lemma~\ref{lemma:operator_norm_Sigma_J}.
Using SVD, we have
\begin{align*}
    \left\|\hat\Sigma_{t}^{\frac{1}{2}}(\hat\Sigma-zI_p)^{-1} \hat\Sigma_{t}^{\frac{1}{2}}\right\|_{\op}=\sup _{\sigma \in \sigma(\hat\Sigma)}\left|\frac{\sigma+t^{-1}}{\sigma-z}\right|
\end{align*}where $\sigma(\hat\Sigma)$ denotes the spectrum of $\hat\Sigma$. We recall that $\hat\sigma_1$ denotes the largest singular values of  $\hat\Sigma_1$ so that $\sigma(\hat\Sigma)\subset [0, \hat \sigma_1]$. Moreover, by Lemma~\ref{lem:largest_sing_Hat_Sigma}, $\hat{\sigma}_1 < 4(\sigma_1 + t^{-1})$ on $\Omega_t$. As a consequence, on $\Omega_t$,
\begin{equation*}
    \left\|\hat\Sigma_{t}^{\frac{1}{2}}(\hat\Sigma-zI_p)^{-1} \hat\Sigma_{t}^{\frac{1}{2}}\right\|_{\op}\leq \sup _{0\leq \sigma\leq 4(\sigma_1 + t^{-1})}\left|\frac{\sigma+t^{-1}}{\sigma-z}\right|.
\end{equation*}
We are now considering two cases: either $z$ belongs to the 'linear' section $\mathcal{C}_{t, 1} \cup \mathcal{C}_{t, 2}$ of the contour $\cC_t$ or to the semi-circle section $\cC_{t,3}$, see the definitions in \eqref{contour}. We start with the linear section.

\textbf{First case, when $z=x \pm L(x) i \in \mathcal{C}_{t, 1} \cup \mathcal{C}_{t, 2}$}, where $x \in\left[-1 / (2t), \sigma_1\right]$, we get
\begin{align*}
    \sup _{\sigma \in \sigma(\hat\Sigma)}\left|\frac{\sigma+t^{-1}}{\sigma-z}\right|^2 \leq \sup _{\sigma \geq 0}\left|\frac{\sigma+t^{-1}}{\sigma-z}\right|^2.
\end{align*}
Let $y = \sigma + t^{-1}$, $B = x+t^{-1}$, and $ C = B^2 + L(x)^2$. Then $|\sigma-z|^2 = (\sigma-x)^2 + L(x)^2 = (y-B)^2+C-B^2$, thus
\begin{align*}
    \left|\frac{\sigma+t^{-1}}{\sigma-z}\right|^2 = \frac{y^2}{y^2 - 2By + C}.
\end{align*}The function $y\mapsto \frac{y^2}{y^2 - 2By + C}$ is maximized at $y = \max\{\frac{C}{B},t^{-1}\}$.
Therefore when $\frac{C}{B}>t^{-1}$, we have the maximum $\frac{C}{C-B^2}$, otherwise we have the maximum when $y = t^{-1}$, when $\sigma = 0$. Solving $t^{-1}=\frac{C}{B}$ gives $x_0 = -\frac{1}{2t}+\frac{1}{2t\sqrt{1+\alpha^2}}$.
\begin{itemize}
    \item     If $\frac{C}{B}>t^{-1}$, combined with $x > -\frac{1}{2t}$, we have $ x > -\frac{2-\sqrt{2}}{4}t^{-1}$, and the maximum is given by $\frac{C}{C-B^2} = 1+\frac{(x+t^{-1})^2}{\alpha^2(x+\frac{1}{2t})^2}$. Let $\delta = tx$, then
\begin{align*}
    &\sup _{\sigma \geq 0}\left|\frac{\sigma+t^{-1}}{\sigma-z}\right|^2 = \sup\left( 1+\frac{1}{\alpha^2}\frac{(\delta+1)^2}{(\delta+\frac{1}{2})^2}:\, -\frac{1}{2} \leq \delta \leq t\sigma_1\right).
\end{align*}One may show that the maximum is achieved when $\delta = tx_0$, and
\begin{align*}
    &\sup _{\sigma \geq 0}\left|\frac{\sigma+t^{-1}}{\sigma-z}\right|^2 = 2 + \frac{2}{\alpha^2}\bigg(1+\sqrt{1+\alpha^2}\bigg),\quad \mbox{ where } \alpha = \frac{5(\sigma_1+t^{-1})}{\sigma_1+t^{-1}/2}.
\end{align*}
\item Else, the maximum is given by
\begin{align*}
    \sup _{\sigma \geq 0}\left|\frac{\sigma+t^{-1}}{\sigma-z}\right|^2 = \frac{t^{-2}}{x^2 + \alpha^2\left(x+\frac{1}{2t}\right)^2} \leq \frac{5(1+\alpha^2)}{\alpha^2}.
\end{align*}
\end{itemize}
As a consequence, when $z \in \mathcal{C}_{t, 1} \cup \mathcal{C}_{t, 2}$, we have
$$
\sup _{\sigma \in \sigma(\hat\Sigma)}\left|\frac{\sigma+t^{-1}}{\sigma-z}\right|^2 \leq 8.
$$

\textbf{Second case, when $z \in \mathcal{C}_{t, 3}$.} We have $|\sigma-z| \geq 2\sigma_1 + t^{-1}$ for $\sigma \in \sigma(\hat{\Sigma}) \subseteq\left[0, \hat{\sigma}_1\right]$, so, on $\Omega_t$, it follows from Lemma~\ref{lem:largest_sing_Hat_Sigma} that $\hat{\sigma}_1 < 4(\sigma_1 + t^{-1})$ and so
$$
\sup _{\sigma \in \sigma(\hat{\Sigma})}\left|\frac{\sigma+t^{-1}}{\sigma-z}\right| \leq \frac{4\sigma_1+5t^{-1}}{2\sigma_1 + 5t^{—1}} \leq 5.
$$Recall that from Lemma~\ref{lemma:operator_norm_Sigma_J},
\begin{align*}
    \|\Sigma_t^{-\frac{1}{2}}\hat{\Sigma}_t^{\frac{1}{2}}\|_{\op}^2 &\le 2,\quad\mbox{ and }\quad\|\Sigma_J^{\frac{1}{2}}\hat{\Sigma}_t^{-\frac{1}{2}}\|_{\op}^2 \le 2.
\end{align*}
The upper bound of $\|\Sigma_J^{\frac{1}{2}}\left(\hat\Sigma-zI_p\right)^{-1} \Sigma_J^{\frac{1}{2}}\|_{\op}$ is given by:
$$
\left\|\Sigma_J^{\frac{1}{2}}\left(\hat\Sigma-zI_p\right)^{-1} \Sigma_J^{\frac{1}{2}}\right\|_{\op} < \left\|\Sigma_J^{\frac{1}{2}}\hat{\Sigma}_t^{-\frac{1}{2}}\right\|_{\op}\left\|\hat\Sigma_{t}^{\frac{1}{2}}(\hat\Sigma-zI_p)^{-1} \hat\Sigma_{t}^{\frac{1}{2}}\right\|_{\op}\left\|\Sigma_J^{\frac{1}{2}}\hat{\Sigma}_t^{-\frac{1}{2}}\right\|_{\op}< 3C,
$$for some absolute constant $C>1$.

The upper bound for $ \|\Sigma_t^{\frac{1}{2}}(\Sigma-zI_p)^{-1} \Sigma_t^{\frac{1}{2}}\|_{\op}$ is similar  but simpler since
\begin{align*}
    &\left\|\Sigma_t^{\frac{1}{2}}(\Sigma-zI_p)^{-1} \Sigma_t^{\frac{1}{2}}\right\|_{\op}=\sup _{\sigma \in \sigma(\Sigma)}\left|\frac{\sigma + t^{-1}}{\sigma-z}\right|
\end{align*} and $\sigma(\Sigma)\subset [0, \sigma_1]$, so we omit it.

Finally, we move to the integral of the holomorphic extensions of the filter and residual functions. We have 
$$
\oint_{\mathcal{C}_{t}}\left|\varphi_{t}(z) \mathrm{d} z\right| \leq C \oint_{\mathcal{C}_{t}} \frac{1}{|z+t^{-1}|}|\mathrm{d} z| .
$$
Now we focus on the latter integral. For $z \in \mathcal{C}_{t,1}$, we have $|z+t^{-1}| \geq \sqrt{17}t^{-1}$ and thus
$$
\int_{\mathcal{C}_{t,1}} \frac{1}{|z+t^{-1}|}|\mathrm{d} z| \leq \frac{1}{\sqrt{17}} t^{-1}\left|\mathcal{C}_{t,1}\right| \leq C
$$for some absolute constant $C>1$, where we notice that $\left|\mathcal{C}_{t,1}\right| \leq C t^{-1}$. For $\mathcal{C}_{t,2}$, we have
\begin{align*}
    \int_{\mathcal{C}_{t,2}} \frac{1}{|z+t^{-1}|}|\mathrm{d} z| =2 \int_0^{\sigma_1} \frac{1}{|x+(x+t^{-1} / 2) i+t^{-1}|} \sqrt{2} \mathrm{~d} x  \leq C \int_0^{\sigma_1} \frac{1}{x+t^{-1}} \mathrm{d} x   \leq C \log(t),
\end{align*}where we have used that assumption that $\sigma_1$ is at most a constant.
For $z \in \mathcal{C}_{t,3}$, we have $|z+t^{-1}| \geq \sqrt{17}(\sigma_1+t^{-1})$ and thus
$$
\int_{\mathcal{C}_{t,3}} \frac{1}{|z+t^{-1}|}|\mathrm{d} z| \leq \frac{1}{\sqrt{17}(\sigma_1+t^{-1})}\left|\mathcal{C}_{t,3}\right| \leq C,
$$for some absolute constant.

\endproof

\subsection{Proof of Lemma~\ref{lemma:operator_norm_Sigma_J}}\label{sec:proof_operator_norm_Sigma_J}

\beginproof On the event $\Omega_t$, we have
\begin{equation*}
\begin{aligned}
\left\|\Sigma_t^{-\frac{1}{2}} \hat\Sigma_t^{\frac{1}{2}}\right\|_{\op}^2 & =\left\|\Sigma_t^{-\frac{1}{2}} \hat\Sigma_t \Sigma_t^{-\frac{1}{2}}\right\|_{\op}=\left\|\Sigma_t^{-\frac{1}{2}}\left(\hat{\Sigma}+t^{-1} I\right) \Sigma_t^{-\frac{1}{2}}\right\|_{\op} =\left\|\Sigma_t^{-\frac{1}{2}}\left(\hat{\Sigma}-\Sigma+\Sigma+t^{-1} I\right) \Sigma_t^{-\frac{1}{2}}\right\|_{\op} \\
& \leq \left\|\Sigma_t^{-\frac{1}{2}}\left(\hat{\Sigma}-\Sigma\right) \Sigma_t^{-\frac{1}{2}}\right\|_{\op} + 1 \leq \square + 1 \leq 2.
\end{aligned}
\end{equation*}
Let us now move to the first statemant of Lemma~\ref{lemma:operator_norm_Sigma_J}. On the event $\Omega_t$ we have $\norm{ \Sigma_t^{-1/2} (\hat\Sigma - \Sigma) \Sigma_t^{-1/2} }_{\op}\leq \square<1$ as a consequence, we have on that event
\begin{equation*}
    \left\|\left(I-\Sigma_t^{-\frac{1}{2}}\left(\hat\Sigma-\Sigma\right) \Sigma_t^{-\frac{1}{2}}\right)^{-1}\right\|_{\op} \leq \sum_{k=0}^{\infty} \left\|\Sigma_t^{-\frac{1}{2}}\left(\hat\Sigma-\Sigma\right) \Sigma_t^{-\frac{1}{2}}\right\|_{\op}^k \leq \sum_{k=0}^{\infty}\square^k \leq 2
\end{equation*}because $\square\leq1/2$. Next, using \eqref{eq:compare_Sigma_t_and_Sigma_J_Jc}, we observe that
\begin{equation*}
\begin{aligned}
&\left\|\Sigma_J^{\frac{1}{2}} \hat\Sigma_t^{-\frac{1}{2}}\right\|_{\op}^2 \leq \norm{\Sigma_J^{\frac{1}{2}}\Sigma_t^{-\frac{1}{2}}}_{\op}^2\norm{\Sigma_t^{\frac{1}{2}}\hat\Sigma_t^{-\frac{1}{2}}}_{\op}^2 \leq \norm{\Sigma_t^{\frac{1}{2}}\hat\Sigma_t^{-\frac{1}{2}}}_{\op}^2 =\left\|\Sigma_t^{\frac{1}{2}} \hat\Sigma_t^{-1} \Sigma_t^{\frac{1}{2}}\right\|_{\op}\\
&=\left\|\left(\Sigma_t^{-\frac{1}{2}} \hat\Sigma_t \Sigma_t^{-\frac{1}{2}}\right)^{-1}\right\|_{\op}= \left\|\left(I-\Sigma_t^{-\frac{1}{2}}\left(\hat\Sigma-\Sigma\right) \Sigma_t^{-\frac{1}{2}}\right)^{-1}\right\|_{\op}.
\end{aligned}
 \end{equation*}
\endproof

\subsection{Examples the assumption of Theorem~\ref{theo:main_RKHS} holds}\label{sec:examples_theo_main_RKHS}

We prove a stronger statement: there exists $\gamma>1$ such that $\sum_{j\in J_*^c}\sigma_j^2f_j^2(X) \leq \gamma\Tr(\Sigma_{J_*^c}^2)$ almost surely, where we abuse the notation to write $\Tr(\Sigma_{J_*^c}^2) = \sum_{j\in J_*^c}\sigma_j^2$. It is easy to verify that this implies the original small-ball probability lower bound. Indeed, for any nonnegative random variable $Z$ satisfying $\bE[Z]>0$, if $Z\leq\gamma\bE[Z]$ almost surely, then $\bE[Z] = \bE[Z\1_{Z<\frac{1}{10}\bE[Z]}] + \bE[Z\1_{Z\geq \frac{1}{10}\bE[Z]}] \leq \frac{1}{10}\bE[Z] + \gamma\bE[Z]\bP(Z\geq\frac{1}{10}\bE[Z])$. Since $\bE[Z]>0$, we have $\bP(Z\geq\frac{1}{10}\bE[Z])\geq \frac{9}{10\gamma}$.

\paragraph{When the eigenfunctions are uniformly bounded.} Let $(f_j)_{j=1}^\infty$ be the eigenfunctions of $K$ in $L^2(\mu)$. Suppose there exists an absolute constant $C$ such that $\sup_{j\in J_*^c}\|f_j(X)\|_{L^\infty(\mu)}\leq C$. Then almost surely,
\begin{align*}
    \sum_{j\in J_*^c}\sigma_j^2 f_j^2(X) \leq C^2\Tr(\Sigma_{J_*^c}^2).
\end{align*}

\paragraph{Inner product kernel over sphere.} Suppose $\cX = S_2^{d-1}$, $\mu$ is the uniform distribution over $\cX$, and there exists $h:\bR\to\bR$ such that $K:(\vx_1,\vx_2)\mapsto h(\langle\vx_1,\vx_2\rangle)$. By spherical harmonic analysis, the eigenfunctions $(f_j)_{j=1}^\infty$ are given by spherical harmonics, and $L^2(\mu) = \bigoplus_{k=0}^\infty V_k$, where $V_k$ is the eigen-subspace formed by degree $k$ spherical harmonics, whose dimension is denoted by $N(d,k)$. For any $k$, take $\{Y_{k,\ell}\}_{\ell=1}^{N(d,k)}$ be an orthonormal basis in $L^2(\mu)$ inner product of $V_k$. Moreover, all the eigenfunctions $\{Y_{k,\ell}\}_{\ell=1}^{N(d,k)}$ have the same eigenvalue, denoted by $\lambda_k$. That is, the distinct spectrum of $K$ is given by $\lambda_1,\cdots$, where for each $\lambda_k$, its multiplicity is $N(d,k)$. By addition theorem of spherical harmonics, \cite[Theorem 4.11]{frye_spherical_2012}, for any $\vx\in \cX$, $\sum_{\ell=1}^{N(d,k)}Y_{k,\ell}^2(\vx)=N(d,k)$. Therefore, if $k^* = \sum_{k=0}^LN(d,k)$ for some $L\in\bN_+$, then $\sum_{j\in J_*^c}\sigma_j^2 f_j^2(X) = \sum_{k=L+1}^\infty \lambda_k^2 \sum_{\ell=1}^{N(d,k)}Y_{k,\ell}^2(X) = \sum_{k>L}\lambda_k^2 N(d,k)=\Tr(\Sigma_{J_*^c}^2)$.

\paragraph{Regular RKHS.}
Let $T_K:f\in L^2(\mu)\mapsto \int_\cX K(\vx,\cdot)f(\vx)d\mu(\vx)$ be the integral operator induced by $K$. Let its distinct eigenvalues be $\mu_1\geq\mu_2\geq\cdots$, with multiplicity $N_1,N_2,\cdots$. Let $I_m = \{\mu_m,\cdots,\mu_m\}$ ($N_m$ times), and $K_m:(\vx_1,\vx_2)\mapsto \sum_{j\in I_m}f_j(\vx_1)f_j(\vx_2)$. Then $\bE[K_m(X,X)]=N_m$. 
\begin{Proposition}
Suppose $J_*^c = \bigcup_{m>L}I_m$ for some $L\in\bN_+$, and suppose there exists some absolute constant $C$, such that for any $M>L$ and any $\vx\in\cX$, $\sum_{m=L+1}^M K_m(\vx,\vx)\leq C\sum_{m=L+1}^M N_m$. Then we have $\sup_{\vx\in\cX}\frac{\sum_{j\in J_*^c}\sigma_j^2 f_j^2(X)}{\Tr(\Sigma_{J_*^c}^2)}\leq C$.
\end{Proposition}
\beginproof
For any $\vx\in\cX$, $\|\phi_{J_*^c}(\vx)\|_{L^2(\mu)}^2 = \sum_{m=L+1}^\infty \mu_m K_m(\vx,\vx)$. On the other hand, $\Tr(\Sigma_{J_*^c}^2) = \sum_{m=L+1}^\infty \mu_m N_m.$ 
Define the partial sums for the tail blocks as $S_m(\vx) = \sum_{k=L+1}^m K_k(\vx,\vx)$ and $D_m = \sum_{k=L+1}^m N_k$ for $m > L$, with the convention $S_L(\vx) = 0$ and $D_L = 0$. By assumption, $S_m(\vx) \leq C D_m$ holds for all $m > L$. For any finite $M > L$, applying Abel summation (summation by parts) to the truncated series yields
\begin{align*}
\sum_{m=L+1}^M \mu_m K_m(\vx,\vx) &= \sum_{m=L+1}^M \mu_m (S_m(\vx) - S_{m-1}(\vx)) = \sum_{m=L+1}^{M-1} (\mu_m - \mu_{m+1}) S_m(\vx) + \mu_M S_M(\vx).
\end{align*}
Since $\mu_1\geq\mu_2\geq\cdots$ are distinct eigenvalues, we strictly have $\mu_m - \mu_{m+1} > 0$ for all $m$. Coupled with the assumptions $S_m(\vx) \leq C D_m$, $S_M(\vx) \geq 0$, and $\mu_M > 0$, we can bound the summation by
\begin{align*}
&\sum_{m=L+1}^{M-1} (\mu_m - \mu_{m+1}) S_m(\vx) + \mu_M S_M(\vx) \leq C \left( \sum_{m=L+1}^{M-1} (\mu_m - \mu_{m+1}) D_m + \mu_M D_M \right) \\
&= C \sum_{m=L+1}^M \mu_m (D_m - D_{m-1}) = C \sum_{m=L+1}^M \mu_m N_m.
\end{align*}
Taking the limit as $M \to \infty$, we obtain
\begin{align*}
\sum_{j\in J_*^c}\sigma_j^2 f_j^2(X) = \sum_{m=L+1}^\infty \mu_m K_m(\vx,\vx) \leq C \sum_{m=L+1}^\infty \mu_m N_m = C \Tr(\Sigma_{J_*^c}^2).
\end{align*}
Since this inequality holds for all $\vx \in \cX$, taking the supremum over $\cX$ yields $\sup_{\vx\in\cX}\frac{\sum_{j\in J_*^c}\sigma_j^2 f_j^2(X)}{\Tr(\Sigma_{J_*^c}^2)}\leq C$, which completes the proof.
\endproof

\paragraph{Sobolev spaces.} An example of a regular RKHS is a Sobolev space satisfying a local Weyl estimate. Here, if for all $m\geq 1$, $\sum_{j=1}^m |f_j(\vx)|^2 \leq Cm$ almost surely, we say it satisfies the local Weyl's estimate, see \cite[Equation 29.1.3]{hormanderAnalysisLinearPartial2009} applied to $Q=I$ and $\mu=0$.
\begin{Proposition}
    Let $\cH$ be a RKHS on $\cX$ generated by some kernel function $K$, satisfying $\sigma_j\sim j^{-\frac{2s}{d}}$ with $s>\frac{d}{2}$. Let $(f_j)_{j=1}^\infty$ be the sequence of eigenfunctions of $T_K:f\in L^2(\mu)\mapsto \int_\cX K(\vx,\cdot)f(\vx)d\mu(\vx)$. Suppose there exists an absolute constant $C$ such that for all $m\geq 1$, $\sum_{j=1}^m |f_j(\vx)|^2 \leq Cm$ almost surely. Then there exists $\gamma>1$, such that for any $k\in\bN_+$ and $J^c = \{k+1,k+2,\cdots\}$, $\sum_{j\in J_*^c}\sigma_j^2 f_j^2(X)\leq\gamma \Tr(\Sigma_{J^c}^2)$ almost surely.
\end{Proposition}

\beginproof Notice that $\Tr(\Sigma_{J^c}^2)\sim k^{1-\frac{4s}{d}}$ for any $k$ such that $J=[k]$. On the other hand, by $\{j>k\} = \bigcup_{\ell\geq 0}\{ 2^\ell k< j \leq 2^{\ell+1}k \}$. Then
\begin{align*}
    &\sup_{\vx}\sum_{j>k}\sigma_j^2f_j^2(\vx)\leq \sum_{\ell\geq 0}\sup_\vx\sum_{2^\ell k<j\leq 2^{\ell+1}k}\sigma_j^2 f_j^2(\vx)\lesssim \sum_{\ell\geq 0}(2^\ell k)^{-\frac{4s}{d}}\sup_\vx\sum_{j\leq 2^{\ell+1}k}f_j^2(\vx)\\
    &\lesssim \sum_{\ell\geq 0}(2^\ell k)^{-\frac{4s}{d}}2^\ell k = k^{1-\frac{4s}{d}}\sum_{\ell\geq 0}2^{\ell(1-\frac{4s}{d})}.
\end{align*}Since $s>\frac{d}{2}$, $1-\frac{4s}{d}<-1$, the series $\sum_{\ell\geq 0}2^{\ell(1-\frac{4s}{d})}$ converges, and hence $\sup_{\vx}\sum_{j>k}\sigma_j^2f_j^2(\vx)\lesssim k^{1-\frac{4s}{d}}$.
\endproof

\section{Statistical analysis of PCR: proof of Theorem~\ref{theo:main_PCR}} 
\label{sec:statistical_analysis_of_pcr}
In this section we prove Theorem~\ref{theo:main_PCR}. We recall that the Principle Component Regression (PCR) estimator $\hat\bbeta = \frac{1}{N}\varphi_t(\hat\Sigma)\bX^\top\vy$ is obtained for the filter function and its associated residual function given for $t\geq1$ by
\begin{align*}
    \varphi_t:x>0\mapsto x^{-1}\1(x\geq bt^{-1}),\mbox{ and } \psi_t:x\in\bR\mapsto 1-x\varphi_t(x) = \1(x< bt^{-1})
\end{align*}where $0<b<1$ is the same parameter used in the definition of $k^*:= \min\left(k\in[p]: \sigma_{k+1}\leq bt^{-1}\right)$, the estimation dimension. In this section, we also denote  $J= J_*=[k^*]$. 

First note that for all $t\geq 1$ and $x>0$, we have
\begin{equation}\label{eq:assum_1_PCR}
    \varphi_t(x) = \frac{1}{x}\1(x\geq bt^{-1}) \leq \frac{C_1}{x+t^{-1}} \mbox{ for } C_1 = \frac{b+1}{b}
\end{equation}so that Assumption~\ref{ass:filter_fct} is satisfied by PCR's filter function for $c_1=0$ and $C_1 = (b+1)/b$.

A key observation in the analysis of PCR estimator is that, for a given SDP matrix $M$, $\psi_t(M)$ is the orthogonal projection on the eigenspace of $M$ spanned by all eigenvectors associated with eigenvalues less than $bt^{-1}$. In particular, $\psi_t(\Sigma) = P_{J^c}=\sum_{j\in J^c}\ve_j\otimes\ve_j$ and so for all $\bbeta\in V_J, \psi_t(\Sigma)\bbeta = 0$. We also observe that $x\varphi_t(x)=\1(x\geq bt^{-1})$ so that $\Sigma \varphi_t(\Sigma) = P_J$; in particular, for $\tilde\bbeta_J$ defined in Section~\ref{ssub:risk_decomposition_of_the_estimation_part_hat_bbeta_j} we have $\tilde\bbeta_J= \varphi_t(\Sigma)\Sigma\bbeta^* = P_J \bbeta^* = \bbeta^*_J$. As a consequence, the risk decomposition of the estimation part from Section~\ref{ssub:risk_decomposition_of_the_estimation_part_hat_bbeta_j} can be made simpler in the PCR case. 

Let us now start the risk analysis of the PCR estimator. As in \eqref{eq:risk_decomposition_origin_FSD}, we recall the risk decomposition that follows from the FSD method:
\begin{equation*}
 \norm{\Sigma^{1/2}\left(\hat\bbeta - \bbeta^*\right)}_2  
    \leq \norm{\Sigma_J^{1/2}\left(\hat\bbeta_J - \bbeta^*_J\right)}_2 + \norm{\Sigma_{J^c}^{1/2}\hat\bbeta_{J^c}}_2 +\norm{ \Sigma_{J^c}^{1/2}\bbeta_{J^c}^*}_2.
\end{equation*}Next, as mentioned above, the risk decomposition of the estimation part is simpler for the PCR estimator than in \eqref{eq:risk_decomp_esti_part} since we have
\begin{equation*}
    \norm{\Sigma_J^{1/2}(\hat\bbeta_J - \bbeta^*_J)}_2 \leq \norm{\Sigma_J^{1/2}(\hat\bbeta_J(\bX\bbeta^*_J) - \bbeta_J^*)}_2 +  \norm{\Sigma_J^{1/2}\hat\bbeta_J(\bX\bbeta_{J^c}^*+\bxi)}_2.
\end{equation*}
Now, we upper bound the two terms from this sum. For the first term, we have on $\Omega_t$
\begin{equation*}
  \norm{\Sigma_J^{1/2}(\hat\bbeta_J(\bX\bbeta^*_J) - \bbeta_J^*)}_2 = \norm{\Sigma_J^{1/2}\big(\psi_t(\hat \Sigma) - \psi_t(\Sigma) \big)\bbeta_J^*}_2 \lesssim \frac{\square}{\theta^2} \left\|\Sigma_J^{-1/2} \bbeta^*_J\right\|_{ 2 } 
\end{equation*}where the last inequality follows from an adaptation of the argument used in \eqref{eq:result_upper_bias_3} to the PCR case for the contour $\cC_t$ defined in \eqref{eq:contour_PCR}:  thanks to \eqref{eq:residual_theorem_PCR}, we indeed have
    \begin{align}\label{eq:result_upper_bias_3_PCR}
    \begin{aligned}
         &\left\|\Sigma_J^{\frac{1}{2}}\left(\psi_{t}(\Sigma)-\psi_{t}(\hat\Sigma)\right)\bbeta^*_J\right\|_2 = \frac{1}{2\pi} \norm{\oint_{\cC_t}\Sigma_J^{\frac{1}{2}}(\hat\Sigma-zI_p)^{-1}(\hat\Sigma-\Sigma)(\Sigma-zI_p)^{-1}\bbeta^*_J dz}_2 \\  
         &\leq \frac{1}{2\pi} \oint_{\mathcal{C}_t} \left\|\Sigma_t^{\frac{1}{2}}\left(\hat\Sigma-zI_p\right)^{-1} \Sigma_t^{\frac{1}{2}}\right\|_{op} \left\|\Sigma_t^{-\frac{1}{2}}\left(\hat\Sigma - \Sigma\right) \Sigma_t^{-\frac{1}{2}}\right\|_{op} 
        \left\|\Sigma_t^{\frac{1}{2}}(\Sigma-zI_p)^{-1} \Sigma_J^{\frac{1}{2}}\right\|_{op}\left\|\Sigma_J^{-1/2} \bbeta^*_J\right\|_{2}\left|d z\right| \\
        &\lesssim   \frac{\square}{\theta^2} \left\|\Sigma_J^{-1/2} \bbeta^*_J\right\|_{ 2 } \oint_{\mathcal{C}_t}\left|d z\right|\lesssim \frac{\square}{\theta^2} \left\|\Sigma_J^{-1/2} \bbeta^*_J\right\|_{ 2 }.
    \end{aligned}
    \end{align}For the second term, we use exactly the same arguments as in Section~\ref{ssub:upper_bound_on_the_variance_term_texorpdfstring_norm_sigma_j_1_2}: with probability at least $1-\exp(-c|J|)-\bP[\Omega_t^c]$,
\begin{equation*}
     \|\Sigma_J^{1/2}\hat\bbeta_J(\bX\bbeta_{J^c}^*+\bxi)\|_2 \lesssim \|\Sigma_{J^c}^{1/2}\bbeta_{J^c}^*\|_2 + \sigma_\xi \sqrt{\frac{|J|}{N}}.
\end{equation*}
As a consequence, we conclude that for the estimation part, we have with probability at least $1-\exp(-c|J|)-\bP[\Omega_t^c]$,
\begin{equation*}
 \norm{\Sigma_J^{1/2}\left(\hat\bbeta_J - \bbeta_J^*\right)}_2  
    \lesssim \frac{\square}{\theta^2} \left\|\Sigma_J^{-1/2} \bbeta^*_J\right\|_{ 2 } +  \|\Sigma_{J^c}^{1/2}\bbeta_{J^c}^*\|_2 + \sigma_\xi \sqrt{\frac{|J|}{N}}.
\end{equation*}     

Now, we prove a high probability  upper bound on the 'noise absorption' part of the PCR estimator, i.e. on the quantity $\norm{\Sigma_{J^c}^{1/2}\hat\bbeta_{J^c}}_2$. We follow the same analysis as in Section~\ref{sec:noise_absorption} but for the contour $\cC_t$ specially designed for the PCR estimator, i.e. the one from \eqref{eq:contour_PCR} and where we use Lemma~\ref{lemma:spectral_calculus_li_PCR} instead of Lemma~\ref{lemma:spectral_calculus_li}: with probability at least $1-2\exp(-c|J|) - \bP[\Omega_t^c]$,
\begin{equation*}
    \norm{\Sigma_{J^c}^{1/2}\hat\bbeta_{J^c}}_2 \lesssim \frac{\square}{\theta^2}\norm{\Sigma_J^{-\frac{1}{2}}\bbeta^*_J}_2 + \norm{\Sigma_{J^c}^{1/2}\bbeta_{J^c}^*}_2 +  \sigma_\xi \sqrt{\frac{|J|}{N}} + \sigma_\xi t \sqrt{\frac{\Tr(\Sigma_{J^c}^2)}{N}}.
\end{equation*}

Gathering both controls on the estimation part and the noise absorption part in the risk decomposition of the PCR estimator that follows from the FSD method, we obtain that with probability at least $1-c\exp(-|J|/c) - c\exp(-\square^2 N/c)$,
\begin{equation*}
    \norm{\Sigma^{1/2}\big(\hat\bbeta - \bbeta^*\big)}_2 \lesssim   \norm{\Sigma_{J^c}^{1/2}\bbeta_{J^c}^*}_2 +  \sigma_\xi \sqrt{\frac{|J|}{N}} + \sigma_\xi t \sqrt{\frac{\Tr(\Sigma_{J^c}^2)}{N}} + \frac{\square}{\theta^2}\norm{\Sigma_J^{-\frac{1}{2}}\bbeta^*_J}_2 \lesssim r(V_J, V_{J^c}) + \frac{\square}{\theta^2}\norm{\Sigma_J^{-\frac{1}{2}}\bbeta^*_J}_2.
\end{equation*}


\subsection{Construction and properties of the contour for the analysis of PCR} 
\label{sub:construction_of_the_contour_for_pcr}
Let $\cC_t\subset\bC$ be a contour such that: 
\begin{itemize}
    \item[(i)] $\cC_t$ surrounds  the set of all singular values of $\Sigma$ and $\hat \Sigma$ below $bt^{-1}$, i.e. the set $\big[\sigma(\Sigma)\cup \sigma(\hat \Sigma)\big]\cap[0, bt^{-1}]$,
    \item[(ii)] all singular values of $\Sigma$ and $\hat \Sigma$ above $bt^{-1}$, i.e. the set $\big[\sigma(\Sigma)\cup \sigma(\hat \Sigma)\big]\cap[bt^{-1}, +\infty]$ are 'outside' $\cC_t$. 
\end{itemize}
For a contour $\cC_t$ satisfying the two points above, it follows from \cite[pp. 39]{kato_perturbation_1995}, see also \cite[pp. 1984]{koltchinskii_asymptotics_2016} that 
\begin{align}\label{eq:residual_theorem_PCR}
 \nonumber\psi_t(\Sigma) - \psi_t(\hat\Sigma) & =P_{J^c} -  \hat P =\frac{1}{2\pi i}\oint_{\cC_t}\left[(\hat\Sigma-zI)^{-1} -(\Sigma-zI)^{-1}\right]dz \\
 &=-\frac{1}{2\pi i}\oint_{\cC_t}(\hat\Sigma-zI)^{-1}(\hat\Sigma-\Sigma)(\Sigma-zI)^{-1}dz  
\end{align}where $\hat P$ is the orthogonal projection onto the space spanned by all singular vectors of $\hat \Sigma$ associated with a singular value less than $bt^{-1}$. In particular, we recover a formulae similar to \eqref{eq:residual_theorem} but for $\psi_t$.

Now we define a contour that to satisfies the two requirements above. This contour is a counterclockwise rectangle $\cC_t = \cC_{t,1}\sqcup \cC_{t,2}\sqcup \cC_{t,3}\sqcup \cC_{t,4}$ made of the four segments:  
\begin{equation}\label{eq:contour_PCR}
   \begin{aligned}
    \cC_{t,1} &= \left\{ -1+ iy: -1\leq y\leq 1 \right\},\, \cC_{t,2} = \left\{ bt^{-1} + iy: -1\leq y\leq 1 \right\},\\
    \cC_{t,3} &= \left\{ x + i: -1\leq x\leq bt^{-1}\right\}\mbox{ and } \cC_{t,4} = \left\{ x - i: -1\leq x\leq bt^{-1} \right\}.
\end{aligned} 
\end{equation}It is clear from the definition of $\cC_t$ that the two conditions \textit{(i)} and \textit{(ii)} are satisfied by this contour.
Let us now turn to properties of $\cC_t$ that will be useful for the statistical analysis of PCR, i.e. to results similar to the one from Lemma~\ref{lemma:spectral_calculus_li}. We first recall that the $k^*$-th spectral gap of $\Sigma$ is the quantity $\gamma_{k^*} = \sigma_{k^*} - \sigma_{k^*+1}$. The following result requires $\gamma_{k^*}$ to be large enough so that $\theta>0$ where we recall that
\begin{equation*}
    \theta := \min\left(bt^{-1} - \big(\sigma_{k^*+1} + \square(\sigma_{k^*+1} + t^{-1})\big), \big(\sigma_{k^*} - \square(\sigma_{k^*} + t^{-1})\big) - bt^{-1}\right) 
\end{equation*}

\begin{Lemma}\label{lemma:spectral_calculus_li_PCR} Let $t\geq1$, $0<\square<1/9$ and $0<b<1$ be such that $\theta>0$. Let $\cC_t$ be the contour defined in \eqref{eq:contour_PCR}.  For all $z\in\cC_t$, we have
\begin{align*}
    \left\|\Sigma_t^{\frac{1}{2}}(\Sigma-zI_p)^{-1} \Sigma_t^{\frac{1}{2}}\right\|_{\op} \leq \frac{2}{\theta} \mbox{ and } \oint_{\mathcal{C}_{t}}\left|dz\right|  \leq 6.
\end{align*}
Moreover, on $\Omega_t$ we have for all $z\in\cC_t$, $\left\|\Sigma_t^{1/2}\left(\hat\Sigma-zI_p\right)^{-1} \Sigma_t^{1/2}\right\|_{\op} \leq  2/\theta$.
\end{Lemma}
\beginproof Let $z\in\cC_t$. We have 
\begin{equation*}
    \left\|\Sigma_t^{\frac{1}{2}}(\Sigma-zI_p)^{-1} \Sigma_t^{\frac{1}{2}}\right\|_{\op}= \max\left(\left|\frac{\sigma_j + t^{-1}}{\sigma_j-z}\right|:\, j\in J\right)\leq \max\left(\left|\frac{\sigma_j + t^{-1}}{\sigma_j-bt^{-1}}\right|:\, j\in J\right)\leq \frac{2}{\theta}.
\end{equation*}
Given that $bt^{-1}\leq1$, the length of $\cC_t$ is at most $6$ and so $\oint_{\mathcal{C}_{t}}\left|d z\right|  \leq 6.$ Next, we have 
   \begin{align*}
        \left\|\Sigma_t^{\frac{1}{2}}\left(\hat\Sigma-zI_p\right)^{-1} \Sigma_t^{\frac{1}{2}}\right\|_{\op} \leq  \frac{\sigma_1 + t^{-1}}{\min_{j}|\hat \sigma_j-z|}\leq \frac{2}{\min_{j}\left|\hat \sigma_j-bt^{-1}\right|}.
    \end{align*}

On the event $\Omega_t$, it follows from \eqref{eq:change_of_norm} that for all $\vu\in\bR^p$,
\begin{equation} 
    (1-\square)\norm{\Sigma^{1/2}\vu}_2^2 - \square t^{-1}\norm{\vu}_2^2 \leq \norm{\hat\Sigma^{1/2} \vu}_2^2 \leq (1+\square)\norm{\Sigma^{1/2}\vu}_2^2  +  \square t^{-1}\norm{\vu}_2^2.
\end{equation}As a consequence, for all $\vu\in V_{J_*}$, we have
\begin{equation}\label{eq:lower_sing_var_hat_sigma_J_PCR}
    \norm{\hat\Sigma \vu}_2  \geq \left[(1-\square)\sigma_{k^*} - \square t^{-1}\right]\norm{\vu}_2
\end{equation}and for all $\vu\in V_{J_*^c}$,
\begin{equation}\label{eq:upper_sing_var_hat_sigma_Jc_PCR}
    \norm{\hat\Sigma \vu}_2  \leq \left[(1+\square)\sigma_{k^*+1} + \square t^{-1}\right]\norm{\vu}_2.
\end{equation}Given that $V_{J_*}$ is of dimension $k^*$ (and so $V_{J_*^c}$ is of dimension $p-k^*$), it follows from \eqref{eq:lower_sing_var_hat_sigma_J_PCR}, \eqref{eq:upper_sing_var_hat_sigma_Jc_PCR} and the Courant-Fischer minimax variational formulas (see for instance Theorem~4.2.1 in \cite{ChafaiGuedonLecuePajor2012}) that 
\begin{equation*}
    \hat \sigma_{k^*} = \max_{V: {\rm dim}(V)=k^*}\min_{\vu\in V: \norm{\vu}_2=1} \norm{\hat\Sigma \vu}_2\geq \min_{\vu\in V_{J_*}: \norm{\vu}_2=1} \norm{\hat\Sigma \vu}_2\geq  \sigma_{k^*} -  \square\left[\sigma_{k^*} + t^{-1}\right].
\end{equation*} and 
\begin{equation*}
    \hat \sigma_{k^*+1} = \min_{V: {\rm dim}(V)=p-k^*}\max_{\vu\in V: \norm{\vu}_2=1} \norm{\hat\Sigma \vu}_2\leq \max_{\vu\in V_{J_*^c}: \norm{\vu}_2=1} \norm{\hat\Sigma \vu}_2\leq  \sigma_{k^*+1} +  \square\left[\sigma_{k^*+1} + t^{-1}\right].
\end{equation*}As a consequence, on $\Omega_t$, we obtain that
\begin{equation*}
    \min_{j}\left|\hat \sigma_j-bt^{-1}\right|\geq \theta
\end{equation*}and so the result follows.
\endproof

\subsection{Proof of Proposition~\ref{prop:single_index_problem}}\label{sec:proof_single_index_problem}

It follows that $f^* \in \mathcal{H}$, which implies $g_{\mathcal{H}}^* = f^*$. We denote $\mu_\iota(\sigma) = (\mathbb{E}_t[\sigma(t) Q_\iota^{(d)}(t)])^2$, where $t = \langle X, \vw^* \rangle$ follows a probability distribution on $[-1, 1]$, and $Q_\iota^{(d)}$ is the $d$-dimensional Gegenbauer polynomial of order $\iota$. Here, $\mu_\iota(\sigma)$ is referred to as the $\iota$-th Gegenbauer coefficient of $\sigma$. For any $\iota \in \mathbb{N}_+$, let $N(d, \iota) = \binom{d+\iota-1}{\iota} - \binom{d+\iota-3}{\iota-2}$. For a fixed $\iota$, $N(d, \iota) = \Theta(d^\iota / \iota!)$ as $d \to \infty$. For any $\ell \leq L$, denote $D_\ell = \sum_{\iota=0}^\ell N(d, \iota)$ with the convention $D_{-1} = 0$. By spherical harmonic analysis (see, for instance, \cite{mei_generalization_2022} and the references therein), the spectrum of the covariance operator of $\mathcal{H}$ satisfies the following property: for any $j \leq D_L$, $\sigma_j = \lambda_\ell$ whenever $D_{\ell-1} < j \leq D_\ell$, where $\lambda_\ell = \mu_\ell(\sigma) / N(d, \ell)$; for $j > D_L$, $\sigma_j = 0$. One can verify that $\mu_\iota(\sigma) = \sum_{j: \sigma_j = \lambda_\iota} \langle f^*, \ve_j \rangle^2$. The eigenfunctions of the covariance operator of $\mathcal{H}$ are given by spherical harmonics.

Therefore, for any $k \leq D_L$, let $\ell$ satisfy $D_{\ell-1} < k \leq D_\ell$. It follows that $\|P_{k+1:\infty}g_{\mathcal{H}}^*\|_{L^2(\mu)}^2 = \sum_{j=k+1}^{D_\ell}\langle f^*, \ve_j \rangle^2 + \sum_{\iota=\ell+1}^\infty (\sum_{j=D_{\iota-1}+1}^{D_\iota} \langle f^*, \ve_j \rangle^2) = \sum_{j=k+1}^{D_\ell} \langle f^*, \ve_j \rangle^2 + \sum_{\iota=\ell+1}^L \mu_\iota(\sigma)$. It is easy to see that $\|P_{D_L+1:\infty}f^*\|_{L^2(\mu)}=0$, hence $k^\circ(N)\leq D_L$. Let $\alpha = \frac{|\langle f^*, \ve_{D_L} \rangle|^2}{2\sigma_\xi^2}$. We consider the following two cases:
\begin{enumerate}
    \item If $N \leq D_L$, there exists $\ell \leq L$ such that $D_{\ell-1} < N \leq D_\ell$. From the aforementioned estimate of $\|P_{k+1:\infty}g_{\mathcal{H}}^*\|_{L^2(\mu)}^2$, we have $\frac{k^\circ(N)}{N} \geq \sigma_\xi^{-2} (\sum_{j=k+1}^{D_\ell} \langle f^*, \ve_j \rangle^2 + \sum_{\iota=\ell+1}^L \mu_\iota(\sigma)) > \alpha$. In this case, alignment is deficient. Moreover, notice that $2\alpha\sigma_\xi^2\leq \|P_{k^\circ+1:\infty}f^*\|_{L^2(\mu)}^2$. Therefore, $2\alpha\sigma_\xi^2\leq\|P_{k^\circ+1:\infty}f^*\|_{L^2(\mu)}^2\leq \sigma_\xi^2\frac{k^\circ(N)}{N}$, and hence $k^\circ(N)\geq 2\alpha N$. Combining with $k^\circ(N)\leq D_L$, there holds $k^\circ(N)\geq\min(\lceil 2\alpha N\rceil, D_L)$.

    \item If $N > D_L$, for $k = D_L$, we have $\|P_{k+1:\infty}g_{\mathcal{H}}^*\|_{L^2(\mu)}^2 = 0$. Consequently, $\frac{k^\circ(N)}{N} \leq \frac{D_L}{N}$. When $N > (1/\alpha) D_L$, alignment is efficient. Suppose there exists $k\leq D_L-1$, such that $\|P_{k+1:\infty}f^*\|_{L^2(\mu)}^2\leq \sigma_\xi^2\frac{k}{N}$, then necessarily $2\alpha\sigma_\xi^2\leq \sigma_\xi^2\frac{k}{N}$, that is, $k\geq 2\alpha N$. However, since $\alpha N>D_L$, we know $k\geq 2\alpha N>2D_L>D_L$, leading to a contradiction. Therefore, in this case, $k^\circ(N)=D_L$.
\end{enumerate}

\subsection{Proof of Proposition~\ref{prop:sobolev_regression}}\label{sec:proof_sobolev_regression}

Let $f^* = \sum_{j=1}^\infty c_j f_j$. Then from $\|T_K^{-s/2}f^*\|_{L^2(\mu)}<\infty$, there exists some absolute constant $0<M<\infty$ such that $\sum_{j=1}^\infty c_j^2\sigma_j^{-s}<M$. Now we have
\begin{align*}
\|P_{k+1:\infty}f^*\|_{L^2(\mu)}^2 = \sum_{j>k}c_j^2 \leq \sup_{j>k}\sigma_j^s\sum_{j>k}c_j^2\sigma_j^{-s} \leq M\sigma_{k+1}^s \lesssim M k^{-s\alpha}.
\end{align*}
Solving $Mk^{-s\alpha}\leq \sigma_\xi^2\frac{k}{N}$ gives $k^\circ(N)\lesssim N^{\frac{1}{1+s\alpha}}$.




\bibliographystyle{alpha}
\bibliography{biblio,other}
\end{document}